\documentclass[11pt]{article}
\usepackage[utf8]{inputenc}
\usepackage{amscd,amsmath,amsthm,amssymb,amsfonts}
\usepackage{txfonts}
\usepackage{eucal}
\usepackage{bbm}
\usepackage{appendix}
\usepackage{authblk}
\usepackage{graphicx} 
\usepackage{hyperref}
\hypersetup{
    colorlinks=true,
    linkcolor=blue,
    filecolor=blue,      
    urlcolor=blue,
    citecolor=blue,
    }
\usepackage{algorithm}
\usepackage{mathrsfs}
\usepackage{derivative}
\usepackage[noend]{algorithmic}
\usepackage[a4paper]{geometry}
\usepackage{tikz-cd}
\usepackage{mathtools}
 \usepackage[nottoc]{tocbibind}    
\usepackage{empheq}
\usepackage{graphics}
\usepackage{indentfirst}
\usepackage[
backend=bibtex,
style=numeric-comp,
maxnames=4
]{biblatex}
\usepackage{color}
\usepackage{tabularx,colortbl}
\numberwithin{equation}{section}

\definecolor{verylight}{gray}{0.97}
\definecolor{light}{gray}{0.9}
\definecolor{medium}{gray}{0.85}

\usepackage{fancyhdr}
\usepackage{subcaption}
\DeclareCaptionLabelFormat{custom}
{%
      \textbf{Fig. #2}
}
\DeclareCaptionLabelSeparator{custom}{ }
\DeclareCaptionFormat{custom}
{%
    #1 #2 #3
}
\newcommand{\helv}{%
\fontfamily{phv}\fontseries{b}\fontsize{9}{11}\selectfont}
\usepackage{etoolbox}

\makeindex

\def\s{$\mbox{\c{s}}$}

\def\opn#1#2{\def#1{\operatorname{#2}}} % to make operators

\opn\Ker{Ker} \opn\Coker{Coker}  \opn\Hom{Hom} \opn\Im{Im}
\opn\End{End} \opn\Aut{Aut} \opn\defect{def} \opn\ord{ord}
\opn\id{id} \opn\dim{dim} \opn\det{det} \opn\tr{tr} \opn\grad{grad} \opn\lcm{lcm}
\opn\min{min} \opn\max{max} 
\opn\Span{Span}   \opn\rang{rang}  \opn\id{id} \opn\Ass{Ass} \opn\Min{Min}
\opn\GL{GL} \opn\SL{SL} \opn\mod{mod} \opn\diag{diag}
\opn\min{min} \opn\sgn{sgn} \opn\ini{in_<}  \opn\Mon{Mon} \opn\LC{LC_<} \opn\Hom{Hom} \opn\Ext{Ext} \opn\gini{gin_{<_{rev}}} \opn\gin{gin_{<}}
\opn\LT{LT_<}
\opn\s{supp} \opn\Tor{Tor} \opn\link{link} \opn\depth{depth} \opn\pd{pd} \opn\reg{reg} %\opn\char{char}
\newcommand{\compactfundif}{\textup{C}^{1}_{\textup{c}}(\mathbb{R}^{2}_{>0})}
\newcommand{\vanishfundif}{\textup{C}_{\textup{0}}^{1}(\mathbb{R}^{2}_{>0})}
\newcommand{\compactfun}{\textup{C}_{\textup{c}}(\mathbb{R}^{2}_{>0})}
\newcommand{\vanishfun}{\textup{C}_{\textup{0}}(\mathbb{R}^{2}_{>0})}

\newcommand{\alltimesolradon}{\textup{C}([0,\infty);\mathscr{M}_{+}(\mathbb{R}_{>0}^{2}))}

\newcommand{\fintimesolradondif}{\textup{C}^{1}([0,T];\mathscr{M}_{+}(\mathbb{R}_{>0}^{2}))}

\newcommand{\supp}{\textup{supp}}
\newcommand{\der}{\textup{d}}

\date{}
\title{Coagulation equations for non-spherical clusters}
\author[1]{Iulia Cristian}
\author[2]{Juan J. L. Vel\'{a}zquez}
\affil[1]{Institute for Applied Mathematics, University of Bonn, Endenicher Allee 60, 53115 Bonn, Germany
\href{mailto:cristian@iam.uni-bonn.de}{cristian@iam.uni-bonn.de}}
\affil[2]{Institute for Applied Mathematics, University of Bonn, Endenicher Allee 60, 53115 Bonn, Germany
\href{mailto:velazquez@iam.uni-bonn.de}{velazquez@iam.uni-bonn.de}}

\begin{document}
\maketitle
%%%%%%%%%%%%%%%%%%%%%%%%%%%%%%%%%%%%%%%%%%%%%%%%%%%%%%%%%%%%%%%%%%%%
%%%%%%%%%%%%%%%%%%%%%%%%%%%%%%%%%%%%%%%%%%%%%%%%%%%%%%%%%%%%%%%%%%%%
\newtheorem{teo}{Theorem}[section]
\newtheorem{ex}[teo]{Example}
\newtheorem{prop}[teo]{Proposition}
\newtheorem{obss}[teo]{Observations}
\newtheorem{cor}[teo]{Corollary}
\newtheorem{lem}[teo]{Lemma}
\newtheorem{prob}[teo]{Problem}
\newtheorem{conj}[teo]{Conjecture}
\newtheorem{exs}[teo]{Examples}
\newtheorem{alg}[teo]{\bf Algorithm}

\theoremstyle{definition}
\newtheorem{defi}[teo]{Definition}

\theoremstyle{remark}
\newtheorem{rmk}[teo]{Remark}
\newtheorem{ass}[teo]{Assumption}
%%%%%%%%%%%%%%%%%%%%%%%%%%%%%%%%%%%%%%%%%%%%%%%%%%%%%%%%%%%%%%%%%%%%
\setlength{\baselineskip}{6mm}

\def\tombstone{$\;$\rule[-0.5mm]{2mm}{3.5mm} }

%%%%%%%% definirea functiei

\newcommand {\func}[3] {$#1:#2\longrightarrow #3$}

%%%%%% printarea pe A4, stabilirea marginilor

\DeclareGraphicsRule{.tif}{bmp}{pcx}{}
%%%%%%%%%%%%%%

\pagenumbering{arabic}

\begin{abstract}
   In this work, we study the long time asymptotics of a coagulation model which describes the evolution of a system of particles characterized by their volume and surface area. The aggregation mechanism takes place in two stages: collision and fusion of particles. During the collision stage, the two particles  merge at a contact point. The newly formed particle has volume and area equal to the sum of the respective quantities of the two colliding particles. After collision, the fusion phase begins and during it the geometry of the interacting particles is modified in such a way that the volume of the total system is preserved and the surface area is reduced. During their evolution, the particles must satisfy the isoperimetric inequality. Therefore, the distribution of particles in the volume and area space is supported in the region where $\{a\geq (36\pi)^{\frac{1}{3}}v^{\frac{2}{3}}\}$. We assume the coagulation kernel has a weak dependence on the area variable. We prove existence of self-similar profiles for some choices of the functions describing the fusion rate for which the particles have a shape that is close to spherical. On the other hand, for other fusion mechanisms and suitable choices of initial data, we show that the particle distribution describes a system of ramified-like particles.
      
\end{abstract}

\tableofcontents

\section{Introduction}

Most of the works on coagulation equations assume that the particles are characterized by a single variable, usually the particle volume (or equivalent quantities like polymer length), see for instance \cite{norris1999,1916ZPhy...17..557S,STEWART,gelation}. Nevertheless, other parameters that might provide insight about the geometry or other features of the particle are usually omitted. In this paper, we study the mathematical properties of a class of coagulation equations in which the aggregating particles are characterized by two degrees of freedom, namely the volume $v$ and the surface area $a$. This type of models was introduced in \cite{KOCH1990419} (see also \cite{book1} for a detailed discussion about its properties). More precisely, the model that we consider in this paper is the following:
\begin{align}
\partial_{t}f(a,v,t)+\partial_{a}[r(a,v)(c_{0}v^{\frac{2}{3}}-a)f(a,v,t)]=\mathbb{K}[f](a,v,t), \textup{  }  c_{0}:=(36\pi)^{\frac{1}{3}},
\label{strongfusioneq}
\end{align}
where 
\begin{align*}
    \mathbb{K}[f](a,v,t):=&\frac{1}{2}\int_{(0,a)\times(0,v)}K(a-a',v-v',a',v')f(a',v',t)f(a-a',v-v',t)\der v'\der a' \\
    &-\int_{(0,\infty)^{2}}K(a,v,a',v')f(a,v,t)f(a',v',t)\der v' \der a'.
\end{align*}

In this model, $f$ is the density of the particles in the space of area and volume for any given time $t\geq 0$. The coagulation operator $\mathbb{K}[f]$ is the classical coagulation operator that was introduced by Smoluchowski (see \cite{1916ZPhy...17..557S}). However, a difference is the fact that this operator now describes the evolution of particles characterized by both volume and surface area. Notice that the coagulation operator gives the coagulation rate of particles which evolve according to the following mechanism:
\begin{align*}
(a_{1},v_{1})+(a_{2},v_{2})\longrightarrow (a_{1}+a_{2},v_{1}+v_{2}).
\end{align*}
It is assumed that the particles attach to each other at their contact point and therefore in this way both the total area and volume of the particles involved in the process are preserved (see Figure \ref{fig1}).

The main difference between (\ref{strongfusioneq}) and the standard one-dimensional coagulation model is the presence of the term $\partial_{a}[r(a,v)(c_{0}v^{\frac{2}{3}}-a)f(a,v,t)]$. We call $\partial_{a}[r(a,v)(c_{0}v^{\frac{2}{3}}-a)f(a,v,t)]$ the "fusion term". This describes an evolution of the particles towards a spherical shape (see Figure \ref{fig1}). The dynamics generated by this term preserves the total number and volume of the particles. The term $c_{0}v^{\frac{2}{3}}-a$ indicates that the area of the particles tends to be reduced as long as it is larger than that of a sphere $c_{0}v^{\frac{2}{3}}$. In particular, spherical particles remain spherical and they do not evolve at all due to the fusion term. This fusion mechanism holds, for example, for the merging of droplets consisting of highly viscous fluids (see \cite{KOCH1990419}). 
\\

\begin{figure}[h]
\centering
\includegraphics[width=0.8\textwidth]{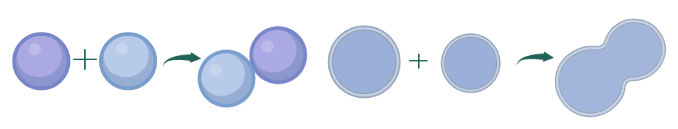}
\caption{Coagulation mechanism (left). Fusion mechanism (right)}\label{fig1}
\end{figure}

Additionally, $r(a,v)$ will indicate the fusion rate and describes how quickly the particles evolve towards the spherical shape and thus has units of the inverse of the fusion time. In the particular case when $r\equiv 0$, fusion does not occur and particles attach at contact points forming a ramified-like system in time. On the contrary, if the fusion rate is much faster than the coagulation rate, the particles tend to become spherical immediately after colliding (see Figure \ref{fig2}). A distinction between these two cases is not possible in the standard one-dimensional model. 
\\

\begin{figure}[h]%
\centering
\includegraphics[width=0.8\textwidth]{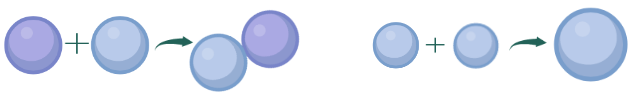}
\caption{Absence of fusion (left). Instantaneous fusion (right)}\label{fig2}
\end{figure}

As stated in \cite{KOCH1990419}, in aerosols, changes of temperature or adding impurities to the system can lead to different fusion rates, showing that the non-spherical shape of the particles plays a significant role. The main goal of this paper is to see how much of the mathematical theory for the one-dimensional coagulation equation can be carried on to the two-dimensional case and to observe the new mathematical phenomena that this model leads us to.

We remark that the particles must satisfy the isoperimetric inequality, therefore the density $f$ should be supported in the region where $\{a\geq c_{0}v^{\frac{2}{3}}\}$. Moreover, the evolution generated by (\ref{strongfusioneq}) has the property that it preserves the set of measures supported in this region. For simplicity, we define the set
\begin{align}\label{set iso}
    \tilde{\textup{S}}:=\{(a,v)\in(0,\infty)^{2},a\geq c_{0}v^{\frac{2}{3}}\}.
\end{align}

To obtain a better understanding of how fusion affects interactions between particles, we can check that it gives a decrease in the total surface area. We multiply by $a$ in (\ref{strongfusioneq}) and integrate formally, obtaining
\begin{align}
\partial_{t}\int_{(0,\infty)^{2}}af(a,v,t)\der a \der v=\int_{(0,\infty)^{2}}r(a,v)[c_{0}v^{\frac{2}{3}}-a]f(a,v,t)\der a \der v\leq 0, \label{decreasing_area_example}
\end{align}
since $f$ is supported in the region where the isoperimetric inequality is satisfied. 

We assume $r(a,v)$ behaves like a power law of $a$ and $v$. This covers the case of coalescence of viscous liquid spheres (see \cite{KOCH1990419}), where the fusion time depends on the diameter. For the coagulation kernel $K$, we assume that it has a weak dependence on the surface area of the interacting particles, but it can have a power law behaviour in the volume of the coalescing particles.

It is well-known that the solutions of coagulation equations behave often as self-similar solutions. Using the fact that the solutions of (\ref{strongfusioneq}) preserve the total volume of the particles, it is natural to look for solutions of (\ref{strongfusioneq}) of the form:
\begin{align}
f(a,v,t)=\frac{1}{(1+t)^{\frac{8}{3}\xi}}g\bigg(\frac{a}{(1+t)^{\frac{2}{3}\xi}},\frac{v}{(1+t)^{\xi}}\bigg) \textup{ for } \xi=\frac{1}{1-\gamma}. \label{scale}
\end{align}
Notice that the total surface area of solutions of the form (\ref{scale}) is not preserved, as it can be expected due to the presence of the fusion term in (\ref{strongfusioneq}).

\subsubsection*{Assumptions on the coagulation kernel}

The reason for the self-similar behavior in the case of the one-dimensional coagulation equation is that the coagulation rate scales like a power law of the particle size. In the case of particles characterized by volume and area, if the particle volume is scaled by a factor $\lambda$ (without modifying the geometry), then the diameter is scaled with a factor $\lambda^{\frac{1}{3}}$ and the area scales like $\lambda^{\frac{2}{3}}$. This suggests the following assumptions for the coagulation kernel:
\begin{align}
K(\lambda^{\frac{2}{3}}a,\lambda v,\lambda^{\frac{2}{3}}a',\lambda v') & =\lambda^{\gamma}K(a,v,a',v'), \label{homogeneity}
\end{align}
for all $(a,v,a',v')\in(0,\infty)^{4}$, $\lambda\in(0,\infty)$ and $\gamma\in [0,1).$ We assume $\gamma<1$ in order to avoid gelation and to obtain volume-conserving solutions. $\gamma>0$ means physically that the coagulation rate increases with the particle size.

Since collision does not change if we permute the colliding particles, i.e. $(a,v)\leftrightarrow(a',v')$, the coagulation kernel must satisfy the following symmetry property:
\begin{align}
    K(a,v,a',v') & =K(a',v',a,v),\label{kersym1}
\end{align}
for all $(a,v,a',v')\in(0,\infty)^{4}$.

We work with non-negative continuous kernels on $(0,\infty)^{4}$ that, in addition to the properties already stated, i.e. (\ref{homogeneity}) and (\ref{kersym1}), have the following bounds:
\begin{align}
K_{1}(v^{-\alpha}v'^{\beta}+v'^{-\alpha}v^{\beta})\leq K(a,v,a',v')\leq K_{0}(v^{-\alpha}v'^{\beta}+v'^{-\alpha}v^{\beta}), \label{lower_bound_kernel}
\end{align}
for some $K_{1},K_{0}>0$, for all $a,v,a',v'$ and for the following coefficients:
\begin{align}
  &\alpha>0 \textup{ and } \gamma=\beta-\alpha\in[0,1) \textup{ and } \beta\in (0,1). \label{alpha non neg}
  \end{align}

Notice that condition (\ref{lower_bound_kernel}) implies that the kernel has a weak dependence on the area variable, but $K$ is not necessarily independent of the area variable.

Most of the results of the paper are obtained for coagulation kernels $K$ with bounds (\ref{lower_bound_kernel}) with $\alpha>0$. In that case, since the coagulation rate is very large for small particles, we can expect $g$ (defined in (\ref{scale})) to be bounded (and small) for small values of $v$. In particular, for the self-similar profiles $g$ obtained when $\alpha>0$, we have $M_{0,d}(g):=\int_{(0,\infty)^{2}}v^{d}g(a,v) \der a\der v<\infty,$ for all $d\in\mathbb{R}$. This is analogous to what happens in the one-dimensional case for the standard coagulation model, where it is known that there exists self-similar profiles for which all the moments with negative powers of $v$ are bounded if $\alpha>0$. For details, see, for example \cite[Chapter  10.2.4, Theorem 10.2.17]{book} or \cite{articlefournier}.
 On the contrary, for the one-dimensional coagulation equation, for coagulation kernels satisfying (\ref{lower_bound_kernel}) with $\alpha=0$, the self-similar profiles can be singular for small values of $v$ and we can expect to have boundedness only for the moments containing powers of the form $v^{d}$, with $d\geq \gamma$, cf. the previously stated result in \cite{book}. An analogous situation takes place for the two-dimensional coagulation model considered in this paper. To illustrate this situation, we show some results concerning self-similar profiles of (\ref{strongfusioneq}) for coagulation kernels $K$ satisfying (\ref{lower_bound_kernel}) with $\alpha=0$. Specifically, we restrict ourselves to the case when
 \begin{align}
  \alpha=0 \textup{ and } \gamma=\beta\in(0;\frac{2}{3}). \label{alpha zero}
\end{align}
The reason to restrict ourselves to the case when $\gamma<\frac{2}{3}$ is because for this range of values it will be easier to obtain estimates for $M_{1,0}(g):=\int_{(0,\infty)^{2}}ag(a,v)\der a \der v$. Due to the isoperimetric inequality, this estimate implies an estimate for the moment $M_{0,\frac{2}{3}}$. Since we expect to have estimates only for moments $M_{0,d}$, with $d\geq \gamma$, it is natural to assume $\gamma<\frac{2}{3}$.

\subsubsection*{Assumptions on the fusion kernel}

Concerning the fusion kernel $r$, we assume that $r\in\textup{C}^{1}(\mathbb{R}_{>0}^{2})$ and that there exist constants $R_{0},R_{1}>0$ such that:
\begin{align}
   R_{0}a^{\mu}v^{\sigma}\leq r(a,v)\leq R_{1}a^{\mu}v^{\sigma}, \label{fusion_form}
\end{align}
for all $(a,v)\in(0,\infty)^{2}$ and some coefficients $\mu,\sigma\in\mathbb{R}$. 

In order to keep the self-similar structure, in other words, to have solutions of (\ref{strongfusioneq}) with the particular form (\ref{scale}), we require in addition:
\begin{align}
    r(\lambda^{\frac{2}{3}}a,\lambda v)=\lambda^{\gamma-1}r(a,v), \label{fusion_homogeneity}
\end{align}
for all $\lambda\in(0,\infty)$ and 
\begin{align}\label{same_rescaling}
    \frac{2}{3}\mu+\sigma=\gamma-1.
\end{align}
The condition (\ref{same_rescaling}) means that the fusion term and the coagulation term in (\ref{strongfusioneq}) rescale in a similar manner as the particle sizes are rescaled (keeping the geometry property). 

The following technical assumption on the kernel $r$ is needed for the existence of self-similar profiles:
\begin{align}\label{ode_fusion}
\begin{cases}
[\partial_{a}r(a,v)-\mu a^{-1} r(a,v)](a-c_{0}v^{\frac{2}{3}})+r(a,v)\geq 0,\textup{ and } \partial_{a}r(a,v)\leq B a^{-1}r(a,v),& \textup{ if } \mu>0;  \\
\partial_{a}r(a,v)(a-c_{0}v^{\frac{2}{3}})+r(a,v)\geq 0, \textup{ and } \partial_{a}r(a,v)\leq B a^{-1}r(a,v),&\textup{ if } \mu\leq 0,
\end{cases}
\end{align}
for all $(a,v)\in(0,\infty)^{2},$ with $a\geq c_{0}v^{\frac{2}{3}}$ and for some constant $B>0$. A particular case used in applications that satisfies the above mentioned properties is when $r(a,v)=a^{\mu}v^{\sigma}$, with $\mu\geq -1$ and $\sigma$ satisfying (\ref{same_rescaling}).
\subsubsection*{Physical interpretation of the results}
The main result that we prove in this article is that depending on the choice of the exponents $ \mu$ and $ \sigma$, we can have different behaviors for the solutions of the described model. 

For $\mu>0$, there exist volume-conserving self-similar solutions with the form (\ref{scale}). For these solutions, the fusion term is comparable to the coagulation term for particles of large sizes. On the other hand, we will obtain in the case when $\mu<0$ that we can have two different behaviors depending on the choice of initial data, see Figure \ref{fig4}. In particular, for some suitable initial data, the fusion term plays a negligible role compared to the coagulation operator $\mathbb{K}[f]$ in (\ref{strongfusioneq}). We will term the long-time behavior of the particle distribution $f$ in this case as ramification. 

In order to explain why we use this terminology, it is convenient to introduce the following notation. Given $H=H(a,v),$ we write
\begin{align*}
\langle H\rangle(t) :=\frac{\int_{(0,\infty)^{2}}H(a,v)f(a,v,t)\der v \der a}{\int_{(0,\infty)^{2}}f(a,v,t)\der v \der a}, \textup{ for any time } t\geq 0.
\end{align*}
 More precisely, for $\mu<0$ and keeping in mind that the fusion does not change the total volume, if we start with a distribution of particles for which 
 \begin{align*}
     \frac{\langle a\rangle (0)}{\big(\langle v\rangle (0)\big)^{\frac{2}{3}}}\geq \lambda_{0},
 \end{align*}
 for $\lambda_{0}$ sufficiently large, then we obtain the following behavior
\begin{align}\label{sphere excluded}
\frac{\langle a\rangle (t)}{\big(\langle v\rangle (t)\big)^\frac{2}{3}}\rightarrow\infty \textup{ as } t\rightarrow\infty.
\end{align}

Notice that (\ref{sphere excluded}) implies that for most of the particles the surface area is much larger than the area of a sphere with the same volume. It is relevant to notice that in the case of self-similar solutions of (\ref{strongfusioneq}) with the form (\ref{scale}), we have
\begin{align*}
\frac{\langle a\rangle (t)}{\big(\langle v\rangle (t)\big)^\frac{2}{3}}\approx 1 \textup{ as } t\rightarrow\infty.
\end{align*}

Actually, we will obtain a result stronger than (\ref{sphere excluded}). Namely, in the case $\mu<0,$ we obtain in addition
\begin{align}\label{stronger res}
\frac{c}{\langle v\rangle(0)}\leq \frac{\langle a\rangle(t)}{\langle v\rangle(t)}\leq \frac{\langle a\rangle(0)}{\langle v\rangle(0)},
\end{align}
for some constant $c>0$. Notice that (\ref{stronger res}) implies immediately (\ref{sphere excluded}) since, due to the coagulation of the particles, $\langle v\rangle (t)\rightarrow\infty$ as $t\rightarrow\infty$. We remark that (\ref{stronger res}) suggests that for most of the particles the surface area is comparable to the volume $a\approx v$, while (\ref{sphere excluded}) suggests that $a>>v^{\frac{2}{3}}$ as $t\rightarrow\infty$. In particular, particles satisfying $a>>v^{\frac{2}{3}}$ differ very much from spherical particles and they have a fractal-like, \textit{ramified} aspect, see Figure \ref{fig3}.
\\

\begin{figure}[h]%
\centering
\includegraphics[width=0.3\textwidth, height=3cm]{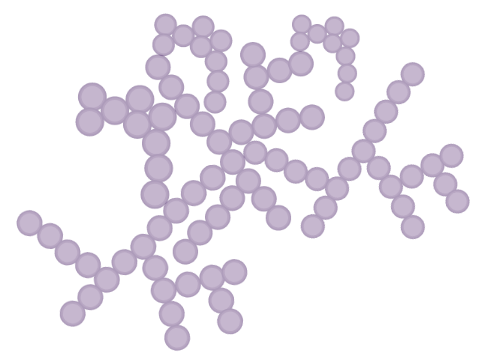}
\caption{Ramification}\label{fig3}
\end{figure}

However, if the fusion kernel $r$ is sufficiently large when $a$ and $v$ are of order one, we can obtain existence of self-similar profiles in the case $\mu< 0$. Actually, the proof covers in addition the case when $\mu=0$, which corresponds to the case of fusion kernels considered in \cite{KOCH1990419}. Thus, in the case $\mu<0$ we find two possible scenarios, see Figure \ref{fig4}.

\begin{figure}
\centering
 \begin{tikzcd}[ampersand replacement=\&]
  \& \underset{\textbf{large initial surface area}}{\textbf{$R_{0}\approx 1$} } \arrow{r} \& \textbf{ramification} \\
\fbox{\textbf{$\mu<0$}} \arrow{ur} \arrow{dr} \& \&   \\
 \& \textbf{$R_{0}\gg1$} \arrow{r} \& \textbf{ existence of self-similar profiles  }
\end{tikzcd}
\caption{Different scenarios in the case $\mu<0$}\label{fig4}
\end{figure}
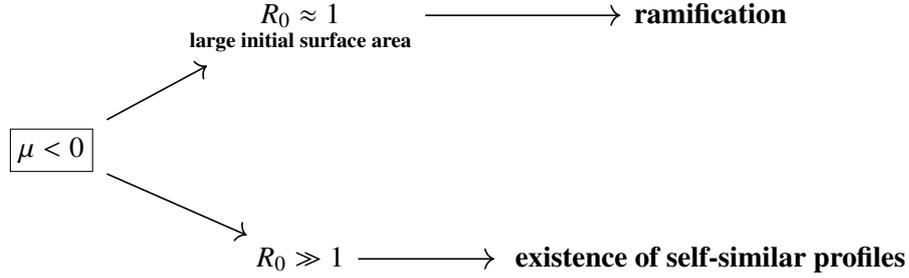

\subsubsection*{Multi-dimensional coagulation equations in the mathematical literature}
One can imagine situations in which the collision kernel $K$ and the fusion kernel $r$ do not rescale in the same manner as the particle size. In such situations we can expect to have one of the terms (fusion or coagulation) to be dominant for small size particles, and the other term to be the dominant one for large particles. The analysis of such type of models is also interesting from the point of view of applications to material science, (see \cite{book1}). 

If the fusion kernel $r$ is very large compared with the coagulation rate, we expect that the particles become spherical in very short times. Therefore, it is possible to approximate the solutions of (\ref{strongfusioneq}) by means of solutions of a coagulation model depending only on the variable $v$, i.e. an one-dimensional coagulation equation. The rigurous proof of this result is presented in \cite{cristian2023fast}.

Coagulation equations for particle distributions characterized by a single variable have been extensively studied. In particular, the long-time behavior for coagulation equations for which solutions can be explicitly computed has been studied in detail in  \cite{Menon_2004,menonpegoattractor,menonpegodynamical}. The existence of self-similar solutions for general classes of kernels has been obtained in \cite{ESCOBEDO,articlefournier}. Coagulation models including drift terms have been studied in several contexts. One example is the classical  Lifshitz-Slyozov-Wagner equation with encounters that was introduced in \cite{LIFSHITZ196135}. A rigorous analysis of the self-similar profiles for this model was studied in \cite{HERRMANN2009909, HERRMANN20092282, articlelau}. Models combining the effect of coagulation and particle growth have been studied extensively in the physical literature, cf. \cite[Chapter 11]{book1} and \cite{LUSHNIKOV2000651}. Rigorous mathematical results for these models can be found in \cite{gajewski}. 

Multi-dimensional coagulation equations have not been as extensively studied in the mathematical literature as the one-dimensional coagulation model.  Several discrete multi-component coagulation problems which are relevant in aerosol physics have been mentioned in \cite{wattis1}. A discrete version of the model in (\ref{strongfusioneq}) has been studied in \cite{wattis2}. The model considered in there includes coagulation of particles and an effect similar to the fusion of particles in (\ref{strongfusioneq}), which has been termed compaction. The diameter of the particles is restricted by the total number of monomers as well as by the isoperimetric inequality. The coagulation and the fusion rates are assumed to be constant. Due to this, the model considered in \cite{wattis2} is explicitly solvable using generating functions. The long-time behavior of the solutions which depends on the ratio between the fusion and coagulation kernels has been then analysed using the explicit formulas of the solutions. 

 In \cite{multicom1, multicom2,multicomponent3}, the mathematical properties of some classes of coagulation equations describing clusters that are composed of several types of monomers with different chemical composition are analysed. More recently, uniqueness of the solutions for the models of multi-component coagulation equations considered in \cite{multicom1, multicom2,multicomponent3} has been studied in \cite{throm}.

The main differences between the models studied in these papers and our model are the following:
\begin{itemize}
    \item The two variables used to describe the particles in this paper rescale in a different manner. Additionally, we consider coagulation kernels that do not have a strong dependence on the area variable. As a consequence,  the variables describing the clusters appear in a less symmetric manner;
    \item The proof in  \cite{multicom1, multicom2,multicomponent3} relies on the conservation of mass for each of the types of monomers. Due to the presence of the fusion term, the solutions of (\ref{strongfusioneq}) do not have two conserved quantities, but only the volume is conserved. 
\end{itemize}

\subsection{Notations and plan of the paper}

For $I\subset [0,\infty)^{2}$, we denote by $\textup{C}_{\textup{c}}(I)$ and $\textup{C}_{\textup{0}}(I)$ the space of continuous functions on $I$ with compact support and the space of continuous functions on $I$ which vanish at infinity, respectively, both endowed with the supremum norm. $\mathscr{M}_{+}(I)$ will denote the space of non-negative Radon measures, while $\mathscr{M}_{+,\text{b}}(I)$ will be the space of non-negative, bounded Radon measures, which we endow with the weak-$^{\ast}$ topology. We denote $\mathbb{R}_{>0}^{2}:=(0,\infty)^{2}$.
\

We make in addition the following simplifications:
\begin{itemize}
\item We use the notation $\eta:=(a,v)$. We will use interchangeably both notations for convenience. 
    \item We keep the notation $f(a,v)\der v \der a$ or $f(\eta)\der \eta$ for Radon measures, independently of the fact the measure may not be absolutely continuous with respect to the Lebesgue measure. 
    \item $M_{k,l}(f):=\int_{(0,\infty)^{2}}a^{k}v^{l}f(a, v)\der v \der a$, for some $k,l\in\mathbb{R}.$
    \item For a suitably chosen $\varphi:\mathbb{R}^{2}_{>0}\rightarrow \mathbb{R}$ and for $(a,v,a',v')\in(0,\infty)^{4}$, we will denote:
   \begin{align}
   \chi_{\varphi}(a,v,a',v')& :=\varphi(a+a',v+v')-\varphi(a,v)-\varphi(a',v'); \nonumber \\
   \langle \mathbb{K}[f],\varphi\rangle& :=\frac{1}{2}\int_{(0,\infty)^{2}}\int_{(0,\infty)^{2}} K(a,v,a',v')\chi_{\varphi} (a,v,a',v')f(a',v')f(a,v) \der v'\der a'\der v \der a. \nonumber
   \end{align}
   \item We use $C$ to denote a generic constant which may differ from line to line and depends only on the parameters characterizing the kernels $K$ and $r$.
      \item We use the symbols $\lesssim$ and $\gtrsim$ when the inequalities hold up to a constant, i.e. $f\lesssim g$ if and only if $f\leq C g$. In addition, for some $N\in\mathbb{N}$, we use the notation $i=\overline{1,N}$ to mean $i\in\mathbb{N},$ $1\leq i\leq N.$
\end{itemize}
\subsubsection*{Structure of the paper}
The structure of the paper is as follows. In Section \ref{setting}, we establish the setting and state the main definitions and results.

In Section \ref{sectionexistence}, we prove the existence of a self-similar solution when $\mu>0$. To this end, we first need to prove well-posedness for the time-dependent problem with a truncated kernel. It turns out that it does not seem feasible to obtain uniform estimates for large values of $a$ for the distribution $f$ if we use approximations of solutions that are compactly supported. In order to avoid this difficulty, we work with a space where large values of the area $a$ are controlled. Since in this space the fusion term will not be well defined, we work with a truncated version of the fusion term which increases linearly at infinity. In order to prove existence of a self-similar solution for the original problem, we need to obtain moment estimates that are uniform in the truncation parameters. The relevant moments to be estimated contain powers of $v$ and $a$. The moments containing only powers of $v$ can be estimated following the ideas in \cite{dust,ESCOBEDO} due to the fact that the fusion term does not affect the volume of particles. Nevertheless, the adaptation of the estimates in these papers is possible in spite of the fact that we have a coagulation model with two variables due to the choice of the space of functions non-compactly supported in the variable $a$ described before. The total area can be controlled making use of the contribution given by the fusion. This will prove to be enough to obtain existence of a self-similar profile. In Section \ref{sectionexistence}, we derive in addition estimates for higher order moments using an iterative argument. 

In Section \ref{case mu negative}, we show that we can have different behaviors for the solutions in the case when $\mu<0$, namely the existence of self-similar profiles as well as ramification. Ramification will be obtained by deriving estimates for the moments of the solutions of the time-dependent problem. On the other hand, in order to prove the existence of self-similar profiles in this case we cannot use the methods described in Section \ref{sectionexistence}. This is because of the fact that the estimates for large values of $a$ for positive $\mu$ are a consequence of the fast growth of the fusion ratio $r(\eta)$, which does not take place now. We will be able in this case to replace the fast growth of $r(\eta)$ with the presence of a sufficiently large constant in front of the fusion term.

\section{Setting and main results} \label{setting}

From the scaling (\ref{scale}), self-similar profiles with fusion satisfy formally the equation
\begin{align}
0=\frac{8}{3}g(\eta)+\frac{2}{3}a\partial_{a}g(\eta)+v\partial_{v}g(\eta)+(1-\gamma)\partial_{a}[r(\eta)(a-c_{0}v^{\frac{2}{3}})g(\eta)]+(1-\gamma)\mathbb{K}[g](\eta). \label{strong_self_sim}
\end{align}
In particular, if $g$ solves (\ref{strong_self_sim}) and $f$ satisfies (\ref{scale}), then $f$ solves (\ref{strongfusioneq}).

\
Since we work with physically relevant particles, i.e. the particles for which the isoperimetric inequality is satisfied, it is helpful to define the following space
\begin{align}
   \mathscr{M}^{I}_{+}(\mathbb{R}_{>0}^{2}):=\{h\in \mathscr{M}_{+}(\mathbb{R}_{>0}^{2})\text{  } | \text{  } h(\{a<c_{0}v^{\frac{2}{3}}\})=0\}. \label{isoperimetricinitial}
\end{align}
The superscript $I$ stands for isoperimetric. We endow the newly-defined space with the weak-$^{\ast}$ topology on $\mathscr{M}_{+}(\mathbb{R}_{>0}^{2})$. Similarly, we denote
\begin{align}
   \mathscr{M}^{I}_{+,b}(\mathbb{R}_{>0}^{2}):=\{h\in \mathscr{M}_{+,b}(\mathbb{R}_{>0}^{2})\text{  } | \text{  } h(\{a<c_{0}v^{\frac{2}{3}}\})=0\}.
\end{align}
In order to study the long-time behavior for the equation (\ref{strongfusioneq}), we analyse the time-dependent version of equation (\ref{strong_self_sim}). We will use the following concept of weak solutions for the time-dependent problem.
\begin{defi}\label{definitiontimedependent}
Assume $\alpha>0$. Let $g\in\textup{C}([0,\infty);\mathscr{M}^{I}_{+}(\mathbb{R}_{>0}^{2}))$. We say that $g$ is a solution for the weak version of the time-dependent fusion problem in self-similar variables if, for every $T>0$,
\begin{align*}
 \sup_{t\in[0,T]}  \int_{(0,\infty)^{2}}(v^{-\alpha}+v^{\beta})g(a,v,t)\der v\der a<\infty
    \end{align*}
and, for all $\varphi\in\textup{C}^{1}_{\textup{c}}([0,\infty);\compactfundif)$ and $t\in[0,\infty)$
\begin{align}
\int_{(0,\infty)^{2}}g(\eta,t)\varphi(\eta,t)\der \eta-\int_{(0,\infty)^{2}}g(\eta,0)&\varphi(\eta,0)\der \eta-\int_{0}^{t}\int_{(0,\infty)^{2}}g(\eta,s)\partial_{s}\varphi(\eta,s)\der \eta\der s=\nonumber\\ 
\int_{0}^{t}\int_{(0,\infty)^{2}}g(\eta,s)\varphi(\eta,s)\der\eta\der s-\frac{2}{3}\int_{0}^{t}\int_{(0,\infty)^{2}}g(\eta,s) a&\partial_{a}\varphi(\eta,s)\der\eta\der s-\int_{0}^{t}\int_{(0,\infty)^{2}}g(\eta,s)v\partial_{v}\varphi(\eta,s)\der \eta\der s \nonumber\\ 
+(1-\gamma)\int_{0}^{t}\langle \mathbb{K}[g](s),\varphi(s)\rangle\der s+(1-\gamma)&\int_{0}^{t}\int_{(0,\infty)^{2}}r( \eta)(c_{0}v^{\frac{2}{3}}-a)g(\eta,s)\partial_{a}\varphi(\eta,s)\der \eta\der s. \label{weak_form_time_dependent}
\end{align}
\end{defi}

Well-posedness of solutions of the form (\ref{weak_form_time_dependent}) has been studied in \cite{iulia}. In this paper, we focus on proving the existence of self-similar profiles and long-time behavior for solutions of equation (\ref{strongfusioneq}). We now give a precise meaning for (\ref{strong_self_sim}).
\begin{defi}\label{definitionselfsimilarprofile}
Assume $\alpha>0$. We will say that a measure $g\in \mathscr{M}^{I}_{+}(\mathbb{R}_{>0}^{2})$ is a self-similar profile for the two-dimensional coagulation-equation if
\begin{align}\label{solwelldef}
   \int_{(0,\infty)^{2}}(v^{-\alpha}+v^{\beta})g(a,v)\der v\der a<\infty 
    \end{align}
and for every $\varphi\in\textup{C}_{\textup{c}}^{1}(\mathbb{R}^{2}_{>0})$ the following equality is satisfied:
\begin{align}
\int_{(0,\infty)^{2}}g(\eta)\varphi(\eta)\der\eta-\frac{2}{3}\int_{(0,\infty)^{2}}g(\eta) & a\partial_{a}\varphi(\eta)\der\eta-\int_{(0,\infty)^{2}}g(\eta)v\partial_{v}\varphi(\eta)\der \eta\nonumber \\
+(1-\gamma)\langle \mathbb{K}[g],\varphi\rangle+(1-\gamma)\int_{(0,\infty)^{2}}  &r( \eta)(c_{0}v^{\frac{2}{3}}-a)g(\eta)\partial_{a}\varphi(\eta)\der \eta =0. \label{self_similar_equation}
\end{align}
\end{defi}

\subsection{The case $\mu>0$}
The following result states the existence of self-similar profiles in the case when $\alpha>0$.
\begin{teo}\label{themostimportanttheorem}
Let $\mu,\alpha >0$ and $v_{0}>0$. Assume $K$ is a continuous, non-negative kernel satisfying (\ref{homogeneity}), (\ref{kersym1}), (\ref{lower_bound_kernel}) and (\ref{alpha non neg}) and suppose that $r(a,v)$ satisfies (\ref{fusion_form}), (\ref{fusion_homogeneity}), (\ref{same_rescaling}) and (\ref{ode_fusion}). Then there exists a self-similar profile for the two-dimensional coagulation-equation in the sense of Definition \ref{definitionselfsimilarprofile} with $\int_{(0,\infty)^{2}}vg(a,v)\der v\der a=v_{0}$. In addition, $g$ satisfies $M_{n,k}(g)<\infty$, for all $n,k\in\mathbb{R}$.
 \end{teo}
The existence of self-similar profiles in the case $\mu>0$ can be explained since, in this regime, fusion overtakes coagulation for large values of $a$. Therefore, the fusion term keeps the particles with a shape that does not differ too much from that of spheres, and thus we can expect $a$ not to be too far away from $c_{0}v^{\frac{2}{3}}$.

We will use the following definition for self-similar profiles for coagulation kernels satisfying (\ref{alpha zero})
\begin{defi}\label{definitionselfsimilarprofile alpha zero}
Assume $\alpha=0$. We will say that a measure $g\in \mathscr{M}^{I}_{+}(\mathbb{R}_{>0}^{2})$ is a self-similar profile for the two-dimensional coagulation-equation if
\begin{align}\label{solwelldef alpha zero}
   \int_{(0,\infty)^{2}}(a+v^{\beta+1}+v^{\beta})g(a,v)\der v\der a<\infty 
    \end{align}
and for every $\varphi\in\textup{C}_{\textup{c}}^{1}(\mathbb{R}^{2}_{>0})$ the following equality is satisfied:
\begin{align}
\frac{2}{3}\int_{(0,\infty)^{2}}g(\eta)& av\partial_{a}\varphi(\eta)\der\eta+\int_{(0,\infty)^{2}}g(\eta)v^{2}\partial_{v}\varphi(\eta)\der \eta-(1-\gamma)\int_{(0,\infty)^{2}} vr( \eta)(c_{0}v^{\frac{2}{3}}-a)g(\eta)\partial_{a}\varphi(\eta)\der \eta\nonumber\\
& =(1-\gamma)\int_{(0,\infty)^{2}}\int_{(0,\infty)^{2}}K(\eta,\eta')g(\eta)g(\eta')v[\varphi(\eta+\eta')-\varphi(\eta)]\der \eta' \der \eta. \label{self_similar_equation for alpha zero}
\end{align}
\end{defi}
Notice that we obtain equation (\ref{self_similar_equation for alpha zero}) by replacing in (\ref{self_similar_equation}) the test function $\varphi$ with a test function of the form $v\varphi$. The new form of the equation is chosen since in the case $\alpha=0$ we expect the self-similar profiles to be singular for values of $v$ near zero and for that reason not all the terms in Definition \ref{definitionselfsimilarprofile} are well-defined. This also justifies why we are assuming condition (\ref{solwelldef alpha zero}) instead of (\ref{solwelldef}).
\begin{teo}\label{main teo case alpha zero}
Let $\alpha=0$. Assume $K$ is a continuous, non-negative kernel satisfying (\ref{homogeneity}), (\ref{kersym1}), (\ref{lower_bound_kernel}) and (\ref{alpha zero}). Suppose that $r(a,v)$ satisfies (\ref{fusion_form}), (\ref{fusion_homogeneity}), (\ref{same_rescaling}) and (\ref{ode_fusion}). If $\mu>0$, there exists a self-similar profile for the two-dimensional coagulation-equation, in the sense of Definition \ref{definitionselfsimilarprofile alpha zero}, satisfying  $M_{n,k}(g)<\infty$, for all $n\geq 0$ and $k\geq \gamma$.
\end{teo}
\begin{rmk}
We observe that the moment estimates obtained in Theorem \ref{main teo case alpha zero} imply estimates for additional moments due to the fact that the self-similar profiles are supported in the isoperimetric region $\{a\geq c_{0}v^{\frac{2}{3}}\}$, namely $M_{-n,\gamma+\frac{2}{3}n}(g)\leq c_{0}^{-n}M_{0,\gamma}(g)<\infty$.
\end{rmk}
\subsection{The case $\mu\leq0$}
When $\mu$ is negative, fusion takes place at a slower pace for particles with large area. We have two different behaviors depending on the fusion rate.

If we start with a sufficiently large fusion rate, a regime similar to the one where fusion overtakes coagulation occurs and thus self-similar profiles exist in this case too.
\begin{teo}[Self-similarity in the case of slow fusion]\label{self sim mu neg}
Let $\mu\leq 0$ and $\alpha>0$. Assume $K$ is a continuous, non-negative kernel satisfying (\ref{homogeneity}), (\ref{kersym1}), (\ref{lower_bound_kernel}) and (\ref{alpha non neg}). Suppose that $r(a,v)$ satisfies (\ref{fusion_form}), (\ref{fusion_homogeneity}), (\ref{same_rescaling}) and (\ref{ode_fusion}). Then, there exists $\lambda>1$, depending only on $K_{0},K_{1}$ and $\gamma$, such that for any $v_{0}>0$, if $r(a,v)$ satisfies (\ref{fusion_form}) with $R_{0}\geq \lambda v_{0}$, then there exists a self-similar profile $g$ for the two-dimensional coagulation-equation, in the sense of Definition \ref{definitionselfsimilarprofile}, with total volume $\int_{(0,\infty)^{2}}vg(a,v)\der v \der a=v_{0}.$
\end{teo}
\begin{rmk}
 Notice that Theorem \ref{self sim mu neg} applies also in the case when $\mu=0$, which corresponds to the type of fusion kernels considered in \cite{KOCH1990419}.    
\end{rmk}

In order to better understand the stated results, we  give some heuristic arguments. We can explain the condition needed for Theorem \ref{self sim mu neg}, namely that $R_{0}$ in (\ref{fusion_form}) needs to be sufficiently large, in the following manner stated below. 

 First we can assume without loss of generality that $v_{0}=1$ by means of a rescaling argument (see Appendix \ref{formal rescaling properties}). With this rescaling, the theorem states that if $R_{0}\geq \lambda$,  for some sufficiently large constant $\lambda>0$, then there exists a self-similar profile. 

Assume for simplicity that $r(a,v)=R_{0} a^{\mu}$ and that $\mu<0$. In other words $r(a,v)=R_{0} a^{\frac{3(\gamma-1)}{2}}$ in order to be consistent with condition (\ref{fusion_form}). Additionally, we work with a coagulation kernel $K\equiv 1$. The most important part is to estimate the total surface area. This is since the moments in the $v$ variable can be bounded using standard arguments used in the study of coagulation equations.

Assume without loss of generality that $a\geq 2c_{0}v^{\frac{2}{3}}$ since the region $\{c_{0}v^{\frac{2}{3}}\leq a\leq 2c_{0}v^{\frac{2}{3}}\}$ is bounded using uniform moment estimates in the $v$ variable. The rigorous proof of Theorem \ref{self sim mu neg} will be a generalization of the following idea.

Denote by $A(t):=\int_{(0,\infty)^{2}}ag(\eta,t)\der\eta$. We test formally in (\ref{weak_form_time_dependent}) with $\varphi\equiv a$. Equation (\ref{weak_form_time_dependent}) 
 becomes
\begin{align}\label{self sim area simplified}
    \partial_{t}A(t)& = \frac{1}{3}A(t)+R_{0}\int_{(0,\infty)^{2}}a^{\mu}(c_{0}v^{\frac{2}{3}}-a)g(\eta,t)\der\eta\nonumber\\
    &\leq \frac{1}{3}A(t)-\frac{R_{0}}{2}\int_{(0,\infty)^{2}}a^{\mu+1}g(\eta,t)\der\eta.
\end{align}
Fix $\epsilon\in(0,1)$. Since $\mu<0$, by Young's inequality, we have that there exists some $\overline{\lambda}_{\epsilon}>0$, depending on $\epsilon$, sufficiently large such that
\begin{align}
    \frac{4}{3}a\leq \overline{\lambda}_{\epsilon} a^{\mu+1}+\epsilon a^{2}.
\end{align}
Choosing $R_{0}\geq 2\overline{\lambda}_{\epsilon}$, (\ref{self sim area simplified})  becomes
\begin{align}\label{epsilon small}
    \partial_{t}A(t)\leq -A(t)+\epsilon\int_{(0,\infty)^{2}}a^{2}g(\eta,t)\der\eta.
\end{align}
We then analyse the moment $M_{2,0}$. To this end, we test (\ref{self sim area simplified}) with $\varphi\equiv a^{2}$. Since $K\equiv 1$ and the fusion term is non-positive, we deduce
\begin{align}\label{equa}
\partial_{t}\int_{(0,\infty)^{2}}a^{2}g(\eta,t)\der \eta \leq - \frac{1}{3}\int_{(0,\infty)^{2}}a^{2}g(\eta,t)\der \eta+(1-\gamma)\bigg(\int_{(0,\infty)^{2}}ag(\eta,t)\der \eta\bigg)^{2}.
\end{align}
Combining (\ref{epsilon small}) and (\ref{equa}) and choosing $\epsilon$ to be sufficiently small, it follows that 
\begin{align}\label{hopefully last}
\partial_{t}\bigg(A(t)+\int_{(0,\infty)^{2}}a^{2}g(\eta,t)\der \eta \bigg)\leq -\frac{1}{6}\int_{(0,\infty)^{2}}a^{2}g(\eta,t)\der \eta-A(t)+(1-\gamma)\big(A(t)\big)^{2}.
\end{align}
In other words, using the notation $D(t):=\int_{(0,\infty)^{2}}a^{2}g(\eta,t)\der \eta+A(t)$, we obtain
\begin{align}\label{now this is}
\partial_{t}D(t)\leq -\frac{1}{6}D(t)+(1-\gamma)\big(D(t)\big)^{2}.
\end{align}
 If we take an initial condition $g_{\textup{in}}$ such that $D(0)\leq \frac{1}{12(1-\gamma)}$, we will have by (\ref{now this is}) that $D(t)\leq \frac{1}{12(1-\gamma)}$ is an invariant region in time. This enables us to use standard methods used in the study of coagulation equations to conclude that there exists a self-similar profile.

Notice that we use in the argument that  $R_{0}$ is large enough.

In order to prove the ramification result, we first prove the existence of a weak solution for the time-dependent fusion problem in self-similar variables which satisfies some suitable moment estimates.

\begin{prop}\label{prop ramification}
Assume $\alpha>0$ and $\mu<0$. Assume $K$ is a continuous, non-negative kernel satisfying (\ref{homogeneity}), (\ref{kersym1}), (\ref{lower_bound_kernel}) and (\ref{alpha non neg}). Suppose that $r(a,v)$ satisfies (\ref{fusion_form}), (\ref{fusion_homogeneity}), (\ref{same_rescaling}) and (\ref{ode_fusion}). Assume $g_{\textup{in}}\in\mathscr{M}_{+}^{I}(\mathbb{R}_{>0}^{2})$ with $\int_{(0,\infty)^{2}}vg_{\textup{in}}\der v \der a=v_{0}$. There exists a constant $C_{1}(v_{0})>0$, depending on $v_{0}$, $K_{0}$, $K_{1}$, and $\gamma$, cf. (\ref{homogeneity}) and (\ref{lower_bound_kernel}), such that,  if
\begin{align}\label{g in 2}
\int_{(0,\infty)^{2}}v^{x_{1}}g_{\textup{in}}(a,v)\der v \der a\leq C_{1}(v_{0}),
\end{align}
where $x_{1}\in\{\frac{\sigma}{|\mu|},\gamma-\frac{1}{3}\}$ and
\begin{align}\label{g in 1}
    \int_{(0,\infty)^{2}}(a^{\mu}+a^{2})(v^{-\alpha-1}+v^{\max\{1+\tilde{\epsilon},\sigma+\frac{2}{3}\}})g_{\textup{in}}(a,v)\der v \der a <\infty,
    \end{align}
for some $\tilde{\epsilon}\in(0,1)$, then there exists a weak solution of the time-dependent fusion problem $g$ as in Definition \ref{definitiontimedependent}  with $g\in\textup{C}([0,\infty);\mathscr{M}^{I}_{+}(\mathbb{R}_{>0}^{2}))$, which in addition satisfies the following moment estimates.
\begin{align} \label{cond2_non-existence}
\sup_{t\in[0,\infty)}\int_{(0,\infty)^{2}}v^{x_{1}}g(a,v,t)\der v \der a\leq C_{1}(v_{0})
\end{align}
 and, for every $T>0$,
\begin{align}\label{conditions_non-existence}
  \sup_{t\in[0,T]}  \int_{(0,\infty)^{2}}(a^{\mu}+a^{2})(v^{-\alpha-1}+v^{\max\{1+\tilde{\epsilon},\sigma+\frac{2}{3}\}})g(a,v,t)\der v \der a<\infty.
    \end{align}
\end{prop}
\begin{rmk}
  We have that $C_{1}(v_{0})=\overline{C}_{1}v_{0}^{\xi(x_{1}-\gamma)}$, where $\overline{C}_{1}>0$ depends only on $K_{0}$, $K_{1}$, and $\gamma$, cf. (\ref{homogeneity}) and (\ref{lower_bound_kernel}), and  $\xi$ is as in (\ref{scale}). For more details explaining the rescaling properties of $g$, see Appendix \ref{formal rescaling properties}.
\end{rmk}
\begin{rmk}
The moment estimates in (\ref{conditions_non-existence}) are needed in order to prove that we are able to test the equation (\ref{weak_form_time_dependent}) with $\varphi\equiv a$. Actually, we do not need estimates for moments of the form $M_{2,0}$ in order to prove this, but we keep the form (\ref{conditions_non-existence}) in order to emphasize that more general moment estimates can be obtained.
    
\end{rmk}
If $\mu<0$ and we start with sufficiently large surface area, we can expect a fast growth in the area in self-similar variables since coagulation overtakes fusion for large particles. Notice that, in equation  (\ref{strongfusioneq}), if we ignore the fusion term, we obtain a particle distribution $f$ for which the area stays constant in time. The exponential growth stated in the next theorem is equivalent to a lower estimate for the total area associated to the distribution $f$, cf. (\ref{stronger res}).
\begin{teo}\label{remarkramification}
Assume $\alpha>0$ and $\mu<0$. Assume $K$ is a continuous, non-negative kernel satisfying (\ref{homogeneity}), (\ref{kersym1}),  
(\ref{lower_bound_kernel}) and (\ref{alpha non neg}). Suppose that $r(a,v)$ satisfies (\ref{fusion_form}), (\ref{fusion_homogeneity}), (\ref{same_rescaling}) and (\ref{ode_fusion}). Let $g\in\textup{C}([0,\infty);\mathscr{M}^{I}_{+}(\mathbb{R}_{>0}^{2}))$ be a solution of the time-dependent fusion problem as in Definition \ref{definitiontimedependent} with total volume of particles equal to $v_{0}$ and satisfying (\ref{cond2_non-existence}) and (\ref{conditions_non-existence}). Then the following holds: there exists a constant $C_{2}(v_{0})>0$, depending on $v_{0}$, $K_{0}$, $K_{1}$, $\gamma,$ $\mu$ and $\sigma$ , cf. (\ref{homogeneity}), (\ref{lower_bound_kernel}) and (\ref{fusion_form}), such that, if
\begin{align}\label{r1}
    R_{1}\leq v_{0},
\end{align}
where $R_{1}$ is as in (\ref{fusion_form}) and
\begin{align}\label{initial area large}
    \int_{(0,\infty)^{2}}ag_{\textup{in}}(a,v)\der v \der a\geq C_{2}(v_{0}),
\end{align}
then
\begin{align}\label{total area large}
    \int_{(0,\infty)^{2}}g(a,v,t)a\der v \der a\geq C_{\mu}C_{2}(v_{0})\textup{e}^{\frac{1}{3}t},
\end{align}
for some $C_{\mu}$ depending on $\mu$.
\end{teo}
\begin{rmk}
We have that $C_{2}(v_{0})=\overline{C}_{2}v_{0}^{\xi(\frac{2}{3}-\gamma)}$, for $\xi$ as in (\ref{scale}), where $\overline{C}_{2}>0$ depends only on $K_{0}$, $K_{1}$, $\gamma,$ $\mu$ and $\sigma$ , cf. (\ref{homogeneity}), (\ref{lower_bound_kernel}) and (\ref{fusion_form}).  For more details about the rescaling properties of $g$, see Appendix \ref{formal rescaling properties}.
\end{rmk}
Notice that Theorem \ref{remarkramification} holds for any weak solution in the sense of Definition \ref{definitiontimedependent} that satisfies in addition  (\ref{cond2_non-existence}), (\ref{conditions_non-existence}) and (\ref{initial area large}). On the other hand, combining Proposition \ref{prop ramification} and Theorem \ref{remarkramification}, we obtain the following
\begin{cor}
Suppose that $g_{\textup{in}}\in\mathscr{M}_{+}^{I}(\mathbb{R}_{>0}^{2})$  and satisfies (\ref{g in 2}), (\ref{g in 1}), (\ref{r1}) and (\ref{initial area large}), with $C_{1}(v_{0})$ and $C_{2}(v_{0})$ as in Proposition \ref{prop ramification} and Theorem \ref{remarkramification}, respectively. Then there exists a weak solution for the time-dependent fusion problem in the sense of Definition \ref{definitiontimedependent} which satisfies  (\ref{total area large}).
\end{cor}

Theorem \ref{remarkramification} can be understood in the following manner. Identically as before, we can assume without loss of generality that $v_{0}=1$ by means of a rescaling argument (see Appendix \ref{formal rescaling properties}). With this rescaling, is we have $R_{1}\leq 1$ and $\int_{(0,\infty)^{2}}ag_{\textup{in}}(a,v)\der v \der a\geq C$, for some sufficiently large constant $C>0$, then ramification occurs. 

We now provide some heuristic explanation in order to justify the validity of Theorem \ref{remarkramification}. We explain below the condition needed for Theorem \ref{remarkramification}, namely condition (\ref{initial area large}).

Assume for simplicity that $r(a,v)=a^{\mu}$ and that $\mu\leq -1$. In other words, in order to be consistent with the condition (\ref{fusion_form}), $r(a,v)=a^{\frac{3(\gamma-1)}{2}}$. The rigorous proof of Theorem \ref{remarkramification} will be a generalization of the following idea.

Test formally in (\ref{weak_form_time_dependent}) with $\varphi\equiv a$. We have $  \langle \mathbb{K}[g],\varphi\rangle=0.$ Denote by $A(t):=\int_{(0,\infty)^{2}}ag(\eta,t)\der\eta$. Equation (\ref{weak_form_time_dependent}) 
 becomes
\begin{align}\label{area ramification simplified}
    \partial_{t}A(t)& = \frac{1}{3}A(t)+\int_{(0,\infty)^{2}}a^{\mu}(c_{0}v^{\frac{2}{3}}-a)g(\eta,t)\der\eta\nonumber\\
    &\geq  \frac{1}{3}A(t)-\int_{(0,\infty)^{2}}a^{\mu+1}g(\eta,t)\der\eta\nonumber\\
     &\geq  \frac{1}{3}A(t)-c_{0}^{\mu+1}\int_{(0,\infty)^{2}}v^{\gamma-\frac{1}{3}}g(\eta,t)\der\eta.
\end{align}
It turns out that we can prove 
\begin{align}\label{volume ramification simplified}
    \int_{(0,\infty)^{2}}v^{\gamma-\frac{1}{3}}g(\eta,t)\der\eta\leq \max\{\int_{(0,\infty)^{2}}v^{\gamma-\frac{1}{3}}g_{\textup{in}}(\eta)\der\eta,C\},
\end{align}
for some fixed constant $C>0$. The proof of this result is made in a similar manner as the proof of the analogous estimate for the one-dimensional coagulation equations. Thus, (\ref{area ramification simplified}) becomes
\begin{align}\label{last}
    \partial_{t}A(t)& \geq \frac{1}{3}A(t)-C,
\end{align}
for some constant $C>0$ depending only on $g_{\textup{in}}$.

From (\ref{last}), we deduce that if $A(0)$ is sufficiently large, then $A(t)$ behaves like $\textup{e}^{\frac{1}{3}t}$. This is the content of the ramification result in Theorem \ref{remarkramification}.
%It suffices thus to find $g_{\textup{in}}$ such that
%\begin{align}\label{random ineq}
   % 6 c_{0}^{\mu+1}\int_{(0,\infty)^{2}}v^{\gamma-\frac{1}{3}}g_{\textup{in}}(\eta)\der\eta\leq  \int_{(0,\infty)^{2}}a g_{\textup{in}}(\eta)\der\eta.
%\end{align}
%Notice that, although we do no analyse it in this paper, (\ref{random ineq}) in the case $\gamma=1$ is just a consequence the isoperimetric inequality up to a change of constant.

\section{Existence of self-similar solutions when $\mu>0$}\label{sectionexistence}
The strategy for the proof of Theorem \ref{themostimportanttheorem} (and Theorem \ref{main teo case alpha zero}) follows the approach of obtaining self-similar profiles as a fixed point of a truncated version of the time-dependent problem by showing invariance in time of a compact set.  It is convenient to work with truncated versions of the coagulation kernel, as well as a modified fusion rate. This is done in order to avoid singular behavior and unbounded terms. Notice that, since $g$ is supported in the region $\{a\geq c_{0}v^{\frac{2}{3}}\}$, information about one of the variables implies some information over the other. Estimates for moments depending only on $v$ follow then in the same manner as for the one-dimensional coagulation equation, due to the particular form of the coagulation kernel.
 
 \
 In order to define the truncated problem, we introduce the following functions. For $\epsilon\in(0,1)$ and $R>1$, we define a truncation $K_{\epsilon,R}:(0,\infty)^{4}\rightarrow[0,\infty)$ for the coagulation kernel  to be a continuous function such that:
\begin{align}\label{truncation_kernel_selfsimilar}
 K_{\epsilon,R}(a,v,a',v')=\min\{K(a,v,a',v'), 2^{1+\beta}K_{0}\epsilon^{-\alpha}R^{\beta}\},
\end{align}
 where $K$ satisfies (\ref{lower_bound_kernel}) and take $\xi_{R}:\mathbb{R}_{>0}\rightarrow [0,1]$ to be continuous and defined in the following manner:
 \begin{align}
     \xi_{R}(v) & =0, & &\text{ when } v\geq 2R, \\
     \xi_{R}(v)& =1, & &\text { on } (0,R].
 \end{align}
We first discuss on how the existence of strong solutions for the truncated version of (\ref{strong_self_sim}) will be proven. We take a truncation for the linear transport terms, namely we take $\Theta_{\epsilon}:(0,\infty)\rightarrow \mathbb{R}$ to be a smooth monotonically increasing function such that:
\begin{align}\label{thetaepsilon}
\Theta_{\epsilon}(v)=
\begin{cases}
1, & v>2\epsilon, \\
0, &v\leq \epsilon.
\end{cases}
\end{align}
The main issue is to find a suitable subset of $\mathscr{M}_{+}^{I}(\mathbb{R}_{>0}^{2})$ in which we can obtain uniform estimates in time. Thus, we first take a cut-off near the origin and show that the support of $g$ remains in this region. As mentioned before, information about the behavior of $v$ near the origin is enough to control $a$ near the origin. However, this is not the case for large values of $a$ and thus we have to deal with the fusion rate.

\
We replace the fusion term by terms linearly increasing in the area variable in order to avoid the characteristics to arrive from infinity in finite time. The linear growth in $a$ will enable us to test with functions that are not compactly supported in the area variable. So, for $\varphi\in\vanishfundif$, we analyse the following regularized version of (\ref{strong_self_sim}):
\begin{align}\label{regularizedformequation}
\int_{(0,\infty)^{2}}g(\eta,t)\Theta_{\epsilon}(v)\varphi(\eta)\der\eta-\frac{2}{3}\int_{(0,\infty)^{2}}g(\eta,t)\Theta_{\epsilon}(v) a\partial_{a}\varphi(\eta)\der\eta-\int_{(0,\infty)^{2}}&g(\eta,t)\Theta_{\epsilon}(v)v\partial_{v}\varphi(\eta)\der \eta \nonumber \\
+(1-\gamma)\langle \mathbb{K}_{\epsilon,R}[g],\varphi\rangle+(1-\gamma)\int_{(0,\infty)^{2}}r_{\delta}(\eta)(c_{0}v^{\frac{2}{3}}-a)g(\eta,t)\partial_{a}\varphi(\eta)\der \eta =\partial_{t}&\int_{(0,\infty)^{2}}g(\eta,t)\varphi(\eta)\der \eta,
\end{align}
where
\begin{align}\label{kernel_term}
    \langle \mathbb{K}_{\epsilon,R}[g],\varphi\rangle:=\frac{1}{2}\int_{(0,\infty)^{2}}\int_{(0,\infty)^{2}}K_{\epsilon,R}(a,v,a',v')\xi_{R}(v+v')\chi_{\varphi}(\eta,\eta')\der \eta'\der \eta.
\end{align}
We have replaced the fusion term $r(a,v)$ by
\begin{align}\label{cutfusion}
r_{\delta}(a,v):= \frac{r( \eta)\max\{v^{\sigma}, L\delta\}}{v^{\sigma}(1+\delta a^{\mu})},
\end{align}
for $\delta\in(0,1)$ and where we denoted
\begin{align}\label{definition_L}
    L:=\frac{12}{R_{0}(1-\gamma)},
\end{align}
where $R_{0}$ is as in (\ref{fusion_form}). Notice that $r_{\delta}(\eta)\rightarrow r(\eta)$ for fixed $\eta$ as $\delta\rightarrow 0$. $L$ was chosen in such a way to derive uniform estimates for the total area of solutions. This means it has to be sufficiently large in order to compensate for the linear transport term appearing due to the coagulation kernel.

For some $\tilde{g}\in\mathscr{M}^{I}_{+}(\mathbb{R}_{>0}^{2})$ such that
\begin{align}\label{finalization_space}
\int_{(0,\infty)^{2}}(1+a)\tilde{g}(a,v)\der v \der a<\infty,
\end{align}
we define the space
\begin{align}\label{existencespace for f}
U_{\epsilon,R}&:= \{\tilde{g}\in\mathscr{M}^{I}_{+}(\mathbb{R}_{>0}^{2}), \tilde{g}\Big(\mathbb{R}^{2}_{>0}\setminus[c_{0}\epsilon^{\frac{2}{3}},\infty)\times[\epsilon,2R)\Big)=0, \tilde{g} \textup{ satisfies } (\ref{finalization_space})\}.
\end{align}

In this section we will prove the following technical results.
\begin{prop}\label{propweaksol}
Take $K_{\epsilon,R}$ as above, i.e. (\ref{truncation_kernel_selfsimilar}) holds, and assume it satisfies (\ref{kersym1}). Let $g_{\textup{in},R}\in\mathscr{M}^{I}_{+,\textup{b}}(\mathbb{R}_{>0}^{2})\cap U_{\epsilon,R}$. There exists a unique solution $g_{\epsilon,R,\delta}\in\textup{C}^{1}([0,\infty);\mathscr{M}^{I}_{+}(\mathbb{R}_{>0}^{2}))$, $g_{\epsilon,R,\delta}(t)\in U_{\epsilon,R}$, that satisfies
\begin{align*}
\sup_{t\in[0,T]}\int_{(0,\infty)^{2}}(1+a)g_{\epsilon,R,\delta}(a,v,t)\der v \der a<\infty,
\end{align*}
for all times $T\in[0,\infty),$ for the weak formulation of the coagulation equation (\ref{regularizedformequation}) with initial datum $g_{\textup{in},R}$.
\end{prop}

The proof of this proposition will be the content of Subsection \ref{chapterone}.

Let $K_{\epsilon,R}$ defined as in (\ref{truncation_kernel_selfsimilar}) and assume $K$ and $r$ satisfy the conditions stated in Theorem \ref{themostimportanttheorem} or in Theorem \ref{main teo case alpha zero}. Let $T>0$. We define the map $S(t):U_{\epsilon,R}\rightarrow U_{\epsilon,R}$ in the following way: 
\begin{align}\label{definition semigroup}
S(t)g_{\textup{in},R}=g_{\epsilon,R,\delta}(\cdot,\cdot,t),
\end{align} for all $t\in[0,T]$, where $g_{\epsilon,R,\delta}$ is the unique solution of the weak formulation of the coagulation equation with coagulation kernel $K_{\epsilon,R}$ found Proposition \ref{propweaksol}.

In order to prove Theorem \ref{themostimportanttheorem} (and Theorem \ref{main teo case alpha zero}), the next lemma will be useful:
\begin{prop}\label{lemmaweaksol}
Let $\epsilon\in(0,1),R>1,\delta\in(0,1)$. Let $K_{\epsilon,R}$ defined as in (\ref{truncation_kernel_selfsimilar}) and $r_{\delta}$ as in (\ref{cutfusion}). Assume $K$ and $r$ satisfy the conditions stated in Theorem \ref{themostimportanttheorem} or in Theorem \ref{main teo case alpha zero}. Let $T>0$ and $S(t):U_{\epsilon,R}\rightarrow U_{\epsilon,R}$ as in (\ref{definition semigroup}), for $t\in[0,T]$.

Let $\mu>0$ in (\ref{fusion_form}). Then there exist constants $c_{0,-\alpha-\tilde{\epsilon}},$ $c_{0,\tilde{m}},$ $c_{1,0}>0$, where $\tilde{\epsilon}\in(0,1)$ and $\tilde{m}>\max\{1,\frac{|\sigma|}{\mu}\}$, with the property that the set $\omega(\epsilon,R,\delta)$, defined as
\begin{align}
    \omega(\epsilon,R,\delta):=&\{ M_{0,1}(g_{\epsilon,R,\delta})=1; M_{0,-\alpha-\tilde{\epsilon}}(g_{\epsilon,R,\delta})\leq c_{0,-\alpha-\tilde{\epsilon}}; M_{0,\tilde{m}}(g_{\epsilon,R,\delta})\leq c_{0,\tilde{m}};M_{1,0}(g_{\epsilon,R,\delta})\leq c_{1,0}\}\label{THESET}\\
    &\textup { if $\alpha>0$ and }  \nonumber \\
    \omega(\epsilon,R,\delta):=&\{ M_{0,1}(g_{\epsilon,R,\delta})=1; M_{0,\gamma}(g_{\epsilon,R,\delta})\leq c_{0,\gamma}; M_{0,\tilde{m}}(g_{\epsilon,R,\delta})\leq c_{0,\tilde{m}};  M_{1,0}(g_{\epsilon,R,\delta})\leq c_{1,0}\}\label{THESET2}\\
    &\textup { if $\alpha=0,$}\nonumber
\end{align}
is preserved in time uniformly in $\epsilon,R,\delta$ under equation (\ref{regularizedformequation}), i.e. $S(t)\omega(\epsilon,R,\delta)\subseteq\omega(\epsilon,R,\delta),$ for all $t\in[0,T]$. 
\end{prop}

The proof of this proposition will be given in Subsection \ref{chapter_two}.

 To get volume-conserving solutions, we need to control the total area. In order to obtain this, we need to assume that an additional moment is bounded and this is why the moment $M_{0,\frac{|\sigma|}{\mu}}$ appears in Proposition \ref{lemmaweaksol}.
\subsection{Existence of solutions for the truncated time-dependent problem}\label{chapterone}
We define the functions $A,V:\mathbb{R}_{>0}^{2}\times\mathbb{R}_{\geq 0}\rightarrow \mathbb{R}$ in the following manner, by looking at the characteristic equations:
\begin{equation}\label{changeofvartable}
\left\{\begin{aligned}
\partial_{t}A(a_{0},v_{0},t)&=(1-\gamma)r_{\delta}(A,V)(c_{0}V^{\frac{2}{3}}-A)-\frac{2}{3}\Theta_{\epsilon}(V)A; \\
 \partial_{t}V(a_{0},v_{0},t)&=-\Theta_{\epsilon}(V)V,
   \end{aligned}\right.
   \end{equation}
   with initial conditions
   \begin{equation}\label{initial conditions}
\left\{\begin{aligned}
 A(a_{0},v_{0},0)&=&a_{0}; \\
V(a_{0},v_{0},0)&=&v_{0}.
   \end{aligned}\right.
   \end{equation}
By (\ref{ode_fusion}), we have the following inequality, that we write for future reference
\begin{align*}
\partial_{A}[r_{\delta}(A,V)(A-c_{0}V^{\frac{2}{3}})]
&=\max\{V^{\sigma},L\delta\}\frac{[\partial_{A}r(A,V)-\frac{\mu\delta A^{\mu-1}}{1+\delta A^{\mu}}r(A,V)](A-c_{0}V^{\frac{2}{3}})+r(A,V)}{V^{\sigma}(1+\delta A^{\mu})}\\
&\geq \max\{V^{\sigma},L\delta\}\frac{[\partial_{A}r(A,V)-\mu A^{-1}r(A,V)](A-c_{0}V^{\frac{2}{3}})+r(A,V)}{V^{\sigma}(1+\delta A^{\mu})}\geq 0.
\end{align*}

 Fix $t\geq 0$. We denote the pair $(A(a_{0},v_{0},t),V(a_{0},v_{0},t))=: \phi_{t}(a_{0},v_{0})$. Observe that, due to the form of the equations in (\ref{changeofvartable}), the function $V$ is independent of $a_{0}$. In particular, there exists a family of functions $\{y_{t}\}_{t\geq 0}:\mathbb{R}_{>0}\rightarrow\mathbb{R}_{>0}$ such that $y_{t}(v_{0}):=V(a_{0},v_{0},t)$.
    In the same manner, we fix $v_{0}$ and we define $x_{t,v_{0}}:\mathbb{R}_{>0}\rightarrow\mathbb{R}$ by $x_{t,v_{0}}(a_{0}):=A(a_{0},v_{0},t)$.
    
We gather in the following proposition a list of properties for the solutions of the system (\ref{changeofvartable}) that will be used throughout the paper.
\begin{prop}[Properties of the characteristics]\label{observationofthedomain2}
Let $A,V$ be as in (\ref{changeofvartable}) with initial conditions (\ref{initial conditions}). Then
    \begin{enumerate}
        \item[(1)] $V(a_{0},v_{0},t)\equiv v_{0},$ when $V\leq \epsilon$;
        \item[(2)] $V(a_{0},v_{0},t)=v_{0}\textup{e}^{-t}$, when $V>2\epsilon$.
    \end{enumerate}
    We have in addition that
    \begin{enumerate}
        \item[(3)] Let $P(v_{0},t)$, with $(v_{0},t)\in(0,\infty)\times[0,\infty)$ be a solution of $\partial_{t}P=-\Theta_{\epsilon}(P)P$. Define then $H=c_{0}P^{\frac{2}{3}}$, for all $t\geq 0$ and all $v_{0}\in(0,\infty)$. Then $H$ solves the first equation in (\ref{changeofvartable}). In particular, we have that, if $(a_{0},v_{0})\in\partial \tilde{\textup{S}}$, then $(A,V)\in\partial \tilde{\textup{S}}$, where $\tilde{\textup{S}}$ was defined in (\ref{set iso});
        \item[(4)] The mapping $\phi_{t}$ sends $\tilde{\textup{S}}$ to $\tilde{\textup{S}}$, i.e. $\phi_{t}(\tilde{\textup{S}})\subseteq\tilde{\textup{S}}$, where $\tilde{\textup{S}}$ was defined in (\ref{set iso});
        \item[(5)] If $(a_{0},v_{0})\in\tilde{\textup{S}}$, we have that $\partial_{a_{0}}x_{t,v_{0}}\geq  0$ and $\partial_{v_{0}}y_{t}\geq 0$.
    \end{enumerate}
\end{prop}
\begin{proof}
The fact that $A(a_{0},v_{0},t)$ and $V(a_{0},v_{0},t)$ are well-defined for $(a_{0},v_{0})\in\tilde{\textup{S}}$, with $\tilde{\textup{S}}$ as in  (\ref{set iso}), and for $t\geq 0$ follows from standard ODE theory, as well as the choice of truncation $r_{\delta}$ in (\ref{cutfusion}), which avoids blow-up in finite time.
Statement $(1)$ follows from the fact that $\Theta_{\epsilon}(v)=0$, when $v\leq \epsilon$. Statement $(2)$ follows from the fact that $\Theta_{\epsilon}(v)=1$, when $v\geq 2\epsilon$.

In order to prove Statement $(3)$, take $H=c_{0}P^{\frac{2}{3}}$. Notice on one hand that
\begin{align*}
    \partial_{t}H=\partial_{t}\bigg(c_{0}P^{\frac{2}{3}}\bigg)=\frac{2c_{0}}{3}P^{-\frac{1}{3}}\partial_{t}P=-\frac{2c_{0}}{3}\Theta_{\epsilon}(P)P^{\frac{2}{3}}.
\end{align*}
On the other hand, it follows that
\begin{align*}
(1-\gamma)r_{\delta}(H,P)(c_{0}P^{\frac{2}{3}}-H)-\frac{2}{3}\Theta_{\epsilon}(P)H= -\frac{2}{3}\Theta_{\epsilon}(P)H=-\frac{2c_{0}}{3}\Theta_{\epsilon}(P)P^{\frac{2}{3}}.
\end{align*}
Thus, $H=c_{0}P^{\frac{2}{3}}$ solves the first equation in (\ref{changeofvartable}). The statement then follows from the uniqueness theory of ODE's.

Statement $(4)$ is a consequence of Statement $(3)$ and also of uniqueness theory of ODE's.

%In order to prove Statement 4., we define
%\begin{align}\label{rescaling defition}
%\begin{pmatrix}
%\overline{A}\\
%\overline{V}
%\end{pmatrix}=\begin{pmatrix}
%\textup{e}^{\frac{2}{3}\int_{0}^{t}\Theta_{\epsilon}(V(s))\der s}A\\
%\textup{e}^{\int_{0}^{t}\Theta_{\epsilon}(V(s))\der s}V
%\end{pmatrix}.
%\end{align}
%Notice that the pair $(\overline{A}, \overline{V})$ solves
%\begin{align}\label{rescaled system}
%\partial_{t}\begin{pmatrix}
%\overline{A}(a_{0},v_{0},t)\\
%\overline{V}(a_{0},v_{0},t)
%\end{pmatrix}=\begin{pmatrix}
%(1-\gamma)\overline{r}_{\delta}(\overline{A},\overline{V})(c_{0}\overline{V}^{\frac{2}{3}}-\overline{A})\\
%0
%\end{pmatrix}
%\end{align}
   %   with initial conditions
  % \begin{equation*}
%\left\{\begin{aligned}
% \overline{A}(a_{0},v_{0},0)&=&a_{0}; \\
%\overline{V}(a_{0},v_{0},0)&=&v_{0},
 %  \end{aligned}\right.
 %  \end{equation*}
%where 
%\begin{align*}
   % \overline{r}_{\delta}(\overline{A},\overline{V})= r_{\delta}(\textup{e}^{\frac{2}{3}\int_{0}^{t}\Theta_{\epsilon}(V(s))\der s}\overline{A},\textup{e}^{\int_{0}^{t}\Theta_{\epsilon}(V(s))\der s}\overline{V}).
%\end{align*}
%We observe that $\overline{A}(a_{0},v_{0},t)=c_{0}\overline{V}^{\frac{2}{3}}(a_{0},v_{0},t)=c_{0}v_{0}^{\frac{2}{3}}$ solves (\ref{rescaled system}) and that the function $(1-\gamma)\overline{r}_{\delta}(a,v)(c_{0}v^{\frac{2}{3}}-a)$ changes sign if $a>c_{0}v^{\frac{2}{3}}$ or $a<c_{0}v^{\frac{2}{3}}$. Thus, we have that the solutions of (\ref{rescaled system}) map $\tilde{\textup{S}}$ onto $\tilde{\textup{S}}$ and thus, by (\ref{rescaling defition}), Statement 4. follows.

Finally, for Statement $(5)$, we will only prove that $\partial_{v_{0}}y_{t}\geq 0$ since the proof of the fact that $\partial_{a_{0}}x_{t,v_{0}}\geq  0$ follows using a similar argument. We have that
\begin{equation*}
\left\{\begin{aligned}
\partial_{t}\partial_{v_{0}}V(a_{0},v_{0},t)&=-\partial_{V}\big(\Theta_{\epsilon}(V)V\big)\partial_{v_{0}}V; \\
 \partial_{v_{0}}V(a_{0},v_{0},t)&=1.
   \end{aligned}\right.
   \end{equation*}
   Thus 
   \begin{align*}
       \partial_{v_{0}}V(a_{0},v_{0},t)=\textup{e}^{-\int_{0}^{t}\partial_{V}\big(\Theta_{\epsilon}(V(s))V(s)\big)\der s}\geq 0.
   \end{align*}
\end{proof}
\begin{defi}\label{observationofthedomain3}
We define the pair $(l_{1},l_{2})\in\tilde{\textup{S}}$, with $\tilde{\textup{S}}$ defined in (\ref{set iso}), to be the solution of
\begin{equation}
\left\{\begin{aligned}
&y_{t}(l_{2}(v_{0},v'_{0},t)) &=&y_{t}(v_{0})+y_{t}(v'_{0}); \\
&x_{t,l_{2}(v_{0},v'_{0},t)}(l_{1}(a_{0},a'_{0},l_{2}(v_{0},v'_{0},t),t)) & =&x_{t,v_{0}}(a_{0})+x_{t,v'_{0}}(a'_{0}),
   \end{aligned}\right.
   \end{equation}
where $(a_{0},v_{0})\in \tilde{\textup{S}}$, $(a'_{0},v'_{0})\in \tilde{\textup{S}}$ and $t\in[0,\infty)$.
\end{defi}

We make the following notation: $\phi_{t}^{-1}(A,V):=(x^{-1}_{t,y_{t}^{-1}(V)}(A), y^{-1}_{t}(V))$.

For further use, we define 
\begin{align}\label{controloverfirstlinearterm}
    h_{\epsilon}(v_{0},t)=\int_{0}^{t}\Theta_{\epsilon}(y_{s}(v_{0})) \der s.
\end{align}
 
 Most of the results in this subsection hold true if, instead of working with $K_{\epsilon,R}$ defined in (\ref{truncation_kernel_selfsimilar}), we work with a function $K_{R}:(0,\infty)^{4}\rightarrow[0,\infty)$, which is defined in the following manner:
 \begin{align}
     K_{R}(a,v,a',v')=\min\{K(a,v,a',v'), R\}, \label{truncationkernel}
 \end{align}
 where $K$ is as in (\ref{lower_bound_kernel}). Thus, in order to simplify computations (and without loss of generality), we will interchange between $K_{R}$ and $K_{\epsilon,R}$ throughout this subsection. We will also write $\mathbb{K}_{R}$ instead of $\mathbb{K}_{\epsilon,R}$ (which was defined in (\ref{kernel_term})) when we work with $K_{R}$. The same notation will be used in Subsection \ref{subsection better moment estimates} and Appendix \ref{appendix b}. Notice that we will need to work with $K_{\epsilon,R}$ in Subsection \ref{chapter_two} in order to obtain suitable moment estimates.

We first start by proving there exists $G\in\textup{C}([0,\infty);\mathscr{M}_{+}^{I}(\mathbb{R}_{>0}^{2}))$ which satisfies the following equation:
\begin{align}\label{equation_change_of_var}
   &\partial_{t}\int_{\mathbb{R}^{2}_{>0}}G(A,V,t)\varphi(A,V)\der V\der A=\frac{1-\gamma}{2}\int_{\mathbb{R}^{2}_{>0}}\int_{\mathbb{R}^{2}_{>0}}K_{R}(\phi_{t}(A,V),\phi_{t}(A',V'))\xi_{R}(y_{t}(V)+y_{t}(V'))G(A',V',t)\nonumber\\
   &G(A,V,t)[\varphi(\phi_{t}^{-1}(\phi_{t}(A,V)+\phi_{t}(A',V')))\Phi(V,V',t)-\varphi(A,V)\textup{e}^{h_{\epsilon}(V',t)}-\varphi(A',V')\textup{e}^{h_{\epsilon}(V,t)}]\der V' \der A' \der V\der A, 
\end{align}
for every $\varphi\in\compactfun$ and 
\begin{align*}
   \Phi(V,V',t) :=\textup{e}^{-h_{\epsilon}(l_{2}(V,V',t),t)+h_{\epsilon}(V,t)+h_{\epsilon}(V',t)}.
\end{align*}
Notice that the operator $K_{R}$ on the right-hand side of equation (\ref{equation_change_of_var}) encodes information about the fusion process and not only about coagulation through the function $\phi_{t}(A,V)$ and that (\ref{equation_change_of_var})  is a reformulation of (\ref{regularizedformequation}) using characteristics.

Let $T>0$. Take $M=2\int_{(0,\infty)^{2}}(1+a)g_{\textup{in},R}(a,v) \der v \der a+1$, for some $g_{\textup{in},R}$ as in Proposition \ref{propweaksol}. Choose $\tau<T$ such that 
\begin{align}
(M+2(1-\gamma)M+1)||K_{R}||_{\infty}(\textup{e}^{\tau}-1) & < \frac{1}{2}; \label{eq1} \\
2(1-\gamma)M^{2}||K_{R}||_{\infty}\tau & <1 \textup{   and   } \tau\leq \ln{2}. \label{eq2}
\end{align}

We will use the following auxiliary metric space \begin{align}\label{spaceexistence}
\textup{Y}_{\epsilon,\tau}=&\{G\in\textup{C}([0,\tau];\mathscr{M}_{+,\textup{b}}^{I}(\mathbb{R}^{2}_{>0})): ||G||= \sup_{0\leq t\leq \tau}\bigg| \int_{(0,\infty)^{2}}\textup{e}^{h_{\epsilon}(V,t)}(1+x_{t,V}(A))G(A,V,t) \der V \der  A\bigg| \leq 2M; \nonumber\\ 
& G\Big(\mathbb{R}^{2}_{>0}\setminus[c_{0}\epsilon^{\frac{2}{3}},\infty)\times[\epsilon,2R\textup{e}^{t}),t\Big)=0, \textup{ for every } t\in[0,\tau]\}.
\end{align}

Given $G\in \textup{Y}_{\epsilon,\tau}$, we define $$a[G](A,V,t):=\int_{\{A'\geq V'^{\frac{2}{3}}\}}K_{R}(\phi_{t}(A,V),\phi_{t}(A',V'))\xi_{R}(y_{t}(V)+y_{t}(V'))G(A',V',t)\textup{e}^{h_{\epsilon}(V,t)}\der V' \der A',$$ which is well-defined since $K_{R}$ is bounded. For $g_{\textup{in},R}$ as in Proposition \ref{propweaksol}, we analyse the  properties of the map $J:\textup{Y}_{\epsilon,\tau}\rightarrow \textup{C}([0,\tau];\mathscr{M}(\mathbb{R}_{>0}^{2}))$, defined by 
\begin{align}\label{definitionjf}
\int_{\mathbb{R}^{2}_{>0}}J[G](A,V,t)\varphi(A,V)\der V\der A =J_{1}(G,t)+J_{2}(G,t),
\end{align}
where:
\begin{align*}
    J_{1}(G,t):=\int_{\mathbb{R}^{2}_{>0}}g_{\textup{in},R}(A,V) e^{-\int_{0}^{t}a[G](A,V,\xi)\der\xi}\varphi(A,V)\der V\der A
\end{align*}
and
\begin{align*}
J_{2}(G,t):=&\frac{1-\gamma}{2}\int_{0}^{t}\int_{\mathbb{R}^{2}_{>0}}\int_{\mathbb{R}^{2}_{>0}}e^{-\int_{s}^{t}a[G](A,V,\xi)\der\xi}
K_{R}(\phi_{s}(A,V), \phi_{s}(A',V'))\xi_{R}(y_{s}(V)+y_{s}(V'))\times \\
&G(A,V,s)G(A',V',s)\varphi(\phi_{s}^{-1}(\phi_{s}(A,V)+\phi_{s}(A',V')))    \Phi(V,V',s)\der V'\der A' \der V \der A\der s, 
\end{align*}
for every $\varphi\in\compactfun$.
\

Since $G$ is non-negative, then $J[G]\in\textup{C}([0,\tau];\mathscr{M}_{+}(\mathbb{R}_{>0}^{2}))$. Our plan is to use Banach fixed-point theorem for the map $J$. The reason is that, as explained below, a fixed point of the operator $J$ will give the desired solution. In other words, we use a similar approach as the one used to prove well-posedness for pure coagulation equations with bounded kernels, which has been repeatedly used in literature (see, for example, \cite{book}).
\begin{prop}\label{supp is kept}
Assume $g_{\textup{in},R}$ is as in Proposition \ref{propweaksol} and $K_{R}$ as in (\ref{truncationkernel}). Assume $G\in\textup{Y}_{\epsilon,\tau}$. Then $J[G]$ is supported in the same domain as the measures in $\textup{Y}_{\epsilon,\tau}$, namely $J[G]\in\textup{C}([0,\tau];\mathscr{M}_{+}^{I}(\mathbb{R}_{>0}^{2}))$ and  $J[G]\Big(\mathbb{R}^{2}_{>0}\setminus[c_{0}\epsilon^{\frac{2}{3}},\infty)\times[\epsilon,2R\textup{e}^{t}),t\Big)=0$, for every $t\in[0,\tau]$.
\end{prop}
\begin{proof}
Let $\varphi\in\vanishfun$ such that $\supp\varphi\subseteq \{A<c_{0}V^{\frac{2}{3}}\}$. Since $g_{\textup{in},R}\in\textup{Y}_{\epsilon,\tau}$, then the term $J_{1}(G,t)$ defined in (\ref{definitionjf}) vanishes. Due to the support of $G$, we have $A\geq c_{0}V^{\frac{2}{3}}$ and $A'\geq c_{0}V'^{\frac{2}{3}}$. By Proposition \ref{observationofthedomain2},  Statement (4), $x_{t,V}(A)\geq c_{0}y_{t}(V)^{\frac{2}{3}}$ and $x_{t,V'}(A')\geq c_{0}y_{t}(V')^{\frac{2}{3}}$. This implies:
\begin{align*}
    x_{t,V}(A)+x_{t,V'}(A')\geq c_{0}y_{t}(V)^{\frac{2}{3}}+c_{0}y_{t}(V')^{\frac{2}{3}}\geq c_{0}(y_{t}(V)+y_{t}(V'))^{\frac{2}{3}},
\end{align*}
which  implies that $l_{1}(A,A',l_{2}(V,V',t))\geq c_{0}l_{2}(V,V',t)^{\frac{2}{3}}$ by Proposition \ref{observationofthedomain2} and using the notation from Definition \ref{observationofthedomain3}. Thus the term $J_{2}(G,t)$ in (\ref{definitionjf}) vanishes as the support of $\varphi$ gives that we have to work in the set where $l_{1}(A,A',l_{2}(V,V',t))<c_{0}l_{2}(V,V',t)^{\frac{2}{3}}$.

\ 
For the sets $(0,\infty)^{2}\setminus (0,\infty)\times (0,2R\textup{e}^{t}]$ and $(0,\infty)^{2}\setminus [c_{0}\epsilon^{\frac{2}{3}},\infty)\times [\epsilon,\infty)$, the proof is done in the same manner taking note of the fact that the kernel is chosen to vanish on these sets. For example, if $l_{2}(V,V',t)>2R\textup{e}^{t}$, then $y_{t}(V)+y_{t}(V')\geq 2R$ and $\xi_{R}(y_{t}(V)+y_{t}(V'))$ vanishes on this set. If $l_{2}(V,V',t)<\epsilon$, then $y_{t}(V)+y_{t}(V')<\epsilon$, meaning $V,V'<\epsilon$ and $G$ vanishes on these sets. Lastly, we deal with the set $(0,\infty)^{2}\setminus [c_{0}\epsilon^{\frac{2}{3}},\infty)\times (0,\infty)$ by making use of the isoperimetric inequality $c_{0}l_{2}(V,V',t)^{\frac{2}{3}}\leq x^{-1}_{t,l_{2}(V,V',t)}(x_{t,V}(A)+x_{t,V'}(A'))<c_{0}\epsilon^{\frac{2}{3}}$. Thus, $l_{2}(V,V',t)<\epsilon$ and we proceed as before.
\end{proof}
\begin{prop}\label{contractivemap}
Let $g_{\textup{in},R}$ as stated in Proposition \ref{propweaksol} and $K_{R}$ as in (\ref{truncationkernel}). Assume $F,G\in\textup{Y}_{\epsilon,\tau}$ with $\tau$ as in (\ref{eq1}) and (\ref{eq2}). Then
\begin{enumerate}
\item $||J[G]||\leq 2M$; 
\item The map $J$ is contractive, more explicitly, $||J[F]-J[G]||\leq \frac{1}{2}||F-G||$.
\end{enumerate}
In particular, Proposition \ref{supp is kept} and Proposition \ref{contractivemap} imply that $J[G]\in\textup{Y}_{\epsilon,\tau}$.
\end{prop}
The proof consists in a combination of standard methods used in the study of coagulation equations and some of the properties proven in Proposition \ref{observationofthedomain2}. A detailed proof is given in Appendix \ref{appendix b}.

Later on, we will prove moment estimates for higher order powers of $a$. For this, it is useful to keep in mind that the above computations can be done in a more general case.

 We now use Banach fixed point theorem to conclude that there exists a fixed point in the space $\textup{Y}_{\epsilon,\tau}$ for the map $J$, which we will denote by $G_{\epsilon,R,\delta}$.  We will extend the solution to arbitrary times. To do so, we show that the previous computations can be done if we replace $g_{\textup{in},R}$ with $G_{\epsilon,R,\delta}(\cdot,\cdot,\tau)$ and then use induction.

\begin{prop}\label{extensionsolutionprop}
Let $G_{\epsilon, R,\delta}$ be the found fixed point for (\ref{definitionjf}) up to time $\tau$ defined as in (\ref{eq1}), (\ref{eq2}) and with initial datum $g_{\textup{in},R}$ taken as in Proposition \ref{propweaksol}. For any $T>0$, there exists a unique solution, for which we keep the notation $G_{\epsilon,R,\delta}\in\textup{C}^{1}([0,T];\mathscr{M}_{+,\textup{b}}^{I}(\mathbb{R}^{2}_{>0}))$ that satisfies (\ref{equation_change_of_var}).
\end{prop}
As before, the proof consists in a combination of standard methods used in the study of coagulation equations and some of the properties proven in Proposition \ref{observationofthedomain2}. A detailed proof of this proposition is given in Appendix \ref{appendix b}.

\addtocontents{toc}{\protect\setcounter{tocdepth}{2}}
\subsubsection{Passage to the initial equation and properties of the semigroup}\label{subsubsection existence}
\begin{defi} \label{definitionchangeofvar}
Let $T>0$. Let $G_{\epsilon,R,\delta}$ be as in Proposition \ref{extensionsolutionprop}, that is $G_{\epsilon,R,\delta}\in \textup{C}^{1}([0,T];\mathscr{M}_{+}^{I}(\mathbb{R}^{2}_{>0}))$ with $G_{\epsilon,R,\delta}\Big(\mathbb{R}^{2}_{>0}\setminus[c_{0}\epsilon^{\frac{2}{3}},\infty)\times[\epsilon,2R\textup{e}^{t}),t\Big)=0$, for every $t\in[0,T]$. We define $g_{\epsilon,R,\delta}\in\fintimesolradondif$ in the following manner:
\begin{align}\label{changeofvarequation}
    \int_{(0,\infty)^{2}}g_{\epsilon,R,\delta}(a,v,t)\varphi(a,v)\der v \der a=\int_{(0,\infty)^{2}}\textup{e}^{h_{\epsilon}(V,t)}G_{\epsilon,R,\delta}(A,V,t)\varphi(\phi_{t}(A,V))\der V\der A,
\end{align}
for every $\varphi \in \vanishfun$ and every $t\in[0,T]$.
\end{defi}
\begin{proof}[Proof of Proposition \ref{propweaksol}]
Let $g_{\epsilon,R,\delta}$ as in Definition \ref{definitionchangeofvar}. We will prove that $g_{\epsilon,R,\delta}\in \textup{C}^{1}([0,T];\mathscr{M}_{+}^{I}(\mathbb{R}^{2}_{>0}))$ and $g_{\epsilon,R,\delta}$ satisfies equation (\ref{regularizedformequation}) with coagulation kernel $K_{R}$ and initial value $g_{\textup{in},R}$. Moreover, $g_{\epsilon,R,\delta}\Big(\mathbb{R}^{2}_{>0}\setminus[c_{0}\epsilon^{\frac{2}{3}},\infty)\times[\epsilon,2R),t\Big)=0$, for every $t\in[0,T]$.

Firstly, we prove that $g_{\epsilon,R,\delta}$ has the stated support. Assume $\varphi\in\vanishfun$ is supported in $\{a<c_{0}v^{\frac{2}{3}}\}$. The right-hand side of (\ref{changeofvarequation}) implies $x_{t,V}(A)<c_{0}y_{t}(V)^{\frac{2}{3}}$. By Proposition \ref{observationofthedomain2}, we have $A<c_{0}V^{\frac{2}{3}}$. Since $G_{\epsilon,R,\delta}(t)\in\mathscr{M}^{I}_{+}(\mathbb{R}^{2}_{>0})$, then: 
\begin{align*}
    \int_{(0,\infty)^{2}}g_{\epsilon,R,\delta}(a,v,t)\varphi(a,v)\der v \der a=0.
\end{align*}
A similar argument can be used to prove $g_{\epsilon,R,\delta}\Big(\mathbb{R}^{2}_{>0}\setminus[c_{0}\epsilon^{\frac{2}{3}},\infty)\times[\epsilon,2R),t\Big)=0$, for every $t\in[0,T]$.

We now prove that $g_{\epsilon,R,\delta}$ satisfies (\ref{regularizedformequation}). We make use of (\ref{changeofvartable}) and (\ref{changeofvarequation}).
\begin{align*}
\frac{\der}{\der t}\int_{(0,\infty)^{2}}g_{\epsilon,R,\delta}(a,v,t)\varphi(a,v,t)\der v \der a=\frac{\der}{\der t}\int_{(0,\infty)^{2}}\int_{(0,\infty)^{2}}\textup{e}^{h_{\epsilon}(V,t)}G_{\epsilon,R,\delta}(A,V,t)\varphi(\phi_{t}(A,V))\der V\der A\\
=\int_{(0,\infty)^{2}}\frac{\der}{\der t}\bigg(\textup{e}^{h_{\epsilon}(V,t)}\varphi(\phi_{t}(A,V))\bigg)G_{\epsilon,R,\delta}(A,V,t)\der V\der A+\int_{(0,\infty)^{2}}\textup{e}^{h_{\epsilon}(V,t)}\varphi(\phi_{t}(A,V))\frac{\der}{\der t}G_{\epsilon,R,\delta}(A,V,t)\der V\der A.
\end{align*}
The first term becomes:
\begin{align*}
\int_{(0,\infty)^{2}}\frac{\der}{\der t}&\bigg(\textup{e}^{h_{\epsilon}(V,t)}\varphi(\phi_{t}(A,V))\bigg)G_{\epsilon,R,\delta}(A,V,t)\der V\der A \\
&=\int_{(0,\infty)^{2}}\bigg((1-\gamma)r_{\delta}(\phi_{t}(A,V))(c_{0}y_{t}(V)^{\frac{2}{3}}-x_{t,V}(A))\partial_{1}\varphi(\phi_{t}(A,V))\\
&-\frac{2}{3}\Theta_{\epsilon}(y_{t}(V))x_{t,V}(A)\partial_{1}\varphi(\phi_{t}(A,V))-\Theta_{\epsilon}(y_{t}(V))y_{t}(V)\partial_{2}\varphi(\phi_{t}(A,V))\\&+\Theta_{\epsilon}(y_{t}(V))\varphi(\phi_{t}(A,V))\bigg)\textup{e}^{h_{\epsilon}(V,t)}G_{\epsilon,R,\delta}(A,V,t)\der V\der A,
\end{align*}
where we used (\ref{controloverfirstlinearterm}). For the second term, we have:
\begin{align*}
    \int_{(0,\infty)^{2}}&\textup{e}^{h_{\epsilon}(V,t)}\varphi(\phi_{t}(A,V))\frac{\der}{\der t}G_{\epsilon,R,\delta}(A,V,t)\der V\der A=  \\ &\frac{1-\gamma}{2}\int_{(0,\infty)^{2}}\int_{(0,\infty)^{2}}K_{R}(\phi_{t}(A,V),\phi_{t}(A',V'))\xi_{R}(y_{t}(V)+y_{t}(V'))G_{\epsilon,R,\delta}(A',V',t)G_{\epsilon,R,\delta}(A,V,t) \\ 
  & \textup{e}^{h_{\epsilon}(V,t)+h_{\epsilon}(V',t)}[\varphi(\phi_{t}(A,V)+\phi_{t}(A',V'))-\varphi(\phi_{t}(A,V))-\varphi(\phi_{t}(A',V'))]\der V'\der A'\der V\der A. 
\end{align*}
Using now the definition in (\ref{changeofvarequation}), we see that $g_{\epsilon,R,\delta}$ satisfies (\ref{regularizedformequation}).
\end{proof}

We now focus on proving the continuity in the weak topology of the semigroup defined in (\ref{definition semigroup}). This will be useful in order to show that there exists a fixed point in time for equation (\ref{regularizedformequation}). In order to prove Propositions \ref{continuoussemigroup} - \ref{continuous semigroup 2}, we need better regularity for the coagulation kernel. We solve the adjoint problem for a mollified version of the coagulation kernel. A similar approach can be found in \cite{Niethammer_2012}. We then show that the difference between the terms containing the two coagulation kernels can be made small due to the uniform estimates for the total surface area.

We first define the rectangles
\begin{align*}
    V_{\epsilon,R}:=(c_{0}\bigg(\frac{\epsilon}{2}\bigg)^{\frac{2}{3}},\infty)\times(\frac{\epsilon}{2},4R)\subset\mathbb{R}_{>0}^{2} \textup{ and }
    \tilde{V}_{\epsilon,R}:=[c_{0}\bigg(\frac{\epsilon}{4}\bigg)^{\frac{2}{3}},\infty)\times[\frac{\epsilon}{4},8R]\subset\mathbb{R}_{>0}^{2}.
\end{align*}
\begin{ass}\label{assumption kernel}
Let then $K^{n}_{\epsilon,R}\in \textup{C}^{1}(V_{\epsilon,R}\times V_{\epsilon,R})$ be a mollified version of $K_{\epsilon,R}\mathbbm{1}_{\tilde{V}_{\epsilon,R}\times \tilde{V}_{\epsilon,R}}(\eta,\eta')$ chosen in such a way that 
\begin{align*}
\sup_{(\eta,\eta')\in K }|K_{\epsilon,R}(\eta,\eta')-K^{n}_{\epsilon,R}(\eta,\eta')|\leq \frac{1}{n},
\end{align*}
for some $n\in\mathbb{N}$ sufficiently large, to be fixed later, and some fixed compact set $K\subset V_{\epsilon,R}\times V_{\epsilon,R}$.
\end{ass}

\begin{prop}\label{continuoussemigroup}
Let $K^{n}_{\epsilon,R}$ be as in Assumption \ref{assumption kernel}. Let $g_{1},g_{2}$ be two solutions of (\ref{regularizedformequation}) with initial values $g_{\textup{in},1},g_{\textup{in},2}$, respectively. Assume both initial conditions satisfy the assumptions in Proposition \ref{propweaksol}. Let $T>0$. We work with functions on the space
\begin{align*}
    \textup{W}_{T}:=\{\varphi\in\textup{C}^{1}([0,T],\textup{C}_{\textup{b}}^{1}(\tilde{\textup{S}}))\textup{ }|\textup{ }\varphi(\eta,t)=0 \textup{ if } v\not\in [\frac{\epsilon}{2},4R],\textup{ for every } t\leq T\},
\end{align*}
where $\tilde{\textup{S}}$ was defined in (\ref{set iso}). Let $\chi(\eta)$ be an arbitrary function in $\textup{C}_{\textup{b}}^{1}(\tilde{S})$ that is zero when $v\not\in [\frac{\epsilon}{2},4R]$. Then, for every $T>0,$ there exists a unique solution $\varphi\in \textup{W}_{T}$, with $\varphi(\cdot,T)=\chi(\cdot)$, which solves the following equation: 
\begin{align*}
    \partial_{t}\varphi(\eta,t)+\Theta_{\epsilon}(v)\bigg(\varphi(\eta,t)-v\partial_{v}\varphi(\eta,t)-\frac{2}{3}a\partial_{a}\varphi(\eta,t)\bigg)+(1-\gamma)r_{\delta}(a,v)(c_{0}v^{\frac{2}{3}}-a)\partial_{a}\varphi(\eta,t)+\mathbb{L}(\varphi)(\eta,t)=0,
\end{align*}
where
\begin{align*}
    \mathbb{L}(\varphi(\eta,t))(\eta,t):=\frac{1-\gamma}{2}\int_{(0,\infty)^{2}}K^{n}_{\epsilon,R}(\eta,\eta')\xi_{R}(v+v')\chi_{\varphi}(\eta,\eta',t)(g_{1}(\eta',t)+g_{2}(\eta',t))\der\eta'.
\end{align*}
\end{prop}
\begin{proof}
Let us start by integrating along the characteristics. This means that it is enough to prove that there exists a function $\tilde{\varphi}$ that satisfies the equation:
\begin{align*}
    \partial_{t}\tilde{\varphi}(\eta,t)+\Theta_{\epsilon}(y_{t}(v))\tilde{\varphi}(\eta,t)+\mathbb{L}(\tilde{\varphi}(\eta,t))(\phi_{t}(\eta),t)=0,
\end{align*}
where $\phi_{t}(\eta)$ is defined as in (\ref{changeofvartable}).

 We consider a modified version of the operator $\mathbb{L}(\varphi)$, which is possible due to the way the kernel was truncated, as well as the support of $g_{1},g_{2}$. Let $\chi_{\epsilon,R}(v)$ be a continuous function which is equal to zero when $v\leq \frac{\epsilon}{2}$ and when $v\geq 4R$ and is equal to $1$, when $v\in[\epsilon,2R)$. Instead of working with $\mathbb{L}(\varphi),$ we work with:
\begin{align*}
   \overline{\mathbb{L}}(\varphi(\eta,t))(\eta,t):= \chi_{\epsilon,R}(v) \mathbb{L}(\varphi(\eta,t))(\eta,t).
\end{align*}
We observe that, if $\varphi$ is continuous and compactly supported in the $v$ variable, then $\overline{\mathbb{L}}(\varphi)(\eta,t)$ is continuous and compactly supported in the $v$ variable.
We emphasize the fact that, since $g_{1}$ and $g_{2}$ are supported in the region where $\{a\geq c_{0}v^{\frac{2}{3}}\}$, the space of functions in $\textup{C}([0,T],\textup{C}_{\textup{b}}^{1}(\tilde{\textup{S}}))$ will be enough to prove continuity of the semigroup in the weak-$^\ast$ topology later on. We thus prove that there exists a $\tilde{\varphi}$ that satisfies the equation
\begin{align}\label{weakstarcontinuity}
    \partial_{t}\tilde{\varphi}(\eta,t)+\Theta_{\epsilon}(y_{t}(v))\tilde{\varphi}(\eta,t)+\overline{\mathbb{L}}(\tilde{\varphi}(\eta,t))(\phi_{t}(\eta),t)=0.
\end{align}
This will impose no problems when we prove the continuity in the weak-$^\ast$ topology of the semigroup, since $g_{1}$ and $g_{2}$ are supported outside the region where $\overline{\mathbb{L}}(\varphi)\neq \mathbb{L}(\varphi)$.  We prove the existence of a $\tilde{\varphi}$ that satisfies (\ref{weakstarcontinuity}) via a fixed-point argument in the space
\begin{align*}
    \tilde{W}_{T}:=\{\tilde{\varphi}\in\textup{C}^{1}([0,T],\textup{C}^{1}_{\textup{b}}(\tilde{\textup{S}}))|\tilde{\varphi}(\eta,t)=0 \textup{ if } v\leq \frac{\epsilon}{2} \textup{ or } v\geq 4R\textup{e}^{t}, \textup{ for every } t\leq T\}.
\end{align*} 
We use the fact that the kernel is bounded and that 
\begin{align*}
    \int_{(0,\infty)^{2}}[g_{1}(\eta',t)+g_{2}(\eta',t)]\der \eta'\leq  \int_{(0,\infty)^{2}}[g_{\textup{in},1}(\eta')+g_{\textup{in},2}(\eta')]\der \eta'
\end{align*}
in order to prove contractivity. To prove that $\varphi(\cdot,t)\in\textup{C}_{\textup{b}}^{1}(\tilde{\textup{S}}),$ for $t\in[0,T]$, we first prove that there exists a constant $C(t)>0$ such that
\begin{align}\label{bounds for derivative ode}
C(t)\leq \partial_{a}x_{t,v}(a)\leq 1.
\end{align}
This is since
\begin{align}\label{derivative1}
0\leq \partial_{a}[r_{\delta}(a,v)(a-c_{0}v^{\frac{2}{3}})]&=\partial_{a}r_{\delta}(a,v)(a-c_{0}v^{\frac{2}{3}})+r_{\delta}(a,v).
\end{align}
By (\ref{fusion_form}) and (\ref{ode_fusion}), there exists a constant $C_{\epsilon,R,\delta}>0$ such that
\begin{align}
\partial_{a}r_{\delta}(a,v)(a-c_{0}v^{\frac{2}{3}})&=\Big[\frac{\partial_{a}r(a,v)}{1+\delta a^{\mu}}-\frac{r(a,v)\frac{\mu \delta a^{\mu-1}}{1+\delta a^{\mu}}}{1+\delta a^{\mu}}\Big]\frac{\max\{v^{\sigma},L\delta\}}{v^{\sigma}}(a-c_{0}v^{\frac{2}{3}})\nonumber\\
&\leq\frac{\partial_{a}r(a,v)}{1+\delta a^{\mu}}\frac{\max\{v^{\sigma},L\delta\}}{v^{\sigma}}(a-c_{0}v^{\frac{2}{3}})\nonumber\\
&\leq \Big[\frac{Ba^{-1}r(a,v)}{1+\delta a^{\mu}}(a-c_{0}v^{\frac{2}{3}})\Big]\frac{\max\{v^{\sigma},L\delta\}}{v^{\sigma}}
\nonumber\\
&\leq \Big[\frac{BR_{1}a^{\mu-1}v^{\sigma}}{1+\delta a^{\mu}}(a-c_{0}v^{\frac{2}{3}})\Big]\frac{\max\{v^{\sigma},L\delta\}}{v^{\sigma}}\nonumber\\
&\leq \frac{B R_{1}a^{\mu}v^{\sigma}}{1+\delta a^{\mu}}\frac{\max\{v^{\sigma},L\delta\}}{v^{\sigma}}\leq \frac{BR_{1}}{\delta }\max\{v^{\sigma},L\delta\}\leq C_{\epsilon,R,\delta}\label{derivative2}
\end{align}
and
\begin{align}\label{derivative3}
r_{\delta}(a,v)&=\frac{r(a,v)}{1+\delta a^{\mu}}\frac{\max\{v^{\sigma},L\delta\}}{v^{\sigma}}\leq \frac{R_{1}a^{\mu}v^{\sigma}}{1+\delta a^{\mu}}\frac{\max\{v^{\sigma},L\delta\}}{v^{\sigma}}\leq C_{\epsilon,R,\delta}.
\end{align}
Combining (\ref{derivative1})-(\ref{derivative3}) and making use of (\ref{changeofvartable}), we obtain (\ref{bounds for derivative ode}). Using (\ref{weakstarcontinuity}) and (\ref{bounds for derivative ode}), we prove the desired bound for the derivative of $\varphi$.
\end{proof}
\begin{rmk}
Let $g_{1},g_{2}$ be two solutions of (\ref{regularizedformequation}) with initial values $g_{\textup{in},1},g_{\textup{in},2}\in U_{\epsilon,R}$, respectively. Let $T>0$. Let $\varphi\in \textup{C}^{1}([0,T],\textup{C}^{1}(\tilde{\textup{S}}))$ such that $\varphi(\eta,T)=\chi(\eta)$ be the function found in Proposition \ref{continuoussemigroup}. Assume, in addition, that $\sup_{\eta\in \tilde{\textup{S}}}|\chi(\eta)|\leq 1$. Then there exists a constant $C_{\epsilon,R}(T), $ which is independent of the choice of $n\in\mathbb{N}$ in Assumption \ref{assumption kernel} and depends only on the norm of $g_{1}$ and $g_{2},$ $T, \epsilon$ and $R$, such that
\begin{align}\label{upper bound independent n}
    \sup_{s\in[0,T]}\sup_{\eta\in\tilde{\textup{S}}}|\varphi(\eta, s)|\leq C_{\epsilon,R}(T).
\end{align}
\end{rmk}
We now prove that the solution $\varphi$ that we have found in Proposition \ref{continuoussemigroup} is Lipschitz continuous. This will provide a suitable compactness property that will allow us to prove the desired weak continuity of the mapping $S(t):U_{\epsilon,R}\rightarrow U_{\epsilon,R}$ defined in (\ref{definition semigroup}).
\begin{prop}\label{dual function lipschitz}
Let $g_{1},g_{2}$ be two solutions of (\ref{regularizedformequation}) with initial values $g_{\textup{in},1},g_{\textup{in},2}\in U_{\epsilon,R}$, respectively. Let $T>0$. Let $\varphi\in \textup{C}^{1}([0,T],\textup{C}^{1}(\tilde{\textup{S}}))$ such that $\varphi(\eta,T)=\chi(\eta)$ be the function found in Proposition \ref{continuoussemigroup}. Assume, in addition, that $\sup_{\eta\in \tilde{\textup{S}}}|\chi(\eta)|\leq 1$ and that $\chi(\eta)$ is Lipschitz. Then $\varphi$ is Lipschitz continuous, in the sense that, for every $t\in[0,T],$ there exists $C(t)>0$ such that 
\begin{align*}
    \sup_{s\in[0,t]}|\varphi(\eta,s)-\varphi(\tilde{\eta},s)|\leq C(t)|\eta-\tilde{\eta}|,
\end{align*}
for every $\eta,\tilde{\eta}\in\tilde{\textup{S}}$. Moreover, $C(t)$ may depend on the norm of $g_{1}$ and $g_{2}$, but is otherwise independent of the choice of $g_{1}$ and $g_{2}$.
\end{prop}
\begin{proof}
Notice first that, since $\sup_{\eta\in \tilde{\textup{S}}}|\chi(\eta)|\leq 1$, then $\sup_{s\in[0,t],\eta\in \tilde{\textup{S}}}|\varphi(\eta,s)|\leq C_{\epsilon,R}(T)$ by (\ref{upper bound independent n}). We use Gr\"{o}nwall in (\ref{weakstarcontinuity}):
\begin{align*}
 |\tilde{\varphi}(\eta,t)-\tilde{\varphi}(\tilde{\eta},t)|&\leq |\chi(\eta)-\chi(\tilde{\eta})|+\sup_{z\in[0,t],\eta\in\tilde{\textup{S}}}|\tilde{\varphi}(\eta,z)|\int_{t}^{T}|\Theta_{\epsilon}(y_{s}(v))-\Theta_{\epsilon}(y_{s}(\tilde{v}))|\der s\\
 &+\int_{t}^{T}|\tilde{\varphi}(\eta,s)-\tilde{\varphi}(\tilde{\eta},s)|\der s+\int_{t}^{T}|\overline{\mathbb{L}}(\tilde{\varphi}(\eta,s)-\tilde{\varphi}(\tilde{\eta},s))(\phi_{s}(\eta),s)|\der s\\
 &+\int_{t}^{T}|\overline{\mathbb{L}}(\tilde{\varphi}(\eta,s))(\phi_{s}(\eta)-\phi_{s}(\tilde{\eta}),s)|\der s.
\end{align*}
By the definition of $K^{n}_{\epsilon,R}$ in Assumption \ref{assumption kernel}, we have its first order derivatives are bounded from above. Moreover, $\phi_{s}(\eta)$ is Lipschitz continuous. Thus, there exists a constant $C>0,$ which can depend on $\epsilon,$ $R,$ $\delta$ and the norms of $g_{1}$ and $g_{2}$ such that 
\begin{align*}
    |\tilde{\varphi}(\eta,t)-\tilde{\varphi}(\tilde{\eta},t)|\leq C\int_{t}^{T}    |\tilde{\varphi}(\eta,s)-\tilde{\varphi}(\tilde{\eta},s)|+C |\eta-\tilde{\eta}|.
\end{align*}
We use Gr\"{o}nwall and obtain $|\tilde{\varphi}(\eta,t)-\tilde{\varphi}(\tilde{\eta},t)|\leq C(t) |\eta-\tilde{\eta}|.$ Combining this with (\ref{bounds for derivative ode}), we conclude that $|\varphi(\eta,t)-\varphi(\tilde{\eta},t)|\leq C(t) |\eta-\tilde{\eta}|$.
\end{proof}
This is enough to enable us to prove that the semigroup defined in (\ref{definition semigroup}) is continuous in the weak-$^\ast$ topology.
\begin{prop}\label{continuous semigroup 2}
Let $K^{n}_{\epsilon,R}$ be as in Assumption \ref{assumption kernel}. Let $t\geq 0$ and $S(t):U_{\epsilon,R}\rightarrow U_{\epsilon,R}$ be the mapping defined in (\ref{definition semigroup}). Then, we have that $S(t)$ is continuous in the weak-$^\ast$ topology.
\end{prop}
\begin{proof}
As before, assume that $g_{1},g_{2}$ are two solutions of (\ref{regularizedformequation}) with initial conditions $g_{\textup{in},1},g_{\textup{in},2}\in U_{\epsilon,R}$, respectively. Assume furthermore that there exists a constant $C>0$ such that
\begin{align*}
    \int_{(0,\infty)^{2}}(1+a)(g_{\textup{in},1}(\eta)+g_{\textup{in},2}(\eta))\der \eta\leq C.
\end{align*}
Let $\phi\in\textup{C}^{1}_{0}(\mathbb{R}_{>0}^{2})$ with $\phi(\eta)=0 \textup{ if } (a,v)\not\in [c_{0}\big(\frac{\epsilon}{2}\big)^{\frac{2}{3}},L]\times[\frac{\epsilon}{2},4R]$, where $\epsilon,R$ are as in Proposition \ref{propweaksol} and for some $L>0$. Assume in addition that $||\phi||_{\infty}\leq 1$. Let $\tilde{\delta}<1$. Our goal is to prove that, if for every $\chi\in\textup{C}_{0}(\mathbb{R}_{>0}^{2}),$ with $||\chi||_{\infty}\leq 1,$
\begin{align*}
    \int_{(0,\infty)^{2}}[g_{\textup{in},2}(\eta)-g_{\textup{in},1}(\eta)]\chi(\eta)\der\eta\textup{ is sufficiently small, }
\end{align*}
 then for every $\chi\in\textup{C}_{0}(\mathbb{R}_{>0}^{2}),$ with $||\chi||_{\infty}\leq 1,$
 \begin{align*}
    \int_{(0,\infty)^{2}}[g_{2}(\eta,t)-g_{1}(\eta,t)]\chi(\eta)\der\eta<\tilde{\delta}.
\end{align*}
We make the following notation:
\begin{align*}
   C_{\textup{norm}}:=\sup_{s\in[0,t]} \int_{(0,\infty)^{2}}(1+a)(g_{1}(\eta,s)+g_{2}(\eta,s))\der \eta<\infty.
\end{align*}

 Notice that $C_{\textup{norm}}$ may depend on the time $t$, but is independent of the choice of $g_{1}(\eta,s), g_{2}(\eta,s)$, for $s\in[0,t]$. This is since we can bound 
 \begin{align*}
     \sup_{s\in[0,t]} \int_{(0,\infty)^{2}}(1+a)(g_{1}(\eta,s)+g_{2}(\eta,s))\der \eta\leq C(t)\int_{(0,\infty)^{2}}(1+a)(g_{\textup{in},1}(\eta)+g_{\textup{in},2}(\eta))\der \eta
 \end{align*}
by using similar arguments to the ones used in Proposition \ref{extensionsolutionprop}.

We fix $n\in\mathbb{N}$ in Assumption \ref{assumption kernel} such that $\frac{3tC_{\epsilon,R}(t)C_{\textup{norm}}^{2}}{n}\leq\frac{\tilde{\delta}}{4},$ where $C_{\epsilon,R}(t)$ is defined as in (\ref{upper bound independent n}).

Let $\varphi$ be the function found in Proposition \ref{continuoussemigroup} associated to the coagulation kernel $K_{\epsilon,R}^{n}$, for $n\in\mathbb{N}$ as above, and with $\varphi(\eta,t)=\phi(\eta).$ Since $\phi\in\textup{C}^{1}_{\textup{c}}(\mathbb{R}_{>0}^{2})$, then by Proposition \ref{dual function lipschitz} the function $\varphi$ is Lipschitz continuous. Then there exists $\tilde{C}(\epsilon,R,\delta,t,C_{\textup{norm}})>0$ such that $|\varphi(\eta,s)-\varphi(\tilde{\eta},s)|\leq \tilde{C}(\epsilon,R,\delta,t,C_{\textup{norm}})|\eta-\tilde{\eta}|,$ for every $\eta,\tilde{\eta}\in\tilde{\textup{S}}$ and $s\in[0,t]$.

Assume $M>1$ is a fixed constant whose value will be determined later. We take $\chi_{M}:\mathbb{R}_{>0}\rightarrow [0,1]$, continuous, to be
\begin{align}\label{cut near infinity}
\chi_{M}(a)=
\begin{cases}
1, &\text{ when } a\in[c_{0}\epsilon^{\frac{2}{3}},M],\\ 
0, &\text{ when } a\not\in[c_{0}(\frac{\epsilon}{2})^{\frac{2}{3}},2M].
\end{cases}
\end{align}

Then there exists $C(\epsilon,R,\delta,t,C_{\textup{norm}})>0$ such that $|\varphi(\eta,s)\chi_{M}(a)-\varphi(\tilde{\eta},s)\chi_{M}(\tilde{a})|\leq$ $ C(\epsilon,R,\delta,t,C_{\textup{norm}})$ $|\eta-\tilde{\eta}|,$ for every $\eta,\tilde{\eta}\in\tilde{\textup{S}}$ and $s\in[0,t]$.
 
 We look at the set
 \begin{align*}
K=\{\tilde{\varphi}\in\textup{C}(\tilde{S})|&\text{ }\tilde{\varphi}(\eta)=0 \textup{ if }  \eta\not\in [c_{0}(\frac{\epsilon}{2})^{\frac{2}{3}},2M]\times[\frac{\epsilon}{2},4R];\text{ } |\tilde{\varphi}(\eta)-\tilde{\varphi}(\eta')|\leq C(\epsilon,R,\delta,t,C_{\textup{norm}})|\eta-\eta'|,\\
&\textup{ for all } \eta,\eta'\in\tilde{S}\}.
 \end{align*}
As the set is totally bounded, there exist $N\in\mathbb{N}$ and  $\psi_{1},\ldots,\psi_{N}\in K$ with the property that $K\subseteq\cup_{i=1}^{N} B(\psi_{i},\frac{\tilde{\delta}}{4C_{\textup{norm}}})$. As $\varphi(\cdot,0)\chi_{M}\in K,$ then $\min_{i=\overline{1,N}}\sup_{\eta\in \tilde{S}}|\varphi(\eta,0)\chi_{M}(a)-\psi_{i}(\eta)|\leq \frac{\tilde{\delta}}{4C_{\textup{norm}}}.$
Assume that, for every $\chi\in\textup{C}_{0}(\mathbb{R}_{>0}^{2}),$ with $||\chi||_{\infty}\leq 1,$
\begin{align*}
    \int_{(0,\infty)^{2}}[g_{\textup{in},2}(\eta)-g_{\textup{in},1}(\eta)]\chi(\eta)\der\eta<\frac{\tilde{\delta}}{8N}.
\end{align*}
We then have
\begin{align}\label{duality argument}
\int_{(0,\infty)^{2}}&[g_{2}-g_{1}](\eta,t)\phi(\eta)\der \eta=\int_{(0,\infty)^{2}}[g_{2}-g_{1}](\eta,t)\varphi(\eta,t)\der \eta
\nonumber\\
&=\int_{(0,\infty)^{2}}[g_{2,\textup{in}}-g_{1,\textup{in}}](\eta)\varphi(\eta,0)\der \eta+\frac{1-\gamma}{2}\int_{0}^{t}\int_{(0,\infty)^{2}}\int_{(0,\infty)^{2}}(K_{\epsilon,R}(\eta,\eta')-K_{\epsilon,R}^{n}(\eta,\eta'))\xi_{R}(v+v')\nonumber\\
&\chi_{\varphi}(\eta,\eta',s)(g_{1}(\eta',s)+g_{2}(\eta',s))(g_{1}(\eta',s)-g_{2}(\eta',s))\der \eta'\der \eta\der s\nonumber\\
&\leq\int_{(0,\infty)^{2}}[g_{2,\textup{in}}-g_{1,\textup{in}}](\eta)\varphi(\eta,0)\der \eta+\int_{0}^{t}\int_{V_{\epsilon,R,M}}\int_{V_{\epsilon,R,M}}|(K_{\epsilon,R}(\eta,\eta')-K_{\epsilon,R}^{n}(\eta,\eta'))\xi_{R}(v+v')\nonumber\\
&\chi_{\varphi}(\eta,\eta',s)(g_{1}(\eta',s)+g_{2}(\eta',s))(g_{1}(\eta',s)-g_{2}(\eta',s))|\der \eta'\der \eta\der s\nonumber\\
&+\int_{0}^{t}\int_{(M,\infty)\times[\epsilon,2R]}\int_{(0,\infty)^{2}}|(K_{\epsilon,R}(\eta,\eta')-K_{\epsilon,R}^{n}(\eta,\eta'))\xi_{R}(v+v')\chi_{\varphi}(\eta,\eta',s)(g_{1}(\eta',s)+g_{2}(\eta',s))\nonumber\\
&(g_{1}(\eta',s)-g_{2}(\eta',s))|\der \eta'\der \eta\der s\nonumber\\
&= I_{1}+I_{2}+I_{3},
\end{align}
where $V_{\epsilon,R,M}:=[c_{0}\epsilon^{\frac{2}{3}},M]\times[\epsilon,2R]$, for some sufficiently large $M$ as in (\ref{cut near infinity}).

We bound the term $I_{1}$ in (\ref{duality argument}) by:
\begin{align}\label{first term}
    \int_{(0,\infty)^{2}}&[g_{2,\textup{in}}-g_{1,\textup{in}}](\eta)\varphi(\eta,0)\der \eta\nonumber\\
    &=\int_{(0,\infty)^{2}}[g_{2,\textup{in}}-g_{1,\textup{in}}](\eta)\varphi(\eta,0)\chi_{M}(a)\der \eta+\int_{(0,\infty)^{2}}[g_{2,\textup{in}}-g_{1,\textup{in}}](\eta)\varphi(\eta,0)[1-\chi_{M}(a)]\der \eta\nonumber\\
     &\leq\int_{(0,\infty)^{2}}[g_{2,\textup{in}}-g_{1,\textup{in}}](\eta)\varphi(\eta,0)\chi_{M}(a)\der \eta+2||\varphi(\eta,0)||_{\infty}\int_{a>M}[g_{2,\textup{in}}+g_{1,\textup{in}}](\eta)\der \eta\nonumber\\
    &\leq \min_{i=\overline{1,N}}\int_{(0,\infty)^{2}}|[g_{2,\textup{in}}-g_{1,\textup{in}}](\eta)||\varphi(\eta,0)\chi_{M}(a)-\psi_{i}(\eta)|\der \eta\nonumber\\
    &+\max_{i=\overline{1,N}}\int_{(0,\infty)^{2}}[g_{2,\textup{in}}-g_{1,\textup{in}}](\eta)\psi_{i}(\eta)\der \eta+2M^{-1}C_{\textup{norm}}||\varphi(\eta,0)||_{\infty}.
\end{align}
Notice that $K$ was chosen independently of $g_{1}, g_{2}$ since it depends only on $C(\epsilon,R,\delta,t, C_{\textup{norm}})$. Thus we can bound $\max_{i=\overline{1,N}}\int_{(0,\infty)^{2}}[g_{2,\textup{in}}-g_{1,\textup{in}}](\eta)\psi_{i}(\eta)\der \eta$ in (\ref{first term}) by $\frac{\tilde{\delta}}{8}$ independently of the choice of $g_{1},g_{2}$, as $\psi_{i}\in K,$ $i=\overline{1,N}$. As $||\phi(\eta)||_{\infty}\leq 1,$ then $\sup_{s\in[0,t]}||\varphi(\cdot,s)||_{\infty}\leq C_{\epsilon,R}(t)$. We can then choose $M>\frac{16C_{\textup{norm}}C_{\epsilon,R}(t)}{\tilde{\delta}}$ in order to bound $2M^{-1}C_{\textup{norm}}||\varphi(\eta,0)||_{\infty}$. This means that
\begin{align}\label{step 2}
     \int_{(0,\infty)^{2}}&[g_{2,\textup{in}}-g_{1,\textup{in}}](\eta)\varphi(\eta,0)\der \eta\leq \frac{\tilde{\delta}}{4C_{\textup{norm}}}\int_{(0,\infty)^{2}}[g_{2,\textup{in}}+g_{1,\textup{in}}](\eta)\der \eta+\frac{\tilde{\delta}}{4}\leq \frac{\tilde{\delta}}{2}.
\end{align}

In order to bound  $I_{2}$ in (\ref{duality argument}),
we have that on $([c_{0}\epsilon^{\frac{2}{3}},M]\times[\epsilon,2R])^{2}$
\begin{align*}
   \sup_{(\eta,\eta')\in ([c_{0}\epsilon^{\frac{2}{3}},M]\times[\epsilon,2R])^{2}}|K^{n}_{\epsilon,R}(\eta,\eta')- K_{\epsilon,R}(\eta,\eta')|\leq \frac{1}{n}
\end{align*} 
by Assumption \ref{assumption kernel}. Thus $I_{2}\leq \frac{3t||\varphi||_{\infty}C_{\textup{norm}}^{2}}{n}\leq\frac{\tilde{\delta}}{4}.$

For $I_{3}$ in (\ref{duality argument}), we take $M\geq 6C_{\epsilon,R}(t)C_{\textup{norm}}^{2}t||K_{\epsilon,R}||_{\infty}\frac{4}{\tilde{\delta}}$. This is because $I_{3}$ can be estimated from above by
\begin{align*}
M^{-1}3tC_{\epsilon,R}(t)(||K^{n}_{\epsilon,R}||_{\infty}+||K_{\epsilon,R}||_{\infty})\sup_{s\in[0,t]}\int_{\{a\geq M\}}a(g_{1}(\eta,s)+g_{2}(\eta,s))\der \eta\int_{(0,\infty)^{2}}(g_{1}(\eta,s)+g_{2}(\eta,s))\der \eta'.
\end{align*}

We then combine this with (\ref{step 2}) and our choice of $n\in\mathbb{N}$ in (\ref{duality argument}) in order to conclude the argument.

To extend this argument to all functions $\phi\in\textup{C}_{0}(\mathbb{R}_{>0}^{2})$ we use again that $(M_{0,1}+M_{0,0})(g_{1}(s)+g_{2}(s))<\infty,$ for all $s\in[0,t]$.
\end{proof}

For $g\in U_{\epsilon,R}$, it is worthwhile to observe that the map $S(\cdot)g:[0,T]\rightarrow U_{\epsilon,R}$ defined in (\ref{definition semigroup}) is also continuous in time. 
\begin{prop}\label{continuous time semigroup}
Let $g\in U_{\epsilon,R}$, with $U_{\epsilon,R}$ defined in (\ref{existencespace for f}) and $T>0$. Let $S(\cdot)g:[0,T]\rightarrow U_{\epsilon,R}$ be as in (\ref{definition semigroup}). Then $S(\cdot)g$ is continuous in time. 
\end{prop}

The proof consists of standard methods used in the study of coagulation equations. A proof of it can be found in Appendix \ref{appendix c}.

\subsection{Existence of self-similar profiles}\label{chapter_two}
Assume $\mu>0$. We focus on proving Proposition \ref{lemmaweaksol}. The following moment estimates involving the  $v$ variable (Propositions \ref{firstmoments}-\ref{m-lprop}) are an adaptation to the two-dimensional case of \cite[Lemma 3.3, Lemma 3.4]{dust}. There was an additional need to control the escape to infinity of the area, which has been dealt with by choosing an appropriate norm in $\textup{Y}_{\epsilon,\tau}$ defined in (\ref{spaceexistence}), as well as a specific truncation for the coagulation kernel and fusion term. We then continue with proving an additional uniform moment estimate involving the area which will be needed for removing the truncation in (\ref{regularizedformequation}).
\begin{rmk}\label{kernel_lower_bound_welldef}
Assume $K_{\epsilon,R}$ satisfies  (\ref{truncation_kernel_selfsimilar}). Then $K_{\epsilon,R}(a,v,a',v')\geq K_{1}(v^{\beta}v'^{-\alpha}+v^{-\alpha}v'^{\beta})$ if $v,v'>\epsilon$ and $v+v'<2R$, for all $(a,a')\in(0,\infty)^{2}$.
\end{rmk}
We first mention that due to the choice of $\textup{Y}_{\epsilon,\tau}$ and $r_{\delta}$, we can test (\ref{regularizedformequation}) with functions that do not necessarily have compact support. More precisely, we can test with functions of the form $a^{c}v^{d}$, $d\in\mathbb{R}$, as long as $c\leq 1$.

As previously mentioned, the proof of Propositions \ref{firstmoments}-\ref{m-lprop} is an adaptation of methods used in \cite{dust} to our setting. We thus only state the estimates here and move their proof to Appendix \ref{appendix c}.
\begin{prop} \label{firstmoments}
Let $g_{\textup{in},R}\in \mathscr{M}^{I}_{+}(\mathbb{R}_{>0}^{2})\cap \omega(\epsilon,R,\delta)$. Let $g_{\epsilon,R,\delta}$ be the solution found in Proposition \ref{propweaksol} with kernel $K_{\epsilon,R}$ as in (\ref{truncation_kernel_selfsimilar}). Then 
\begin{align}
 \sup_{t\geq 0}M_{0,1}(g_{\epsilon,R,\delta}(\cdot,\cdot,t))= M_{0,1}(g_{\textup{in},R}). \label{m1}
\end{align}
 Moreover, there exists a constant $C_{0,\gamma}>0$ such that:
\begin{align}
\sup_{t\geq 0}M_{0,\gamma}(g_{\epsilon,R,\delta}(\cdot,\cdot,t))\leq \max\{{M_{0,\gamma}(g_{\textup{in},R});C_{0,\gamma}}\}, \label{mlambda}
\end{align}
uniformly in $\epsilon, R$ and $\delta$.
\end{prop}

\begin{prop}\label{prop m moment}
Let $g_{\textup{in},R}$ and $g_{\epsilon,R,\delta}$ (which we will denote by $g_{\textup{in}}$ and $g$, respectively) be as in Proposition \ref{firstmoments}. Then, for any $m>1$, there exists a constant $C_{0,m}>0$, independent of $\epsilon, R$ and $\delta$, such that:
\begin{align}
\sup_{t\geq 0}M_{0,m}(g_{\epsilon,R,\delta}(\cdot,\cdot,t))\leq \max\{M_{0,m}(g_{\textup{in}}), C_{0,m}\}. \label{mm}
\end{align}
\end{prop}

We can also obtain bounds independent of time for moments of the form $M_{0,-l}$, with $l>0$, if we require in addition that $\alpha>0$. 
\begin{prop}\label{m-lprop}
Let $g_{\textup{in},R}$ and $g_{\epsilon,R,\delta}$ (which we will denote by $g_{\textup{in}}$ and $g$, respectively) be as in Proposition \ref{firstmoments}. Assume $\alpha>0$. Let $l>0$. Then there exists a constant $C_{0,-l}>0$ for which the following estimate holds:
\begin{align}
\sup_{t\geq 0}M_{0,-l}(g_{\epsilon,R,\delta}(t))\leq \max\{M_{0,-l}(g_{\textup{in}}),C_{0,-l}\}, \label{m-l}
\end{align}
uniformly in $\epsilon, R$ and $\delta$.
\end{prop}

We now find uniform estimates for the total surface area of $g_{\epsilon,R,\delta}$. The uniform estimates for the surface area are a consequence of the fusion term overtaking the coagulation operator in the case $\mu>0$. We prove a more general statement which will be needed later on for the improvement of moment estimates.
\begin{prop}\label{avnegativeprop}
 Let $g_{\textup{in},R}$ and $g_{\epsilon,R,\delta}$ (which we will denote by $g_{\textup{in}}$ and $g$, respectively) be as in Proposition \ref{firstmoments}. Then
\begin{align}
\sup_{t\geq 0}M_{1,0}(g(\cdot,\cdot,t))\leq \max\{M_{1,0}(g_{\textup{in}}), C_{1,0}\}.\label{mainarea}
\end{align}
Moreover, if (\ref{alpha non neg}) holds and if $M_{1,-\alpha}(g_{\textup{in}})<\infty$, there exists $C_{1,-\alpha}>0$ such that:
\begin{align}
\sup_{t\geq 0}M_{1,-\alpha}(g(\cdot,\cdot,t))\leq \max\{M_{1,-\alpha}(g_{\textup{in}}), C_{1,-\alpha}\}. \label{avnegative}
\end{align}
If (\ref{alpha zero}) holds, there exists $\delta_{1}>0,$ for which $\gamma\leq \frac{2}{3}-\delta_{1}$, and there exists $C_{1,-\delta_{1}}>0$ such that:
\begin{align}\label{avnegative alpha zero}
\sup_{t\geq 0}M_{1,-\delta_{1}}(g(\cdot,\cdot,t))\leq \max\{M_{1,-\delta_{1}}(g_{\textup{in}}), C_{1,-\delta_{1}}\}. 
\end{align}
\end{prop}
\begin{proof}
We test equation (\ref{regularizedformequation}) with $\varphi(a,v)=av^{l}$, with $l\leq 0$. This is possible due to our choice of the space $\textup{Y}_{\epsilon,\tau}$, which allows us to test with functions that do not have compact support if they are bounded from above by a function of the form $a^{c}v^{d}$, with $c\leq 1$ and $d\in\mathbb{R}$. Since $l\leq0$, we have:
\begin{align*}
  K_{\epsilon,R}(a,v,a',v')[(a+a')(v+v')^{l}-av^{l}-a'v'^{l}]\leq 0.   
\end{align*}
Equation (\ref{regularizedformequation}) thus becomes:
\begin{align}\label{equation thus becomes}
   \frac{\der}{\der t} M_{1,l}(g(t))\leq (\frac{1}{3}+|l|)M_{1,l}(g(t))+(1-\gamma)R_{0}\int_{(0,\infty)^{2}}\frac{a^{\mu}\max\{v^{\sigma}, L\delta\}}{1+\delta a^{\mu}}(c_{0}v^{\frac{2}{3}}-a)g(a,v,t)v^{l}\der v \der a.
\end{align}
We need some control over the drift term:
\begin{align}
    \int_{\{a\geq c_{0}v^{\frac{2}{3}}\}}\frac{a^{\mu}\max\{v^{\sigma}, L\delta\}}{1+\delta a^{\mu}}(c_{0}v^{\frac{2}{3}}-a)g(\eta,t)v^{l}\der \eta &\leq  \int_{\{a\geq 2c_{0}v^{\frac{2}{3}}\}}\frac{a^{\mu}\max\{v^{\sigma}, L\delta\}}{1+\delta a^{\mu}}(c_{0}v^{\frac{2}{3}}-a)g(\eta,t)v^{l}\der \eta \nonumber \\
    &\leq -\frac{1}{2}\int_{\{a\geq 2c_{0}v^{\frac{2}{3}}\}}\frac{a^{\mu}\max\{v^{\sigma}, L\delta\}}{1+\delta a^{\mu}}av^{-|l|}g(\eta,t)\der \eta. \label{we need some control over the drift term}
\end{align}
This integral can now be estimated as follows:
\begin{align}\label{surface area bound 1}
& -\frac{1}{2}\int_{\{a\geq 2c_{0}v^{\frac{2}{3}}, a^{\mu}\geq \frac{1}{\delta}\}}\frac{a^{\mu}\max\{v^{\sigma}, L\delta\}}{1+\delta a^{\mu}}av^{-|l|}g(\eta,t)\der \eta-\frac{1}{2}\int_{\{a\geq 2c_{0}v^{\frac{2}{3}}, a^{\mu}< \frac{1}{\delta}\}}\frac{a^{\mu}\max\{v^{\sigma}, L\delta\}}{1+\delta a^{\mu}}av^{-|l|}g(\eta,t)\der \eta\nonumber\\
 &\leq -\frac{1}{4\delta}\int_{\{a\geq 2c_{0}v^{\frac{2}{3}}, a^{\mu}\geq \frac{1}{\delta}\}}\max\{v^{\sigma}, L\delta\}av^{-|l|}g(a,v,t)\der v \der a-\frac{1}{4}\int_{\{a\geq 2c_{0}v^{\frac{2}{3}}, a^{\mu}< \frac{1}{\delta}\}}a^{\mu+1}v^{\sigma-|l|}g(a,v,t)\der v \der a\nonumber\\
  &\leq -\frac{L}{4}\int_{\{a\geq 2c_{0}v^{\frac{2}{3}}, a^{\mu}\geq \frac{1}{\delta}\}}av^{-|l|}g(a,v,t)\der v \der a-\frac{1}{4}\int_{\{a\geq 2c_{0}v^{\frac{2}{3}}, a^{\mu}< \frac{1}{\delta}\}}a^{\mu+1}v^{\sigma-|l|}g(a,v,t)\der v \der a.
\end{align}
In order to estimate the second term on the right-hand side of (\ref{surface area bound 1}), we apply Young's inequality. This is possible since $\mu+1>1$. Thus, we obtain:
\begin{align*}
    av^{l}=av^{\frac{\sigma-|l|}{\mu+1}}v^{l-\frac{\sigma-|l|}{\mu+1}}\leq\frac{\epsilon}{p}(av^{\frac{\sigma-|l|}{\mu+1}})^{\mu+1}+\frac{\epsilon^{-\frac{q}{p}}}{q}v^{(l-\frac{{\sigma-|l|}}{\mu+1})q},
\end{align*}
where $p=\mu+1$ and $\frac{1}{p}+\frac{1}{q}=1$. Thus, there exists a constant $C_{\epsilon}>0$ depending on $\epsilon$
for which $av^{l}\lesssim \epsilon a^{\mu+1}v^{\sigma-|l|}+\epsilon C_{\epsilon}v^{l-\frac{\sigma}{\mu}}$, or equivalently:
\begin{align*}
-a^{\mu+1}v^{\sigma-|l|} \lesssim -\frac{1}{\epsilon}  av^{l}+C_{\epsilon}v^{l-\frac{\sigma}{\mu}}
\end{align*}
and this implies:
\begin{align}\label{surface area bound 2}
    -\int_{\{a\geq 2c_{0}v^{\frac{2}{3}}, a^{\mu}<\frac{1}{\delta}\}}a^{\mu+1}v^{\sigma-|l|}g(a,v,t)\der v \der a\lesssim -\frac{1}{\epsilon}\int_{\{a\geq 2c_{0}v^{\frac{2}{3}},a^{\mu}<\frac{1}{\delta}\}}av^{l}g(a,v,t)\der v \der a + C_{\epsilon}M_{0,l-\frac{\sigma}{\mu}}(g(t)).
\end{align}
From Propositions \ref{firstmoments}-\ref{m-lprop}, we have that $M_{0,l-\frac{\sigma}{\mu}}(g(t))$ is uniformly bounded from above. We combine (\ref{we need some control over the drift term}), (\ref{surface area bound 1}) and (\ref{surface area bound 2}) and then we choose $\epsilon$ sufficiently small in order to obtain that
\begin{align}\label{surface area bound added}
     \int_{\{a\geq c_{0}v^{\frac{2}{3}}\}}\frac{a^{\mu}\max\{v^{\sigma}, L\delta\}}{1+\delta a^{\mu}}(c_{0}v^{\frac{2}{3}}-a)g(\eta,t)v^{l}\der \eta \lesssim -\frac{L}{4}\int_{\{a\geq 2c_{0}v^{\frac{2}{3}}\}}av^{-|l|}g(a,v,t)\der v \der a+C_{\epsilon}.
\end{align}
Moreover, in order to pass back to the region $\{a\geq c_{0}v^{\frac{2}{3}}\}$ in the integral term on the right-hand side of (\ref{surface area bound added}), we notice that:
\begin{align}\label{surface area bound 3}
   -\int_{\{a\geq 2c_{0}v^{\frac{2}{3}}\}}av^{l}g(a,v,t)\der v \der a&= -\int_{\{a\geq c_{0}v^{\frac{2}{3}}\}}av^{l}g(a,v,t)\der v \der a+\int_{\{c_{0}v^{\frac{2}{3}}\leq a<2c_{0}v^{\frac{2}{3}}\}}av^{l}g(a,v,t)\der v \der a \nonumber\\
   &\leq-\int_{\{a\geq c_{0}v^{\frac{2}{3}}\}}av^{l}g(a,v,t)\der v \der a+2c_{0}\int_{(0,\infty)^{2}}v^{\frac{2}{3}+l}g(a,v,t)\der v \der a.
\end{align}
Combining (\ref{equation thus becomes}), (\ref{surface area bound added}) and (\ref{surface area bound 3}), we deduce that 
\begin{align}\label{neededlaterforav^m}
    \frac{\der}{\der t} M_{1,l}(g(t))\lesssim (\frac{1}{3}+|l|)M_{1,l}(g(t))-(1-\gamma)R_{0}\frac{L}{4}M_{1,l}(g(t))+C_{\epsilon}.
\end{align}
Assume $\alpha>0$. If $l=0$ we conclude using a comparison argument since $(1-\gamma)R_{0}\frac{L}{4}=3$ cf. (\ref{definition_L}). The proof remains valid if we choose $l=-\alpha.$

If (\ref{alpha zero}) holds and we choose $l=-\delta_{1}$ to be such that $\frac{2}{3}-\delta_{1}\geq \gamma$, we have that $M_{0,\frac{2}{3}+l}$ and $M_{0,\frac{|\sigma|}{\mu}+l}$ are uniformly bounded from above and we can conclude in a similar manner.
\end{proof}
\begin{proof}[Proof of Proposition \ref{lemmaweaksol}]
    Let $T>0$. Let $S(t): U_{\epsilon,R}\rightarrow U_{\epsilon,R}$, for $t\in[0,T]$, be as in (\ref{definition semigroup}). Proposition \ref{propweaksol} guarantees that the semigroup is well-defined. Proposition \ref{continuous time semigroup} gives us continuity in time of the mapping $S(\cdot)g:[0,T]\rightarrow U_{\epsilon,R}$, for $g\in U_{\epsilon,R}$. Propositions \ref{firstmoments} - \ref{avnegativeprop} prove that $S(t)\omega(\epsilon,R,\delta)\subseteq\omega(\epsilon,R,\delta)$ if we choose the constants $c_{0,-\alpha-\tilde{\epsilon}},$ $c_{0,\tilde{m}},$ $c_{1,0}>0$ in Proposition \ref{lemmaweaksol} to correspond to the constants found in Propositions \ref{firstmoments} - \ref{avnegativeprop}.
\end{proof}
We are now in a position to prove Theorem \ref{themostimportanttheorem}.
\begin{proof}[Proof of Theorem \ref{themostimportanttheorem}, case $\alpha>0$. Existence of self-similar profiles]
Assume first $\alpha>0$. Let $T>0$. Let $S(t): U_{\epsilon,R}\rightarrow U_{\epsilon,R}$, for $t\in[0,T]$, be as in (\ref{definition semigroup}). For fixed $t\in[0,T]$, Proposition \ref{continuous semigroup 2} assures that the mapping  $S(t): U_{\epsilon,R}\rightarrow U_{\epsilon,R}$ is continuous in the weak-$^{\ast}$ topology. Proposition  \ref{lemmaweaksol} gives us that $S(t)\omega(\epsilon,R,\delta)\subseteq\omega(\epsilon,R,\delta)$.

Using a variant of Tykonov’s fixed point theorem (see, for instance, \cite[Theorem 1.2]{ESCOBEDO}), we can find a stationary solution in time, which we denote by $\tilde{g}_{\epsilon,R,\delta}(\eta)$. As $\tilde{g}_{\epsilon,R,\delta}\in\omega(\epsilon,R,\delta),$ we can find a subsequence $(\tilde{g}_{\epsilon_{k},R_{k},\delta_{k}})_{k\geq 1}$ that converges to some $g$ in the weak-$^{\ast}$ topology as $k\rightarrow\infty$. This argument has been used extensively to prove the existence of self-similar profiles for coagulation equations and therefore we do not give more details.
\

We need to prove that this $g$ satisfies (\ref{self_similar_equation}). Fix $\varphi\in\compactfundif$.
\

As $\tilde{g}_{\epsilon_{k},R_{k},\delta_{k}}\rightharpoonup g$, as $k\rightarrow \infty$, the linear terms containing $\tilde{g}_{\epsilon_{k},R_{k},\delta_{k}}$ will converge to the linear terms in (\ref{self_similar_equation}). We deal with the fusion term in the following manner:
\begin{align}\label{convergence fusion term}
&\bigg| \int_{(0,\infty)^{2}}  r_{\delta}( \eta)(c_{0}v^{\frac{2}{3}}-a)\partial_{a}\varphi(\eta)\tilde{g}_{\epsilon_{k},R_{k},\delta_{k}}(\der\eta)-  \int_{(0,\infty)^{2}} r( \eta)(c_{0}v^{\frac{2}{3}}-a)\partial_{a}\varphi(\eta)g(\der \eta)\bigg|\nonumber\\
& \leq C\int_{(0,\infty)^{2}} | r_{\delta}( \eta)-  r( \eta)||\partial_{a}\varphi(\eta)|\tilde{g}_{\epsilon_{k},R_{k},\delta_{k}}(\der \eta)+  \bigg|\int_{(0,\infty)^{2}} r( \eta)(c_{0}v^{\frac{2}{3}}-a)\partial_{a}\varphi(\eta) [g-\tilde{g}_{\epsilon_{k},R_{k},\delta_{k}}](\der \eta)\bigg|
\end{align}
As stated before, $r_{\delta}(\eta)\rightarrow r(\eta)$ for fixed $\eta$ as $\delta\rightarrow 0$. Due to the compact support of $\varphi$ and since $\tilde{g}_{\epsilon_{k},R_{k},\delta_{k}}\in\omega(\epsilon,R,\delta)$, the first term in (\ref{convergence fusion term}) goes to zero as $k\rightarrow\infty$. The second term in (\ref{convergence fusion term}) goes to zero as $k\rightarrow\infty$ since $\tilde{g}_{\epsilon_{k},R_{k},\delta_{k}}\rightharpoonup g$ as $k\rightarrow \infty$.

We now examine the coagulation term
\begin{align}
\int_{(0,\infty)^{2}}\int_{(0,\infty)^{2}}K_{\epsilon,R}(\eta,\eta')\xi_{R}(v+v')[\varphi(\eta+\eta')-\varphi(\eta)-\varphi(\eta')]\tilde{g}_{\epsilon_{k},R_{k},\delta_{k}}(\eta)\tilde{g}_{\epsilon_{k},R_{k},\delta_{k}}(\eta')\der \eta'\der \eta.   
\end{align}
Notice first that the support of $\tilde{g}_{\epsilon_{k},R_{k},\delta_{k}}$ is contained in the strip $v\in[\epsilon,2R)$. We can then replace $K_{\epsilon,R}(\eta,\eta')$ by $K(\eta,\eta')$. On the other hand, 
\begin{align*}
\int_{(0,\infty)^{2}}\int_{(0,\infty)^{2}}K(\eta,\eta')|\xi_{R}(v+v')-1|[\varphi(\eta+\eta')-\varphi(\eta)-\varphi(\eta')]\tilde{g}_{\epsilon_{k},R_{k},\delta_{k}}(\eta)\tilde{g}_{\epsilon_{k},R_{k},\delta_{k}}(\eta')\der \eta'\der \eta
\end{align*}
will go uniformly to zero as $R\rightarrow \infty$ using only the uniform moment estimates in $v$, namely $M_{0,1}$ and $M_{0,-\alpha}$.
Take a continuous function $p:\mathbb{R}_{+}\rightarrow[0,1]$ such that $p(x)=1,$ when $x\leq 1$ and $p(x)=0,$ when $x\geq 2$. Define then $p_{M}(a,v,a',v')=p(\frac{a}{M})p(\frac{v}{M})p(\frac{a'}{M})p(\frac{v'}{M})p(\frac{1}{aM})p(\frac{1}{vM})p(\frac{1}{a'M})p(\frac{1}{v'M})$, for $M>1$ large enough. We have that
\begin{align*}
\int_{(0,\infty)^{2}}\int_{(0,\infty)^{2}}K(\eta,\eta')[\varphi(\eta+\eta')-\varphi(\eta)-\varphi(\eta')]p_{M}(\eta,\eta')\tilde{g}_{\epsilon_{k},R_{k},\delta_{k}}(\eta)\tilde{g}_{\epsilon_{k},R_{k},\delta_{k}}(\eta')\der \eta'\der \eta   
\end{align*}
converges to
\begin{align*}
\int_{(0,\infty)^{2}}\int_{(0,\infty)^{2}}K(\eta,\eta')[\varphi(\eta+\eta')-\varphi(\eta)-\varphi(\eta')]p_{M}(\eta,\eta')g(\eta)g(\eta')\der \eta'\der \eta,  
\end{align*}
as $k\rightarrow\infty$ since the support of the integral is now on a compact set. An argument closely related can be found in \cite[Proof of Theorem 2.3]{ferreira2020stationary}. It remains to estimate the contribution outside the set $[\frac{1}{M},M]^{4}$. When $v<\frac{1}{M}$ or $v>M$, the estimates are similar to the one-dimensional case as the kernel is bounded by functions depending only on $v$. We can prove that the integral converges to zero as $M\rightarrow\infty$ in the region $\{v<\frac{1}{M}\}$, using the uniform estimates for $M_{0,-\alpha-\tilde{\epsilon}}$ and $M_{0,\beta}$. When $v>M$, we use that $M_{0,-\alpha}$ and $M_{0,1}$ are uniformly bounded to obtain the desired convergence. The regions where $\{v'<M\}$ and $\{v'>\frac{1}{M}\}$ are treated in the same manner by symmetry. We can treat the region where $\{a<\frac{1}{M}\}$ in the same manner as the region where $\{v<\big(\frac{1}{c_{0}M}\big)^{\frac{3}{2}}\}$ due to the isoperimetric inequality $a\geq c_{0}v^{\frac{2}{3}}$. For $a>M$, the main point is to control 
\begin{align*}
\int_{(0,\infty)^{2}\cap\{a>M\}}(v^{-\alpha}+v^{\beta})\tilde{g}_{\epsilon,R,\delta}(\eta).
\end{align*}
We split the region into three parts and use the fact that for large and small values of $v$, we can control this region using only moments in $v$ as before. More precisely, let $\tilde{C}>1$ large, then
\begin{align}\label{control large area}
 \int_{\{a>M\}}(v^{-\alpha}+v^{\beta})\tilde{g}_{\epsilon,R,\delta}(\eta)\der \eta&\leq 2\int_{(0,\infty)^{2}\cap\{v<\frac{1}{\tilde{C}}\}}v^{-\alpha}\tilde{g}_{\epsilon,R,\delta}(\eta)\der \eta+\int_{\{a>M,v\in[\frac{1}{\tilde{C}},\tilde{C}]\}}(v^{-\alpha}+v^{\beta})\tilde{g}_{\epsilon,R,\delta}(\eta)\der \eta\nonumber\\
 &+2\int_{(0,\infty)^{2}\cap\{v>\tilde{C}\}}v^{\beta}\tilde{g}_{\epsilon,R,\delta}(\eta)\der \eta\nonumber\\
 &\leq2\tilde{C}^{-\tilde{\epsilon}}\int_{(0,\infty)^{2}\cap\{v<\frac{1}{\tilde{C}}\}}v^{-\alpha-\tilde{\epsilon}}\tilde{g}_{\epsilon,R,\delta}(\eta)\der \eta+2\tilde{C}M^{-1}\int_{\{a>M,v\in[\frac{1}{\tilde{C}},\tilde{C}]\}}a\tilde{g}_{\epsilon,R,\delta}(\eta)\der \eta\nonumber\\
 &+2\tilde{C}^{\beta-1}\int_{(0,\infty)^{2}\cap\{v>\tilde{C}\}}v\tilde{g}_{\epsilon,R,\delta}(\eta)\der \eta\nonumber\\
  &\leq2\tilde{C}^{-\tilde{\epsilon}}c_{0,-\alpha-\tilde{\epsilon}}+2\tilde{C}M^{-1}c_{1,0}+2\tilde{C}^{\beta-1},
\end{align}
where $\tilde{\epsilon}$ is chosen as in Proposition \ref{lemmaweaksol} and $c_{1,0}, c_{0,-\alpha-\tilde{\epsilon}}$ are the constants found in Proposition \ref{lemmaweaksol}.  For given $\tilde{\delta}\in(0,1)$, we first take $\tilde{C}$ to be such that $2\tilde{C}^{-\tilde{\epsilon}}c_{0,-\alpha-\tilde{\epsilon}}+2\tilde{C}^{\beta-1}\leq \frac{\tilde{\delta}}{2}$ and then $M$ such that $2\tilde{C}M^{-1}c_{1,0}\leq\frac{\tilde{\delta}}{2}$. This proves our desired convergence of $\int_{(0,\infty)^{2}\cap\{a>M\}}(v^{-\alpha}+v^{\beta})\tilde{g}_{\epsilon,R,\delta}(\eta)$ to zero as $M\rightarrow\infty$, thus concluding our proof.
\end{proof}
\begin{proof}[Proof of Theorem \ref{main teo case alpha zero}, case $\alpha=0$. Existence of self-similar profiles]
We keep the notation used in the Proof of Theorem \ref{themostimportanttheorem}, case $\alpha>0$. The proof is done in the same manner as in the case $\alpha>0$. The difference is that we derived uniform bounds for moments of the form $M_{0,c},$ with $c\geq \gamma$ and for the moment $M_{1,0}$, but we have no information about moments of the form $M_{0,c},$ with $c<\gamma$. This implies that we now have to control the contribution of the regions $\{v\leq \frac{1}{M}\}$ and $\{v'\leq \frac{1}{M}\},$ for a sufficiently large $M$, in the term containing the coagulation kernel
\begin{align}\label{full coag kernel}
\int_{(0,\infty)^{2}}\int_{(0,\infty)^{2}}K_{\epsilon,R}(\eta,\eta')\xi_{R}(v+v')v[\varphi(\eta+\eta')-\varphi(\eta)]\tilde{g}_{\epsilon_{k},R_{k},\delta_{k}}(\eta)\tilde{g}_{\epsilon_{k},R_{k},\delta_{k}}(\eta')\der \eta'\der \eta.
\end{align}
For the region $\{v<\frac{1}{M}\}$, we have that:
\begin{align}
&\bigg|\int_{\{v<\frac{1}{M}\}}\int_{(0,\infty)^{2}}K_{\epsilon,R}(\eta,\eta')\xi_{R}(v+v')v[\varphi(\eta+\eta')-\varphi(\eta)]\tilde{g}_{\epsilon_{k},R_{k},\delta_{k}}(\eta)\tilde{g}_{\epsilon_{k},R_{k},\delta_{k}}(\eta')\der \eta'\der \eta\bigg| \label{full term v small}\\
&\lesssim\int_{\{v<\frac{1}{M}\}}\int_{(0,\infty)^{2}}(v^{\beta+1}+vv'^{\beta})|\varphi(\eta+\eta')-\varphi(\eta)|\tilde{g}_{\epsilon_{k},R_{k},\delta_{k}}(\eta)\tilde{g}_{\epsilon_{k},R_{k},\delta_{k}}(\eta')\der \eta'\der \eta\nonumber\\
&\lesssim2||\varphi||_{\infty}M^{\beta-1}M_{0,\beta}(\tilde{g}_{\epsilon_{k},R_{k},\delta_{k}})^{2}+M^{-\beta}\int_{\{v<\frac{1}{M}\}}\int_{(0,\infty)^{2}}v|\varphi(\eta+\eta')-\varphi(\eta)|\tilde{g}_{\epsilon_{k},R_{k},\delta_{k}}(\eta)\tilde{g}_{\epsilon_{k},R_{k},\delta_{k}}(\eta')\der \eta'\der \eta\nonumber. 
\end{align}
Thus, in order to show that the term in (\ref{full term v small}) goes to zero as $M\rightarrow\infty,$ it suffices to show that the integral above is bounded.
\begin{align*}
   &\int_{\{v<\frac{1}{M}\}}\int_{(0,\infty)^{2}}v|\varphi(\eta+\eta')-\varphi(\eta)|\tilde{g}_{\epsilon_{k},R_{k},\delta_{k}}(\eta)\tilde{g}_{\epsilon_{k},R_{k},\delta_{k}}(\eta')\der \eta'\der \eta\\
   &\leq ||\nabla{\varphi}||_{\infty}M_{0,1}(\tilde{g}_{\epsilon_{k},R_{k},\delta_{k}})\int_{(0,\infty)^{2}}|\eta'|\tilde{g}_{\epsilon_{k},R_{k},\delta_{k}}(\eta')\der \eta'\leq C.
\end{align*}
We now deal with the region where $\{v'<\frac{1}{M}\}$:
\begin{align}
&\bigg|\int_{(0,\infty)^{2}}\int_{\{v'<\frac{1}{M}\}}K_{\epsilon,R}(\eta,\eta')\xi_{R}(v+v')v[\varphi(\eta+\eta')-\varphi(\eta)]\tilde{g}_{\epsilon_{k},R_{k},\delta_{k}}(\eta)\tilde{g}_{\epsilon_{k},R_{k},\delta_{k}}(\eta')\der \eta'\der \eta\bigg| \nonumber\\
&\lesssim\int_{(0,\infty)^{2}}\int_{\{v'<\frac{1}{M}\}}(v^{\beta+1}+vv'^{\beta})|\varphi(\eta+\eta')-\varphi(\eta)|\tilde{g}_{\epsilon_{k},R_{k},\delta_{k}}(\eta)\tilde{g}_{\epsilon_{k},R_{k},\delta_{k}}(\eta')\der \eta'\der \eta.\label{firstsecond term}
\end{align}
For the first term in (\ref{firstsecond term}), we have:
\begin{align*}
  &\int_{(0,\infty)^{2}}\int_{\{v'<\frac{1}{M}\}}v^{\beta+1}|\varphi(\eta+\eta')-\varphi(\eta)|\tilde{g}_{\epsilon_{k},R_{k},\delta_{k}}(\eta)\tilde{g}_{\epsilon_{k},R_{k},\delta_{k}}(\eta')\der \eta'\der \eta\\
  &\leq ||\nabla \varphi||_{\infty} \int_{(0,\infty)^{2}}\int_{\{v'<\frac{1}{M}\}}v^{\beta+1}(a'+v')\tilde{g}_{\epsilon_{k},R_{k},\delta_{k}}(\eta)\tilde{g}_{\epsilon_{k},R_{k},\delta_{k}}(\eta')\der \eta'\der \eta\\
 & \leq M^{\beta-1}||\nabla\varphi||_{\infty} M_{0,\beta+1}(\tilde{g}_{\epsilon_{k},R_{k},\delta_{k}})M_{0,\beta}(\tilde{g}_{\epsilon_{k},R_{k},\delta_{k}})+ M^{-\delta_{1}}||\nabla\varphi||_{\infty}M_{0,\beta+1}(\tilde{g}_{\epsilon_{k},R_{k},\delta_{k}})M_{1,-\delta_{1}}(\tilde{g}_{\epsilon_{k},R_{k},\delta_{k}}).
\end{align*}
For the second term in (\ref{firstsecond term}), we have:
\begin{align*}
&\int_{(0,\infty)^{2}}\int_{\{v'<\frac{1}{M}\}}vv'^{\beta}|\varphi(\eta+\eta')-\varphi(\eta)|\tilde{g}_{\epsilon_{k},R_{k},\delta_{k}}(\eta)\tilde{g}_{\epsilon_{k},R_{k},\delta_{k}}(\eta')\der \eta'\der \eta\\
&\leq M^{-\beta}\int_{(0,\infty)^{2}}\int_{\{v'<\frac{1}{M}\}}v|\varphi(\eta+\eta')-\varphi(\eta)|\tilde{g}_{\epsilon_{k},R_{k},\delta_{k}}(\eta)\tilde{g}_{\epsilon_{k},R_{k},\delta_{k}}(\eta')\der \eta'\der \eta\\
    &\leq  M^{-\beta}||\nabla{\varphi}||_{\infty}[M_{0,1}^{2}(\tilde{g}_{\epsilon_{k},R_{k},\delta_{k}})+M_{0,1}(\tilde{g}_{\epsilon_{k},R_{k},\delta_{k}})M_{1,0}(\tilde{g}_{\epsilon_{k},R_{k},\delta_{k}})].
\end{align*}
We prove that the term in (\ref{full coag kernel}) converges to zero as $M\rightarrow\infty$ in the regions $\{v>M\}, \{v'>M\}$, $\{a<\frac{1}{M}\}$ and $\{a'<\frac{1}{M}\}$ in a similar manner. We deal with the regions where $\{a>M\}$ and $\{a'>M\}$ as in (\ref{control large area}). This concludes the proof.
\end{proof}
\subsection{Moment estimates for arbitrary powers of area and volume}\label{subsection better moment estimates}
We construct a self-similar profile $\tilde{g}$ which satisfies $M_{n,k}(\tilde{g})<\infty,$ for $n,k\in\mathbb{R}$, if (\ref{alpha non neg}) holds, and $M_{n,k}(\tilde{g})<\infty,$ for $n\in\mathbb{N}$, $k\geq \gamma$, if (\ref{alpha zero}) holds. In order to do so, we apply the strategy used to find a self-similar profile in Theorem \ref{themostimportanttheorem}. In order to estimate the moments $M_{n,0}(\tilde{g})$, $n\in\mathbb{N}$, we notice that we can improve the estimates found in Proposition \ref{avnegativeprop}, while to derive bounds for $M_{-n,0}(\tilde{g})$, $n\in\mathbb{N}$, we use the fact that they can be estimated in terms of $M_{0,-\frac{2}{3}n}(\tilde{g})$.
\begin{rmk}\label{rmkanegavnega}
Let $n>0$ and $k\in\mathbb{R}$. In order to bound $M_{-n,k}(\tilde{g})$, it is enough to use $a\geq c_{0}v^{\frac{2}{3}}$ to obtain $\int_{\{a\geq c_{0}v^{\frac{2}{3}}\}}a^{-n}v^{k}\tilde{g}(a,v)\der v \der a\leq c_{0}^{-n}\int_{(0,\infty)^{2}}v^{k-\frac{2}{3}n}\tilde{g}(a,v)\der v \der a$ if $M_{0,k-\frac{2}{3}n}(\tilde{g})<\infty$.
\end{rmk}
In order to find uniform bounds for the moments $M_{n,0}(\tilde{g})$, we need to be able to test (\ref{regularizedformequation}) with higher powers of $a$.

Let $N\in\mathbb{N}, N\geq 2$. 
For some $\tilde{g}\in\mathscr{M}^{I}_{+}(\mathbb{R}_{>0}^{2})$ such that
\begin{align}\label{finalization_space_generalcase}
\int_{(0,\infty)^{2}}(1+a^{N})\tilde{g}(a,v)\der v \der a<\infty,
\end{align}
we define the space
\begin{align}\label{existence_space_strong_general}
U_{\epsilon,R,N}&:= \{\tilde{g}\in\mathscr{M}_{+}^{I}(\mathbb{R}^{2}_{>0}), \tilde{g}\Big(\mathbb{R}^{2}_{>0}\setminus[c_{0}\epsilon^{\frac{2}{3}},\infty)\times[\epsilon,2R)\Big)=0, \tilde{g} \textup{ satisfies } (\ref{finalization_space_generalcase})\}.
\end{align}
\begin{prop}\label{propweaksol_generalcase}
Let $K_{\epsilon,R}$ be the kernel defined in (\ref{truncation_kernel_selfsimilar}). Let $N\in\mathbb{N}$, $N\geq 2$. Let $g_{\textup{in},R}\in\mathscr{M}_{+,\textup{b}}^{I}(\mathbb{R}^{2}_{>0})\cap U_{\epsilon,R,N}$. 
There exists a unique solution $g_{\epsilon,R,\delta}\in\textup{C}^{1}([0,\infty);\mathscr{M}_{+,\textup{b}}^{I}(\mathbb{R}^{2}_{>0}))$, $g_{\epsilon,R,\delta}(t)\in U_{\epsilon,R,N}$, that satisfies
\begin{align*}
\sup_{t\in[0,T]}\int_{(0,\infty)^{2}}(1+a^{N})g_{\epsilon,R,\delta}(a,v,t)\der v \der a<\infty,
\end{align*} 
for every $T>0$, for the weak formulation of the coagulation equation (\ref{regularizedformequation}) with initial datum $g_{\textup{in},R}$.
\end{prop}
\begin{proof}
The proof is done via a standard fixed-point argument as in the proof of Proposition \ref{propweaksol}. The space $U_{\epsilon,R,N}$ was chosen since we would like to later test with higher powers of the area $a$ and we need to control the escape to infinity of $a$. We prove existence in the same manner as in Proposition \ref{propweaksol} with the aid of Propositions \ref{contractivemap_generalcase} and \ref{extensionalltime_generalcase}, which will be stated below. 
\end{proof}

Proposition \ref{propweaksol_generalcase} states that we can obtain better moment estimates if we assume moment estimates for higher order powers for the initial datum, namely with (\ref{finalization_space_generalcase}) instead of (\ref{finalization_space}).

As before, we use $K_{R}$, see (\ref{truncationkernel}), and $K_{\epsilon,R}$, see (\ref{truncation_kernel_selfsimilar}), interchangeably.

We define the space \begin{align*}
\textup{Y}_{N}=&\{G\in\textup{C}([0,\tau_{N}];\mathscr{M}_{+,\textup{b}}^{I}(\mathbb{R}^{2}_{>0})): ||G||_{N}= \sup_{0\leq t\leq \tau_{N}}\bigg| \int_{(0,\infty)^{2}}\textup{e}^{h_{\epsilon}(V,t)}(1+x_{t,V}^{N}(A))G(A,V,t) \der V \der  A \bigg|\leq 2M_{N};\\ 
& G(\mathbb{R}^{2}_{>0}\setminus [c_{0}\epsilon^{\frac{2}{3}},\infty)\times[\epsilon,2R\textup{e}^{t}), t)=0, \textup{ } t\in[0,\tau_{N}]\}
\end{align*}
instead of $Y_{\epsilon,\tau}$ in (\ref{spaceexistence}), where $M_{N}=2\int_{(0,\infty)^{2}}(1+a^{N})g_{\textup{in},R}(a,v) \der v \der a+1$, $g_{\textup{in},R}$ is as in Proposition  \ref{propweaksol_generalcase} and we fix a time $\tau_{N}$ instead of $\tau$ such that
\begin{align}
(2^{N}(1-\gamma)M_{N}+M_{N}+1)||K_{R}||_{\infty}(\textup{e}^{\tau_{N}}-1) & < \frac{1}{2}; \label{time general 1}\\
2^{N}(1-\gamma)M_{N}^{2}||K_{R}||_{\infty}\tau_{N} & <1 \textup{   and   } \tau_{N}\leq \ln{2}.\label{time general 2}
\end{align}
\begin{prop}\label{contractivemap_generalcase}
Let $N\in\mathbb{N}, N\geq 2$. Let $K_{\epsilon,R}$ defined as in (\ref{truncation_kernel_selfsimilar}) and assume K and r satisfy the conditions stated in Theorem \ref{themostimportanttheorem} or in
Theorem \ref{main teo case alpha zero}. Let $g_{\textup{in},R}$ as stated in Proposition \ref{propweaksol_generalcase}. Let $J$ be as in (\ref{definitionjf}). Assume $F,G\in\textup{Y}_{N}$ with $\tau_{N}$ as in (\ref{time general 1}) and (\ref{time general 2}). Then
\begin{enumerate}
\item $||J[G]||\leq 2M_{N}$; 
\item The map $J$ is contractive, more explicitly, $||J[F]-J[G]||_{N}\leq \frac{1}{2}||F-G||_{N}$.
\end{enumerate}
\end{prop}
\begin{proof}
We follow the proof of Proposition \ref{contractivemap}, working with $\textup{Y}_{N}$, $M_{N}$ and $\tau_{N}$ instead of $\textup{Y}_{\epsilon,\tau}$, $M$ and $\tau$, respectively. We make use in addition of the inequality 
\begin{align*}
    1+(x_{t,V}(A)+x_{t,V'}(A'))^{N}\leq 2^{N-1}(1+x_{t,V}^{N}(A))(1+x_{t,V'}^{N}(A')).
\end{align*}
For more details, see Proof of Proposition \ref{contractivemap} in Appendix \ref{appendix b}, where similar arguments were used.
\end{proof}
 We make use in addition of the following inequality provided in  \cite[Lemma 2]{inequality}.
\begin{lem}\label{inequalityforiterationargument}
Assume that $p>1$ and let $k_{p}:=\big[\frac{p+1}{2}\big]$. Then, for
all $x, y > 0$, the following inequalities hold:
\begin{align*}
    \sum_{k=1}^{k_{p}-1}\binom{p}{k}(x^{k}y^{p-k}+x^{p-k}y^{k})\leq (x+y)^{p}-x^{p}-y^{p}\leq \sum_{k=1}^{k_{p}}\binom{p}{k}(x^{k}y^{p-k}+x^{p-k}y^{k}).
\end{align*}
\end{lem}
\begin{prop}\label{extensionalltime_generalcase}
Let $N\in\mathbb{N}$, $N\geq 2$, $N$ fixed. Let $G_{\epsilon, R,\delta}$ be the found fixed point for (\ref{definitionjf}) up to time $\tau_{N}$ defined as in  (\ref{time general 1}), (\ref{time general 2}) and with initial datum $g_{\textup{in},R}$ taken as in Proposition \ref{propweaksol_generalcase}. For any $T>0$, there exists a unique solution, for which we keep the notation $G_{\epsilon,R,\delta}\in\textup{C}^{1}([0,T];\mathscr{M}_{+,\textup{b}}^{I}(\mathbb{R}^{2}_{>0}))$ that satisfies (\ref{equation_change_of_var}). 
\end{prop}
The proof of this proposition is based on an iterative argument and is given in Appendix \ref{appendix b}.

Let $K_{\epsilon,R}$ defined as in (\ref{truncation_kernel_selfsimilar}) and assume $K$ and $r$ satisfy the conditions stated in Theorem \ref{themostimportanttheorem} or in Theorem \ref{main teo case alpha zero}. Let $T>0$. Define $S(t): U_{\epsilon,R,N}\rightarrow U_{\epsilon,R,N}$ in the following way:  
\begin{align}\label{definition semigroup general}
S(t)g_{\textup{in},R}=g_{\epsilon,R,\delta}(\cdot,\cdot,t)
\end{align}
for all $t\in[0,T]$, where $g_{\epsilon,R,\delta}$ is the unique solution of the weak formulation of the coagulation equation with coagulation kernel $K_{\epsilon,R}$ found Proposition \ref{propweaksol_generalcase}.
\begin{prop}\label{lemma_weak_sol_general_case}
Assume $g_{\epsilon,R,\delta}$ is the solution found in Proposition \ref{propweaksol_generalcase} for some $N\in\mathbb{N},$ $N\geq 2$, fixed, with kernel $K_{\epsilon,R}$ taken as in (\ref{truncation_kernel_selfsimilar}). Let $\mu>0$.

If we are in the case (\ref{alpha non neg}), for $k\in\mathbb{R}, n\in[1,N]$, there exist constants $c_{0,k}, c_{1,-\alpha}, c_{1,1}, c_{n,0}>0$ for which the set
\begin{align*}
\omega=\{& M_{0,1}(g_{\epsilon,R,\delta}(t))=1, M_{0,k}(g_{\epsilon,R,\delta}(t))\leq c_{0,k},M_{1,-\alpha}(g_{\epsilon,R,\delta}(t))\leq c_{1,-\alpha}, M_{1,1}(g_{\epsilon,R,\delta}(t))\leq c_{1,1},\\
&M_{n,0}(g_{\epsilon,R,\delta}(t))\leq c_{n,0}\}
\end{align*}
is preserved for all $t\in[0,T]$, that means $S(t)\omega\subseteq\omega$.

If we are in the case (\ref{alpha zero}), for $k\geq \gamma, n\in[1,N]$, there exist constants $c_{0,k}, c_{1,1}, c_{n,0}>0$ for which the set
\begin{align*}
\omega=\{M_{0,1}(g_{\epsilon,R,\delta}(t))=1, M_{0,k}(g_{\epsilon,R,\delta}(t))\leq c_{0,k},
M_{1,1}(g_{\epsilon,R,\delta}(t))\leq c_{1,1}, M_{n,0}(g_{\epsilon,R,\delta}(t))\leq c_{n,0}\}
\end{align*}
is preserved for all $t\in[0,T]$, that means $S(t)\omega\subseteq\omega$.
\end{prop}
We prove this proposition after proving first some uniform moment estimates. To bound moments of the form $M_{0,c},$ $c\in\mathbb{R}$ and $M_{1,-\alpha}$, we use the same method that we used in Section \ref{chapter_two}, namely Propositions \ref{firstmoments} - \ref{avnegativeprop}. In the next propositions we prove uniform bounds for $M_{1,1}$ and $M_{n,0},$ for $n\in[2,N]$.

We prove uniform bounds for $M_{1,m}(g_{\epsilon,R,\delta})$, $m\geq 1,$ instead of only proving uniform bounds for the moment $M_{1,1}$ as the proof is done in a similar manner.
\begin{prop}\label{avpositiveprop}
Let $g_{\textup{in},R}\in\omega$ and $g_{\epsilon,R,\delta}$ (which we will denote by $g_{\textup{in}}$ and $g$, respectively) be as in Proposition \ref{propweaksol_generalcase}. Let $m \geq 1$. Assume in addition that $\int_{(0,\infty)^{2}}av^{m}g_{\textup{in},R}(a,v)\der v \der a<\infty$. Then there exists $C_{1,m}>0$, which does not depend on $\epsilon, R,\delta,$ such that:
\begin{align}
\sup_{t\geq 0}M_{1,m}(g(\cdot,\cdot,t))\leq \max\{M_{1,m}(g_{\textup{in}}), C_{1,m}\}. \label{avpositive}
\end{align}
\end{prop}
\begin{proof}
The proof relies on the fact that the coagulation operator does not contribute with too fast interactions, using in addition the uniform estimate for the total surface area, which we were able to derive making use of the fusion term.

We use
\begin{align*}
-\int_{(0,\infty)^{2}}av^{m}g(a,v,t)\Theta_{\epsilon}(v)\der v \der a\leq -M_{1,m}(g(t))+2^{m}\int_{(0,\infty)^{2}}ag(a,v,t)\der v \der a
\end{align*}
and, as $M_{1,0}(g)$ is uniformly bounded from above, equation (\ref{regularizedformequation}) becomes:
\begin{align*}
   \frac{\der}{\der t} M_{1,m}(g(t))\leq (1-\frac{2}{3}-m)M_{1,m}(g(t))+C+(1-\gamma)R_{0}\int_{(0,\infty)^{2}}\frac{a^{\mu}\max\{v^{\sigma},L\delta\}}{1+\delta a^{\mu}}(c_{0}v^{\frac{2}{3}}-a)g(a,v,t)v^{m}\der v \der a\\
   +\frac{1-\gamma}{2}\int_{(0,\infty)^{2}}\int_{(0,\infty)^{2}}K(a,v,a',v')[(a+a')(v+v')^{m}-av^{m}-a'v'^{m}]g(a,v,t)g(a',v',t)\der v' \der a' \der v \der a.
\end{align*}
The term with the fusion kernel is non-positive and can be omitted. For the term with the coagulation kernel, we combine
\begin{align*}
    (a+a')(v+v')^{m}-av^{m}-a'v'^{m}= a[(v+v')^{m}-v^{m}]+a'[(v+v')^{m}-v'^{m}]
\end{align*}
with the inequality:
\begin{align*}
    (v+v')^{m}-v^{m}=m\int_{v}^{v+v'}s^{m-1}\der s\leq m(v+v')^{m-1}v'\leq C_{m}(v'^{m}+v^{m-1}v'),
\end{align*}
and obtain
\begin{align*}
   K(\eta,\eta')[(a+a')(v+v')^{m}-av^{m}-a'v'^{m}]&\lesssim C_{m}K(a,v,a',v')(av^{m-1}v'+av'^{m}+a'v'^{m-1}v+a'v^{m}) \\
   &\lesssim K_{0}C_{m}(v^{-\alpha}v'^{\beta}+v'^{-\alpha}v^{\beta})(av^{m-1}v'+av'^{m}+a'v'^{m-1}v+a'v^{m}).
\end{align*}
The moments $M_{0,c},$ $c\geq \gamma,$ are bounded from above. The term $M_{1,-\alpha}$ is uniformly bounded from above. This was proven in Proposition \ref{avnegativeprop}. If $m>1$, then $m+\beta-1,\beta,m-\alpha-1\in(-\alpha,m)$ and we use H\"{o}lder's inequality and conclude using a comparison argument. If $m=1$, we only have to bound $av^{\beta}$. As $\beta\in(-\alpha,m),$ we conclude in the same manner as before.
\end{proof}

\begin{prop}\label{iteration_exponential_decay}
Let $g_{\textup{in},R}\in\omega$ and $g_{\epsilon,R,\delta}$ (which we will denote by $g_{\textup{in}}$ and $g$, respectively) be as in Proposition \ref{propweaksol_generalcase}, for $N\in\mathbb{N},$ $N\geq 2,$ fixed. Then, for any $n\in[2,N]$, there exists $C_{n,0}>0$, independent of $\epsilon,R,\delta,N,$ but dependent on $n$, such that:
\begin{align}
\sup_{t\geq 0}M_{n,0}(g(\cdot,\cdot,t))\leq \max\{M_{n,0}(g_{\textup{in}}), C_{n,0}\}. \label{alargevall}
\end{align}
\end{prop}
\begin{proof}
This estimate follows from the fact that fusion overtakes coagulation in the case $\mu>0$. 

Notice first that due to the choice of the space $\textup{Y}_{N}$, we can test equation (\ref{regularizedformequation}) with $a^{n}$, for $n\in[2,N]$. We have that there exists $C_{n}>0$ such that:
\begin{align*}
  K_{\epsilon,R}(a,v,a',v')[(a+a')^{n}-a^{n}-a'^{n}]&\leq K(a,v,a',v')\sum_{k=1}^{k_{n}}\binom{n}{k}(a^{k}a'^{n-k}+a^{n-k}a'^{k})\\
  &\leq C_{n}(v^{-\alpha}v'^{\beta}+v^{\beta}v'^{-\alpha})(a^{n-1}a'+aa'^{n-1}),
\end{align*}
where $k_{n}$ is taken as in Lemma \ref{inequalityforiterationargument}.
We have a uniform upper bound for $\int_{(0,\infty)^{2}}(av^{-\alpha}+av^{\beta})g(a,v,t)\der v \der a$ due to Proposition \ref{avpositiveprop}.

If $\alpha>0$, we use the fact that there exists $C_{n,\epsilon_{1}}>0$ for which:
\begin{align*}
    \int_{(0,\infty)^{2}}a^{n-1}v^{-\alpha}g(a,v,t)\der v \der a\leq \epsilon_{1}\int_{(0,\infty)^{2}}a^{n}g(a,v,t)\der v \der a+C_{n,\epsilon_{1}}\int_{(0,\infty)^{2}}v^{-n\alpha}g(a,v,t)\der v \der a
    \end{align*}
    and
\begin{align*}
\int_{(0,\infty)^{2}}a^{n-1}v^{\beta}g(a,v,t)\der v \der a\leq \epsilon_{1}\int_{(0,\infty)^{2}}a^{n}g(a,v,t)\der v \der a+C_{n,\epsilon_{1}}\int_{(0,\infty)^{2}}v^{n\beta}g(a,v,t)\der v \der a.
\end{align*}

If $\alpha=0$ and $n=2,$ we use that $M_{1,0}$ and $M_{1,\beta}$ are uniformly bounded from above. If $\alpha=0$ and $n\geq 3$, then
\begin{align*}
    \int_{(0,\infty)^{2}}a^{n-1}g(a,v,t)\der v \der a\leq \epsilon_{1}\int_{(0,\infty)^{2}}a^{n}g(a,v,t)\der v \der a+C_{n,\epsilon_{1}}\int_{(0,\infty)^{2}}ag(a,v,t)\der v \der a.
    \end{align*}

As the linear term is non-positive, equation (\ref{regularizedformequation}) becomes:
\begin{align*}
   \frac{\der}{\der t} M_{n,0}(g(t))&\leq (1-\gamma)nR_{0}\int_{(0,\infty)^{2}}\frac{a^{\mu}\max\{v^{\sigma},L\delta\}}{1+\delta a^{\mu}}(c_{0}v^{\frac{2}{3}}-a)g(\eta,t)a^{n-1}\der \eta\\
   &+C_{(-\alpha,n,\beta)}\epsilon_{1} M_{n,0}(g(t))+C_{(-\alpha,n,\beta)}.
\end{align*}

We now use the same argument as in Proposition \ref{avnegativeprop}, modifying the Young's inequality part. As $\mu+n>n,$ we can apply Young's inequality to obtain:
\begin{align*}
    a^{n}=a^{n}v^{\frac{n\sigma}{\mu+n}}v^{-\frac{n\sigma}{\mu+n}}\leq\frac{\epsilon_{2}}{p}(av^{\frac{\sigma}{\mu+n}})^{\mu+n}+\frac{\epsilon_{2}^{-\frac{q}{p}}}{q}v^{-\frac{{n\sigma}}{\mu+n}q},
\end{align*}
where $p=\frac{\mu+n}{n}$ and $\frac{1}{p}+\frac{1}{q}=1$. So there exists a constant $C_{\epsilon_{2}}>0$, depending on $\epsilon_{2}$,
for which:
\begin{align*}
    a^{n}\lesssim \epsilon_{2} a^{\mu+n}v^{\sigma}+C_{\epsilon_{2}}v^{-n\frac{\sigma}{\mu}}
\end{align*}
and this implies:
\begin{align*}
    -\int_{\{a\geq 2c_{0}v^{\frac{2}{3}}\}}a^{\mu+n}v^{\sigma}g(a,v,t)\der v \der a\lesssim -\frac{1}{\epsilon_{2}}\int_{\{a\geq 2c_{0}v^{\frac{2}{3}}\}}a^{n}g(a,v,t)\der v \der a + C_{\epsilon_{2}}\int_{(0,\infty)^{2}}v^{-n\frac{\sigma}{\mu}}g(a,v,t)\der v\der a.
\end{align*}
As moments of the form $M_{0,c}$ are bounded from above, we choose $\epsilon_{2}$ sufficiently small in order to be able to repeat the argument in (\ref{surface area bound added}) and then we recover the region where $\{a\geq c_{0}v^{\frac{2}{3}}\}:$
\begin{align*}
   -\int_{\{a\geq 2c_{0}v^{\frac{2}{3}}\}}a^{n}g(a,v,t)\der v \der a&= -\int_{\{a\geq c_{0}v^{\frac{2}{3}}\}}a^{n}g(a,v,t)\der v \der a+\int_{\{c_{0}v^{\frac{2}{3}}\leq a<2c_{0}v^{\frac{2}{3}}\}}a^{n}g(a,v,t)\der v \der a \\
   &\leq-\int_{\{a\geq c_{0}v^{\frac{2}{3}}\}}a^{n}g(a,v,t)\der v \der a+(2c_{0})^{n}\int_{(0,\infty)^{2}}v^{\frac{2}{3}n}g(a,v,t)\der v \der a.
\end{align*}
We can choose $\epsilon_{1}>0$ sufficiently small so that equation (\ref{regularizedformequation}) becomes:
\begin{align*}
    \frac{\der}{\der t} M_{n,0}(g(t))\lesssim -n\frac{L(1-\gamma)}{4}R_{0}M_{n,0}(g(t))+3M_{n,0}(g(t))+C_{\epsilon_{1},\epsilon_{2},n}
    \leq -M_{n,0}(g(t))+C_{\epsilon_{1},\epsilon_{2},n}
\end{align*}
and we conclude using a comparison argument.
\end{proof}
\begin{proof}[Proof of Proposition \ref{lemma_weak_sol_general_case}]
    Let $T>0$. Let $N\in\mathbb{N}$, $N\geq 2$, fixed. Let $S(t): U_{\epsilon,R,N}\rightarrow U_{\epsilon,R,N}$, for $t\in[0,T]$, be as in (\ref{definition semigroup general}). Proposition \ref{propweaksol_generalcase} guarantees that the semigroup is well-defined. Propositions \ref{avpositiveprop} - \ref{iteration_exponential_decay} together with Propositions \ref{firstmoments} - \ref{avnegativeprop} prove that $S(t)\omega\subseteq\omega$ if we choose the constants $c_{0,k}, c_{1,-\alpha}, c_{1,1}, c_{n,0}>0$ in Proposition \ref{lemma_weak_sol_general_case} to correspond to the constants found in Propositions \ref{avpositiveprop} - \ref{iteration_exponential_decay}.
\end{proof}
We now prove the stated moment bounds for the self-similar profiles found in Theorem \ref{themostimportanttheorem} and Theorem \ref{main teo case alpha zero}.
\begin{proof}[Proof of Theorem \ref{themostimportanttheorem} and Theorem \ref{main teo case alpha zero}. Moment bounds]
The main point is to control positive powers of the area. Following the proof of Theorem \ref{themostimportanttheorem} and with the help of Proposition \ref{propweaksol_generalcase} and Proposition \ref{lemma_weak_sol_general_case}, we can prove that there exists a self-similar profile $g_{N}$ for the two-dimensional coagulation equation satisfying 
\begin{align}\label{allmoments}
    \int_{(0,\infty)^{2}}(1+a^{n})g_{N}(a,v)\der v \der a\leq c_{n,0},
\end{align}
for every $n\in[1,N]$. Moreover, if $\alpha>0$, we have that      $\int_{(0,\infty)^{2}}(v^{-\alpha-\tilde{\epsilon}}+v+a)g_{N}(a,v)\der v\der a$ is bounded uniformly independently of $N$, for $\tilde{\epsilon}$ as in Proposition \ref{lemmaweaksol}. If $\alpha=0$, we have that      $\int_{(0,\infty)^{2}}(v^{\gamma}+v^{\gamma+1}+a)g_{N}(a,v)\der v\der a$ is bounded uniformly independently of $N$. We can thus find a subsequence of $\{g_{N}\}_{N\geq 2}$ that converges to some $\tilde{g}$ in the weak-$^{\ast}$ topology and $\tilde{g}$ is a self-similar profile for the two-dimensional coagulation equation. This is since these are the moment estimates needed to prove the existence of a self-similar profile.

\
As the constants $c_{n,0}$ in (\ref{allmoments}) are independent of $N$, for $n\leq N$, we can conclude that there exist some constants $c_{d}>0$ such that $\int_{(0,\infty)^{2}}a^{d}\tilde{g}(a,v)\der v \der a <c_{d}$, for all $d\in[1,\infty).$

\
To control negative powers of $a$, we use Remark \ref{rmkanegavnega}. For combinations of positive powers of $a$ and powers of $v,$ we use Young's inequality. Take $d>0$ and $k\in\mathbb{R}$. Then:
\begin{align*}
    \int_{(0,\infty)^{2}}a^{d}v^{k}\tilde{g}(a,v)\der v \der a\leq  \frac{d}{d+1}\int_{(0,\infty)^{2}}a^{d+1}\tilde{g}(a,v)\der v \der a+ \frac{1}{d+1}\int_{(0,\infty)^{2}}v^{k(d+1)}\tilde{g}(a,v)\der v \der a.
\end{align*}
\end{proof}
\section{Different asymptotic behaviors for $\mu<0$}\label{case mu negative}
\subsection{Ramification theory}
First, we prove existence of a weak solution for the time-dependent fusion problem as in Definition \ref{definitiontimedependent} that satisfies the moment bounds (\ref{cond2_non-existence}) and (\ref{conditions_non-existence}) and with initial value $g_{\textup{in}}\in\mathscr{M}_{+}^{I}(\mathbb{R}_{>0}^{2})$. We initially prove well-posedness for a truncated version of the time-dependent problem. Since $\mu$ is negative, we can prove well-posedness using the truncation for the coagulation kernel, while the cut-off in the fusion rate is not needed. So we look at the equation:
\begin{align}\label{regularized_form_equation_negative}
\int_{(0,\infty)^{2}}g(\eta,t)\Theta_{\epsilon}(v)\varphi(\eta)\der\eta-\frac{2}{3}\int_{(0,\infty)^{2}}g(\eta,t)\Theta_{\epsilon}(v) a\partial_{a}\varphi(\eta)\der\eta-\int_{(0,\infty)^{2}}&g(\eta,t)\Theta_{\epsilon}(v)v\partial_{v}\varphi(\eta)\der \eta \nonumber \\
+(1-\gamma)\langle \mathbb{K}_{R}[g],\varphi\rangle+(1-\gamma)\int_{(0,\infty)^{2}}r(\eta)(c_{0}v^{\frac{2}{3}}-a)g(\eta,t)\partial_{a}\varphi(\eta)\der \eta =\partial_{t}&\int_{(0,\infty)^{2}}g(\eta,t)\varphi(\eta)\der \eta
\end{align}
and assume the initial value is $g_{\textup{in},R}=g_{\textup{in}}\mathbbm{1}_{\{[c_{0}\epsilon^{\frac{2}{3}},\infty)\times[\epsilon,2R)\}}$. For this form of the equation, we make use of Proposition \ref{propweaksol_generalcase} for the case $N=2$ and obtain a sequence of solutions $\{g_{\epsilon, R}\}_{\{\epsilon\in(0,1),R>1\}}$. Denote by $\alpha_{1}=\min\{-\alpha-\tilde{\epsilon},\gamma-\frac{1}{3}\}$ and let $\tilde{m}$ as in Proposition \ref{lemmaweaksol}. Using the same proof as for Proposition \ref{lemmaweaksol}, we can prove that the set
\begin{align}\label{gamma-1pe3}
    \omega=\omega(\epsilon,R):=\{\frac{1}{2}\leq M_{0,1}(g_{\epsilon,R})\leq 1; M_{0,\alpha_{1}}(g_{\epsilon,R})\leq c_{0,\alpha_{1}}; M_{0,\tilde{m}}(g_{\epsilon,R})\leq c_{0,\tilde{m}}, M_{0,\frac{\sigma}{|\mu|}}(g_{\epsilon,R})\leq c_{0,\frac{\sigma}{|\mu|}}\} 
\end{align}
stays preserved in time uniformly in $\epsilon, R$. 

In order to prove uniform moment bounds for negative powers of $v$, i.e. in order to prove $M_{0,c}$, with $c\leq 0$, we need a lower bound for the $M_{0,1}$ moment. Thus, we need to prove that the lower bound on $M_{0,1}$ is preserved in time, independently of $\epsilon,R$ and of the initial condition $g_{\textup{in},R}$.

To do this, we test equation (\ref{regularized_form_equation_negative}) with $\varphi(a,v)=v$ and make use of the following technical proposition.
\begin{prop} \label{lowerbound}
Let $g_{\epsilon,R}$ be the solution of equation (\ref{regularized_form_equation_negative}) with initial value $g_{\textup{in},R}=g_{\textup{in}}\mathbbm{1}_{\{[c_{0}\epsilon^{\frac{2}{3}},\infty)\times[\epsilon,2R)\}}$ and assume in addition that $g_{\textup{in},R}\in\omega$, where $\omega$ is as in (\ref{gamma-1pe3}). Then there exists $M>0$ such that:
\begin{align}
 \frac{M_{0,1}(g_{\textup{in}})}{2} \leq M_{0,1}(g_{\epsilon,R}(\cdot,\cdot,t))\leq M_{0,1}(g_{\textup{in}}),
\end{align}
for every $t\geq 0$ and uniformly in $\epsilon<\frac{1}{M}, R>M$.
\end{prop}
\begin{proof}
Fix $M>0$, sufficiently large, such that
\begin{align*}
    M^{\gamma-1}\int_{(0,\infty)^{2}}v^{\gamma}g_{\textup{in}}(a,v)\der v \der a+M^{1-\tilde{m}}\int_{(0,\infty)^{2}}v^{\tilde{m}}g_{\textup{in}}(a,v)\der v \der a+M^{-\frac{2}{3}}\int_{(0,\infty)^{2}}v^{\frac{1}{3}}g_{\textup{in}}(a,v)\der v \der a\leq \frac{M_{0,1}(g_{\textup{in}})}{2},
\end{align*}
where $\tilde{m}$ is as in Proposition \ref{lemmaweaksol} and $\beta$ as in (\ref{lower_bound_kernel}).
\

Testing in (\ref{regularized_form_equation_negative}) with $\varphi(a,v)=v$, we obtain that
\begin{align}\label{technical 1}
\int_{(0,\infty)^{2}}vg_{\epsilon,R}(a,v,t)\der v \der a=\int_{(0,\infty)^{2}}vg_{\textup{in},R}(a,v)\der v \der a.
\end{align}
Take $\epsilon<\frac{1}{M}$ and $R>M$. Since $a\geq c_{0}v^{\frac{2}{3}}$, we have that
\begin{align}\label{technical 2}
 \int_{(0,\infty)^{2}}vg_{\textup{in},R}(a,v)\der v \der a&   \geq \int_{(c_{0}M^{-\frac{2}{3}},\infty)\times(\frac{1}{M},M)}vg_{\textup{in}}(a,v)\der v \der a\nonumber\\
&\geq M_{0,1}(g_{\textup{in}})-M^{\gamma-1}\int_{(0,\infty)^{2}}v^{\gamma}g_{\textup{in}}(a,v)\der v \der a-M^{1-\tilde{m}}\int_{(0,\infty)^{2}}v^{\tilde{m}}g_{\textup{in}}(a,v)\der v \der a\nonumber\\
&-M^{-\frac{2}{3}}\int_{(0,\infty)^{2}}v^{\frac{1}{3}}g_{\textup{in}}(a,v)\der v \der a\geq  \frac{M_{0,1}(g_{\textup{in}})}{2}.
\end{align}
We combine (\ref{technical 1}) and (\ref{technical 2}) in order to conclude our proof.
\end{proof}

The estimates for the remaining moments follow as in the proof of Proposition \ref{lemmaweaksol}. In particular, we obtain that there exists a function $g\in\alltimesolradon\cap\omega$ and a subsequence of $\{g_{\epsilon,R}\}$ such that $g_{\epsilon,R}(t)$ converges to $g(t)$ in the weak-$^{\ast}$ topology, for every $t\geq 0$. We now want to prove that $g$ satisfies equation (\ref{regularized_form_equation_negative}). For this, we follow the steps used in the proof of Theorem \ref{themostimportanttheorem}. The difference is that now we work with an equation that depends on time instead of trying to prove the existence of a self-similar profile. Therefore, unlike Theorem \ref{themostimportanttheorem}, we do not need to obtain an estimate for $M_{1,0}$ independent of time, but it is enough to derive an estimate of the form $M_{1,0}(t)\leq m(T)$, for $0\leq t\leq T$, for some function $m:\mathbb{R}_{+}\rightarrow\mathbb{R}_{+}$, increasing, but finite for any fixed time. Notice that we can expect to have $M_{1,0}$ growing as $t\rightarrow\infty$ since this is the predicted behavior for $M_{1,0}$ in Theorem \ref{remarkramification}.
\begin{prop}\label{existencemunegative time dependent}
Let $\mu<0$. There exists a solution for the weak version of the time-dependent fusion problem as in Definition \ref{definitiontimedependent}.
\end{prop}
\begin{proof}
We test equation (\ref{regularized_form_equation_negative}) with $\varphi(\eta)=a$. The fusion term is non-positive due to the isoperimetric inequality and we have that $\langle \mathbb{K}_{R}[g_{\epsilon,R}],a\rangle=0$. Thus, we obtain:
\begin{align*}
\int_{(0,\infty)^{2}}g_
{\epsilon,R}(\eta,t)a\der \eta\leq \int_{(0,\infty)^{2}}g_{\textup{in}}(\eta)a\der \eta+ \frac{1}{3}\int_{0}^{t}\int_{(0,\infty)^{2}}g_
{\epsilon,R}(\eta,s)a\der\eta\der s.
\end{align*}
By Gr\"{o}nwall's inequality, this implies that $M_{1,0}(g_{\epsilon,R}(t))$ is bounded for every $t\geq 0$. Thus, we can redo the argument used in Theorem \ref{themostimportanttheorem} to prove that $g$ satisfies equation (\ref{weak_form_time_dependent}).
\end{proof}

We now wish to prove that $g$ satisfies the estimates in (\ref{conditions_non-existence}) in order to prove Proposition \ref{prop ramification}.
\begin{prop}\label{cond exst neg}
The solutions $g_{\epsilon,R}$ of the truncated version of the time-dependent fusion problem (\ref{regularized_form_equation_negative}) satisfy (\ref{conditions_non-existence}) uniformly in $\epsilon,R$.
\end{prop}
\begin{proof}
In order to prove the bounds in (\ref{conditions_non-existence}), we prove that the estimates hold uniformly in $\epsilon, R$ for $g_{\epsilon,R}$ and then pass to the limit. As the ideas for different moment estimates are similar, we focus on proving $M_{2,0}$ is uniformly bounded. Since it may not be possible to prove bounds for $M_{2,0}$ which are independent of time, we restrict ourselves to proving that, for some fixed $T>0$, we have that $M_{2,0}(t)$ is bounded for all $t\leq T.$ \

Let $T>0$, $t\leq T$. Due to Proposition \ref{propweaksol_generalcase} in the case $N=2$, we are able to test (\ref{regularized_form_equation_negative}) with $\varphi(a,v)=a^{2}$. Since the fusion term is non-positive due to the isoperimetric inequality, the equation becomes:
\begin{align*}
\int_{(0,\infty)^{2}}g_{\epsilon,R}(\eta,t)a^{2}\der \eta&\leq \int_{(0,\infty)^{2}}g_{\textup{in}}(\eta)a^{2}\der \eta+(1-\frac{4}{3})\int_{0}^{t}\int_{(0,\infty)^{2}}g_{\epsilon,R}(\eta,s)\Theta_{\epsilon}(v)a^{2}\der\eta\der s\\
&+C\int_{0}^{t}\int_{(0,\infty)^{2}}av^{-\alpha}g_{\epsilon,R}(\eta,s)\der\eta\int_{(0,\infty)^{2}} a'v'^{\beta}g_{\epsilon,R}(\eta',s)\der \eta'\der s\\
&\leq  \int_{(0,\infty)^{2}}g_{\textup{in}}(\eta)a^{2}\der \eta+C\int_{0}^{t}\int_{(0,\infty)^{2}}av^{-\alpha}g_{\epsilon,R}(\eta,s)\der\eta\int_{(0,\infty)^{2}} a'v'^{\beta}g_{\epsilon,R}(\eta',s)\der \eta'\der s
\end{align*}
Thus, if $M_{1,-\alpha}$ and $M_{1,\beta}$ are bounded, we obtain that:
\begin{align*}
    \int_{(0,\infty)^{2}}g_{\epsilon,R}(\eta,t)a^{2}\der \eta\leq \int_{(0,\infty)^{2}}g_{\textup{in}}(\eta)a^{2}\der \eta+C(t)\leq \int_{(0,\infty)^{2}}g_{\textup{in}}(\eta)a^{2}\der \eta+C(T)
\end{align*}
and we can conclude using Gr\"{o}nwall's inequality.

$M_{1,-\alpha}$ is bounded in the same manner as $M_{1,0}$ in Proposition \ref{existencemunegative time dependent}.

We now estimate $M_{1,\beta}$. The contribution of the fusion term is non-positive. We estimate the term containing the coagulation kernel using:
\begin{align*}
    \int_{(0,\infty)^{2}}&K_{\epsilon,R}(a,v,a',v')[(a+a')(v+v')^{\beta}-av^{\beta}-a'v'^{\beta}]g_{\epsilon,R}(\eta,t)g_{\epsilon,R}(\eta',t')\der \eta'\der \eta\\
    &\leq\int_{(0,\infty)^{2}}K_{\epsilon,R}(a,v,a',v')[av'^{\beta}+a'v^{\beta}]g_{\epsilon,R}(\eta,t)g_{\epsilon,R}(\eta',t')\der \eta'\der \eta\\
    &\leq C\int_{(0,\infty)^{2}}av^{\beta}g(\eta,t)\der \eta\int_{(0,\infty)^{2}}v'^{\beta-\alpha}g(\eta',t')\der \eta'+C\int_{(0,\infty)^{2}}av^{-\alpha}g(\eta,t)\der \eta\int_{(0,\infty)^{2}}v'^{2\beta}g(\eta',t')\der \eta'.
\end{align*}
Since we can bound from above the moment estimates of the form $M_{0,d},$ with $d=2\beta$ or $d=\beta-\alpha$, when testing (\ref{regularized_form_equation_negative}) with $av^{\beta}$, the equation becomes
\begin{align*}
    \partial_{t}M_{1,\beta}(g_{\epsilon,R}(t))&\leq (1-\frac{2}{3}-\beta)\int_{(0,\infty)^{2}}g_{\epsilon,R}(\eta,t)\Theta_{\epsilon}(v)av^{\beta}\der\eta+C_{1}M_{1,\beta}(g_{\epsilon,R}(t))+C_{2}M_{1,-\alpha}(g_{\epsilon,R}(t))\\
    &\leq \overline{C}_{1}M_{1,\beta}(g_{\epsilon,R}(t))+\overline{C}_{2}(t),
\end{align*}
for some constants $C_{1},C_{2},\overline{C}_{1},\overline{C}_{2}(t)>0$, where $\overline{C}_{2}$ is a function $\overline{C}_{2}:\mathbb{R}_{+}\rightarrow\mathbb{R}_{+}$, increasing, but finite for any fixed time. Thus, for $t\in[0,T]$, we have
\begin{align*}
    \partial_{t}M_{1,\beta}(g_{\epsilon,R}(t))    \leq \overline{C}_{1}M_{1,\beta}(g_{\epsilon,R}(t))+\overline{C}_{2}(T).
\end{align*}
We can conclude again using Gr\"{o}nwall's inequality.

The rest of the moments in (\ref{conditions_non-existence}) are estimated using the same methods. The main idea is to use an iterative argument that allows to reduce the exponents of the powers of $a$ and $v$ to lower order terms, for which we already obtained uniform bounds.

For example, we analyse the moment $M_{2,m},$ when $m>1$. The fusion term is non-positive and the linear terms in (\ref{regularized_form_equation_negative}) satisfy  $(1-\frac{4}{3}-m)\int_{(0,\infty)^{2}}g_{\epsilon,R}(\eta,t)\Theta_{\epsilon}(v)a^{2}v^{m}\der\eta\leq 0$. We thus obtain:
\begin{align}\label{moment 2 m}
 \partial_{t} M_{2,m}(g_{\epsilon,R}(t))&\lesssim
\int_{(0,\infty)^{2}}\int_{(0,\infty)^{2}}K_{\epsilon,R}(\eta,\eta')[(a+a')^{2}(v+v')^{m}-a^{2}v^{m}-a'^{2}v'^{m}]g_{\epsilon,R}(\eta,t)g_{\epsilon,R}(\eta',t)\der \eta' \der \eta\nonumber\\
&\leq \int_{(0,\infty)^{2}}\int_{(0,\infty)^{2}}K(\eta,\eta')[(a+a')^{2}(v+v')^{m}-a^{2}v^{m}-a'^{2}v'^{m}]g_{\epsilon,R}(\eta,t)g_{\epsilon,R}(\eta',t)\der \eta' \der \eta
\end{align}
and use the fact that
\begin{align*}
    (a+a')^{2}(v+v')^{m}-a^{2}v^{m}-a'^{2}v'^{m}&\leq [(a+a')^{2}-a^{2}-a'^{2}](v+v')^{m}+(a^{2}+a'^{2})(v+v')^{m}-a^{2}v^{m}-a'^{2}v'^{m}\\
    &\leq 2aa'(v+v')^{m}+(a^{2}+a'^{2})[(v+v')^{m}-v^{m}-v'^{m}]\\
    &+(a^{2}+a'^{2})(v^{m}+v'^{m})-a^{2}v^{m}-a'^{2}v'^{m}\\
    &\leq 2aa'(v+v')^{m}+\sum_{k=1}^{k_{m}}\binom{m}{k}(a^{2}+a'^{2})(v^{k}v'^{m-k}+v^{m-k}v'^{k})\\
    &+a^{2}v'^{m}+a'^{2}v^{m}.
\end{align*}
The terms containing powers of $a$ of lower order, such as $av^{k},$ $k\in[0,m]$, are bounded iteratively. The rest of the terms are of the form $a^{2}v^{b}$, with $b-\alpha, b+\beta\in(-\alpha,m)$ or with $b-\alpha=-\alpha$. For $b-\alpha, b+\beta\in(-\alpha,m)$, we can use H\"{o}lder's inequality for the integral containing the coagulation kernel:
\begin{align}\label{holder 2 m}
\int_{(0,\infty)^{2}}g_{\epsilon,R}(\eta,t)a^{2}v^{b+\beta}\der \eta \leq \bigg(\int_{(0,\infty)^{2}}g_{\epsilon,R}(\eta,t)a^{2}v^{m}\der \eta\bigg)^{\theta} \bigg(\int_{(0,\infty)^{2}}g_{\epsilon,R}(\eta,t)a^{2}v^{-\alpha}\der \eta\bigg)^{1-\theta},
\end{align}
for some $\theta\in(0,1)$. We can estimate $M_{2,-\alpha}$ in a completely analogous manner as  $M_{2,0}$. The moments containing only powers of $v$ are bounded from above.

Thus, we combine (\ref{moment 2 m}) and (\ref{holder 2 m}) to deduce that
\begin{align*}
 \partial_{t} M_{2,m}(g_{\epsilon,R}(t))&\lesssim
C_{1}(t)\bigg(\int_{(0,\infty)^{2}}g_{\epsilon,R}(\eta,t)a^{2}v^{m}\der \eta\bigg)^{\theta} +C_{2}(t)\\
&\leq 
C_{1}(T)\bigg(\int_{(0,\infty)^{2}}g_{\epsilon,R}(\eta,t)a^{2}v^{m}\der \eta\bigg)^{\theta} +C_{2}(T),
\end{align*}
for some $C_{1},C_{2}:\mathbb{R}_{+}\rightarrow\mathbb{R}_{+}$, increasing, but finite for any fixed time.  We conclude using Gr\"{o}nwall's inequality.
\end{proof}
In Proposition \ref{cond exst neg}, we derived uniform estimates in  $\epsilon,R$ for the moments in (\ref{conditions_non-existence}) of $g_{\epsilon,R}$. This means that the limit $g$ will satisfy the same moment estimates.
\begin{rmk}\label{remark needed for proposition}
$g$ satisfies the estimates in (\ref{conditions_non-existence}).
\end{rmk}

\begin{proof}[Proof of Proposition \ref{prop ramification}]
 Proposition \ref{prop ramification} follows by combining Propositions \ref{lowerbound} - \ref{cond exst neg} and Remark \ref{remark needed for proposition}.
\end{proof}
We now prove a technical proposition which shows that we can test (\ref{weak_form_time_dependent}) with $\varphi\equiv a$.
\begin{prop}\label{extension compact support}
Let $T>0$. Let $\mu<0$. Then equation (\ref{weak_form_time_dependent}) holds for every $\varphi\in\textup{C}^{1}(\mathbb{R}_{>0}^{2})$ with $\sup_{\eta\in\mathbb{R}_{>0}^{2}}|\varphi(\eta)+a\partial_{a}\varphi(\eta)|\leq Ca$ and $\sup_{\eta\in\mathbb{R}_{>0}^{2}}|\partial_{v}\varphi(\eta,t)|\leq C$ if (\ref{conditions_non-existence}) holds.
\end{prop}
\begin{proof}
Let $T>0$, $t\in[0,T]$. We construct a sequence of functions $\{\zeta_{n}\}_{n\in\mathbb{N}}\subset \textup{C}_{\textup{c}}^{1}(\mathbb{R}_{>0}^{2})$ such that $\zeta_{n}(\eta)=1$ when $\eta\in[\frac{1}{n},n]^{2}$ and $\zeta_{n}(\eta)=0$ when  $\eta\not\in[\frac{1}{2n},2n]^{2}$. The idea is to use Lebesgue's dominated convergence theorem in (\ref{weak_form_time_dependent}) for the functions $\varphi_{n}=\zeta_{n}\varphi$. We thus show below only the needed estimates for the proof. The term with the coagulation kernel in (\ref{weak_form_time_dependent}) can be bounded directly by
\begin{align*}
    \big|\langle \mathbb{K}[g](t),\varphi_{n}\rangle \big|\leq C\sup_{s\in[0,T]} M_{1,-\alpha}(g(s))\sup_{s\in[0,T]} M_{0,\beta}(g(s))+C\sup_{s\in[0,T]} M_{0,-\alpha}(g(s))\sup_{s\in[0,T]} M_{1,\beta}(g(s))\leq C(T).
\end{align*}
We now wish to control the fusion term in (\ref{weak_form_time_dependent}), namely
  \begin{align*}
\int_{(0,\infty)^{2}}|r(\eta)(c_{0}v^{\frac{2}{3}}-a)\partial_{a}\varphi_{n}(\eta)|g(\eta,s)\der \eta.
 \end{align*}

Notice that we can construct $\zeta_{n}$ such that $|a\partial_{a}\zeta_{n}(\eta)|\leq C,$ for some constant independent of $n\in\mathbb{N}.$ Moreover, we know that the fusion kernel satisfies (\ref{fusion_form}) and that $a\geq c_{0}v^{\frac{2}{3}}$. Thus
 \begin{align}\label{using the above inequality}
     |r(a,v)(c_{0}v^{\frac{2}{3}}-a)|\lesssim a^{\mu+1}v^{\sigma}+a^{\mu}v^{\sigma+\frac{2}{3}}\leq 2 a^{\mu+1}v^{\sigma}\leq 2c_{0}^{\mu}av^{\gamma-1}.
 \end{align} 
 Using the above inequality, we find that the following upper bound holds
 \begin{align*}
|r(\eta)(c_{0}v^{\frac{2}{3}}-a)\partial_{a}\varphi_{n}(\eta)|\lesssim av^{\gamma-1}|\partial_{a}\varphi(\eta)|+av^{\gamma-1}|\varphi(\eta)\partial_{a}\zeta_{n}(\eta)|
\end{align*}
up to a multiplicative constant. We then use Young's inequality to deduce that
\begin{align*}
    av^{\gamma-1}|\partial_{a}\varphi(\eta)|\lesssim av^{\beta}+av^{-\alpha-1}
\end{align*}
and
\begin{align*}
av^{\gamma-1}|\varphi(\eta)\partial_{a}\zeta_{n}(\eta)|\leq Cav^{\gamma-1}\lesssim av^{\beta}+av^{-\alpha-1}.
\end{align*}
Since we have that
\begin{align*}
\int_{(0,\infty)^{2}}(av^{\beta}+av^{-\alpha-1})g(\eta,s)\der \eta\leq\sup_{s\in[0,T]}M_{1,\beta}(g(s))+\sup_{s\in[0,T]}M_{1,-\alpha-1}(g(s))<\infty,
\end{align*}
the moment estimates in (\ref{conditions_non-existence}) suffice in order to conclude our proof.
\end{proof}
We now have all the necessary parts in order to prove Theorem \ref{remarkramification}. We first prove that the total area goes to infinity as $t\rightarrow\infty$. More precisely, we will obtain a lower bound for the total surface area that increases exponentially in time. 
\begin{proof}[Proof of Theorem \ref{remarkramification}]
From Proposition \ref{extension compact support}, we notice that the estimates (\ref{conditions_non-existence}) on $g$ allow us to test with $\varphi(a,v,t)=a$.

Since $\langle \mathbb{K}[g],a\rangle=0$ and using (\ref{fusion_form}), equation (\ref{weak_form_time_dependent}) becomes
\begin{align*}
\partial_{t}\int_{(0,\infty)^{2}}g(\eta,t)a\der \eta&=\int_{(0,\infty)^{2}}g(\eta,t)a\der\eta -\frac{2}{3}\int_{(0,\infty)^{2}}g(\eta,t)  a\der\eta+(1-\gamma)\int_{(0,\infty)^{2}} r( \eta)(c_{0}v^{\frac{2}{3}}-a)g(\eta,t)\der \eta \\
&\geq\frac{1}{3}\int_{(0,\infty)^{2}}g(\eta,t)a\der\eta-(1-\gamma)R_{1}\int_{(0,\infty)^{2}} a^{-|\mu|+1}v^{\sigma}g(\eta,t)\der \eta .
\end{align*}
We distinguish now two cases:
\begin{itemize}
    \item Case 1:
    
\
If $\mu\in(-1,0),$ we can apply Young's inequality to obtain:
\begin{align}\label{young_improv}
    a^{1+\mu}v^{\sigma}\leq\frac{\epsilon}{p}a+\frac{\epsilon^{-\frac{q}{p}}}{q}v^{\frac{\sigma}{|\mu|}},
\end{align}
where $p=\frac{1}{1-|\mu|}$ and $q=\frac{1}{|\mu|}$. Thus, there exists a constant $C_{\epsilon}>0$ for which:
\begin{align*}
    -\int_{(0,\infty)^{2}}a^{\mu+1}v^{\sigma}g(a,v,t)\der v \der a\gtrsim -\frac{\epsilon}{p}\int_{(0,\infty)^{2}}ag(a,v,t)\der v \der a - C_{\epsilon}\int_{(0,\infty)^{2}}v^{-\frac{\sigma}{\mu}}g(a,v,t)\der v\der a.
\end{align*}
Choose $\epsilon(R_{1})$ such that $(1-\gamma)R_{1}\frac{\epsilon(R_{1})}{p}<\tilde{\delta},$ for some $\tilde{\delta}\in(0,\frac{1}{3})$. Then, condition $(\ref{cond2_non-existence})$ implies:
\begin{align}\label{delta tilde}
\partial_{t}\int_{\mathbb{R}^{2}_{>0}}g(\eta,t)a\der \eta\geq&(\frac{1}{3}-\tilde{\delta})\int_{\mathbb{R}^{2}_{>0}}g(\eta,t)a\der\eta-\tilde{C}_{\epsilon},
\end{align}
where $\tilde{C}_{\epsilon}=(1-\gamma)R_{1}\frac{\epsilon(R_{1})^{-\frac{q}{p}}}{q}\sup_{t\geq 0}M_{0,\frac{\sigma}{|\mu|}}(g(t))$.
\item Case 2:

\
If $\mu\leq -1$, using (\ref{same_rescaling}), we have that  $a^{1+\mu}v^{\sigma}\leq c_{0}^{1+\mu}v^{\gamma-\frac{1}{3}}$ and the equation for $g$ becomes:
\begin{align*}
\partial_{t}\int_{\mathbb{R}^{2}_{>0}}g(\eta,t)a\der \eta\geq\frac{1}{3}\int_{\mathbb{R}^{2}_{>0}}g(\eta,t)a\der\eta -C\int_{(0,\infty)^{2}}v^{\gamma-\frac{1}{3}}g(\eta,t)\der \eta.
\end{align*}
The moment $M_{0,\gamma-\frac{1}{3}}$ is bounded from above, cf. (\ref{gamma-1pe3}), so we are in the same situation as in Case 1.
\end{itemize}

We define $C(R_{1}):=\frac{(1-\gamma)R_{1}}{\frac{1}{3}-\tilde{\delta}}\frac{\epsilon(R_{1})^{-\frac{q}{p}}}{q}\sup_{t\geq 0}M_{0,\frac{\sigma}{|\mu|}}(g(t))$, where $\tilde{\delta}$ is as in (\ref{delta tilde}), when $\mu\in(-1,0)$ and $C(R_{1}):=3(1-\gamma)R_{1}c_{0}^{1+\mu}\sup_{t\geq 0}M_{0,\gamma-\frac{1}{3}}(g(t)),$ when $\mu\leq -1$. 

We have that, if $\mu\in(-1,0)$, 
\begin{align*}
\partial_{t}\int_{\mathbb{R}^{2}_{>0}}g(\eta,t)a\der \eta\geq&(\frac{1}{3}-\tilde{\delta})\int_{\mathbb{R}^{2}_{>0}}g(\eta,t)a\der\eta-(\frac{1}{3}-\tilde{\delta})C(R_{1}).
\end{align*}
If $\mu\leq -1$, we have
\begin{align*}
\partial_{t}\int_{\mathbb{R}^{2}_{>0}}g(\eta,t)a\der \eta\geq&\frac{1}{3}\int_{\mathbb{R}^{2}_{>0}}g(\eta,t)a\der\eta-\frac{1}{3}C(R_{1}).
\end{align*}
Thus, if we choose an initial condition such that
\begin{align*}
    \int_{(0,\infty)^{2}} a g_{\text{in}}(a,v)\der v\der a\geq 2C(R_{1})
    \end{align*}
    then
    \begin{align}\label{initialexponential}
   \int_{(0,\infty)^{2}}ag(a,v,t)\der v\der a\geq e^{(\frac{1}{3}-\tilde{\delta})t}C(R_{1}),
\end{align}
if $\mu\in(-1,0)$ and
\begin{align*}
    \int_{(0,\infty)^{2}}ag(a,v,t)\der v\der a\geq e^{\frac{1}{3}t}C(R_{1}),
\end{align*}
when $\mu\leq -1$.

We now prove that we can improve the lower bound found in (\ref{initialexponential}) by removing the $\tilde{\delta}$ in the exponential. This will conclude the proof of Theorem \ref{remarkramification}.

 In the following, we do not write explicitly the constants that are not necessarily relevant for the proof. We want to improve the lower bound in (\ref{young_improv}) in the case when $\mu\in(-1,0)$. For $\mu\in(-1,0)$, we let $\epsilon$ to decrease to $0$ as $t\rightarrow\infty$  in (\ref{young_improv}).
 
 More precisely, we take a function $\epsilon(t):=\epsilon_{1}\textup{e}^{-\epsilon_{1}\frac{|\mu|}{1-|\mu|}t}\in(0,\epsilon_{1}],$ for some constant
\begin{align}\label{defepsilon1}
  \epsilon_{1}\in(0,\frac{1}{12}).
\end{align}
This means that $\epsilon(t)^{-\frac{q}{p}}=\epsilon_{1}^{-\frac{q}{p}}\textup{e}^{\epsilon_{1}t}$ in (\ref{young_improv}) and thus:
\begin{align*}
       a^{1+\mu}v^{\sigma}\leq\frac{\epsilon(t)}{p}a+C_{\epsilon_{1},\mu}\frac{\textup{e}^{\epsilon_{1}t}}{q}v^{\frac{\sigma}{|\mu|}},
\end{align*}
where $C_{\epsilon_{1},\mu}$ is a constant depending on $\epsilon_{1}$ and $\mu$. We denote $A(t):=\int_{\mathbb{R}^{2}_{>0}}g(\eta,t)a\der \eta$ and equation (\ref{weak_form_time_dependent}) then becomes: 
\begin{align}\label{eq11}
\partial_{t}A(t)\geq&\bigg(\frac{1}{3}-\epsilon(t)\bigg)A(t)-\tilde{C}\textup{e}^{\epsilon_{1}t},
\end{align}
where $\tilde{C}$ is a constant depending on $\epsilon_{1}$, $\mu$ and $M_{0,\frac{\sigma}{|\mu|}}(g(t))$. We note that $\int_{0}^{t}\epsilon(s)\der s=\frac{1-|\mu|}{|\mu|}\bigg(1-\textup{e}^{-\epsilon_{1}\frac{|\mu|}{1-|\mu|}t}\bigg)$. 

Since $\epsilon(t)\leq\epsilon_{1}$, for every $t\geq 0$, we have that $\int_{0}^{t}\epsilon(s)\der s\leq\epsilon_{1}t$. Multiplying (\ref{eq11}) with $\textup{e}^{-\frac{1}{3}t+\int_{0}^{t}\epsilon(s)\der s}$, we obtain:
\begin{align}\label{improv_epsilon}
\textup{e}^{-\frac{1}{3}t+\int_{0}^{t}\epsilon(s)\der s}\bigg[\partial_{t}A(t)-(\frac{1}{3}-\epsilon(t))A(t)\bigg]\geq-\tilde{C}\textup{e}^{\epsilon_{1}t}\textup{e}^{-\frac{1}{3}t+\int_{0}^{t}\epsilon(s)\der s}\geq-\tilde{C}\textup{e}^{2\epsilon_{1}t}\textup{e}^{-\frac{1}{3}t}\geq -\tilde{C}\textup{e}^{-\frac{1}{6}t} ,
\end{align}
since $\epsilon_{1}$ satisfies (\ref{defepsilon1}). Then we integrate in time from $0$ to $t$:
\begin{align*}
\textup{e}^{-\frac{1}{3}t+\int_{0}^{t}\epsilon(s)\der s}A(t)-A(0)\geq 6\tilde{C}\textup{e}^{-\frac{1}{6}t}-6\tilde{C}
\end{align*}
and thus
\begin{align*}
A(t)\geq A(0)\textup{e}^{\frac{1}{3}t-\int_{0}^{t}\epsilon(s)\der s}-6\tilde{C}\textup{e}^{\frac{1}{3}t-\int_{0}^{t}\epsilon(s)\der s}+6\tilde{C}\textup{e}^{\frac{1}{6}t-\int_{0}^{t}\epsilon(s)\der s}\geq A(0)\textup{e}^{\frac{1}{3}t-\int_{0}^{t}\epsilon(s)\der s}-6\tilde{C}\textup{e}^{\frac{1}{3}t-\int_{0}^{t}\epsilon(s)\der s}.
\end{align*}
Using the definition of $A(t)$, this can be written as
\begin{align*}
\int_{\mathbb{R}^{2}_{>0}}g(\eta,t)a\der \eta\geq\bigg(\int_{\mathbb{R}^{2}_{>0}}g(\eta,0)a\der \eta-6\tilde{C}\bigg)\textup{e}^{\frac{1}{3}t-\int_{0}^{t}\epsilon(s)\der s}.
\end{align*}
This implies that, if we start with sufficiently large total surface area, $\int_{\mathbb{R}^{2}_{>0}}g(\eta,0)a\der \eta\geq 12\tilde{C}$, we have:
\begin{align*}
    \int_{\mathbb{R}^{2}_{>0}}g(\eta,t)a\der \eta\geq6\tilde{C}\textup{e}^{\frac{1}{3}t-\int_{0}^{t}\epsilon(s)\der s}\gtrsim C \textup{e}^{\frac{1}{3}t} ,
\end{align*}
where we used that $1-\textup{e}^{-\epsilon_{1}\frac{|\mu|}{1-|\mu|}t}\in[0,1]$ and thus $\textup{e}^{-\int_{0}^{t}\epsilon(s)\der s}\geq \textup{e}^{-\frac{1-|\mu|}{|\mu|}}$. 

This concludes the proof of Theorem \ref{remarkramification}.
\end{proof}
\subsection{Self-similarity in the case of fast fusion}
In this subsection we prove Theorem \ref{self sim mu neg}.
\bigskip

We look at the following truncated version of equation (\ref{weak_form_time_dependent}):
\begin{align}\label{regularizedformequation mu negative}
\int_{(0,\infty)^{2}}g(\eta,t)\Theta_{\epsilon}(v)\varphi(\eta)\der\eta-\frac{2}{3}\int_{(0,\infty)^{2}}g(\eta,t)\tilde{\Theta}_{\epsilon}(a,v) a\partial_{a}\varphi(\eta)\der\eta-\int_{(0,\infty)^{2}}&g(\eta,t)\Theta_{\epsilon}(v)v\partial_{v}\varphi(\eta)\der \eta \nonumber \\
+(1-\gamma)\langle \mathbb{K}_{\epsilon,R}[g],\varphi\rangle+(1-\gamma)\int_{(0,\infty)^{2}}r(\eta)(c_{0}v^{\frac{2}{3}}-a)g(\eta,t)\partial_{a}\varphi(\eta)\der \eta =\partial_{t}&\int_{(0,\infty)^{2}}g(\eta,t)\varphi(\eta)\der \eta,
\end{align}
where $\mathbb{K}_{\epsilon,R}$ is the term in (\ref{kernel_term}) with coagulation kernel $K_{\epsilon,R}$ defined in (\ref{truncation_kernel_selfsimilar}),  $\Theta_{\epsilon}$ was defined in (\ref{thetaepsilon}) and $\varphi\in\textup{C}^{1}_{\textup{c}}(\mathbb{R}_{>0}^{2})$. We define a continuous function $\tilde{\Theta}_{\epsilon}:\mathbb{R}^{2}_{>0}\rightarrow [0,1]$ such that:
\begin{align}\label{new cut off}
  \tilde{\Theta}_{\epsilon}(a,v)=
  \begin{cases}
0 & a\leq1-\epsilon \textup{ and }v \leq\epsilon;\\
    \Theta_{\epsilon}(v) &a=c_{0}v^{\frac{2}{3}};\\
  1& a>1 \textup{ or } v \geq 2\epsilon,
  \end{cases}
\end{align}
which satisfies $\tilde{\Theta}_{\epsilon}(a,v)\geq \Theta_{\epsilon}(v).$ The choice in (\ref{new cut off}) was made in order for (\ref{regularizedformequation mu negative}) to preserve the isoperimetric inequality in time. This means that if $g$ solves (\ref{regularizedformequation mu negative}) and $g(0)\in\mathscr{M}_{+}^{I}(\mathbb{R}_{>0}^{2}),$ then $g(t)\in\mathscr{M}_{+}^{I}(\mathbb{R}_{>0}^{2})$, for any $t\geq 0$. Additionally, $\tilde{\Theta}_{\epsilon}$ does not vanish in the region $\{a>1\}$. This form will simplify the proof of the moment estimates.

Notice that, differently from Section \ref{sectionexistence}, we do not truncate the fusion rate $r(\eta)$ since $r(\eta)(c_{0}v^{\frac{2}{3}}-a)$ grows at most linearly.

The statement of Proposition \ref{propweaksol_generalcase} holds if we replace equation  (\ref{regularizedformequation}) with (\ref{regularizedformequation mu negative}). This is obtained using the same proof and the growth rate of $r(\eta)(c_{0}v^{\frac{2}{3}}-a)$ in (\ref{regularizedformequation mu negative}).     Thus, we prove the existence of solutions of (\ref{regularizedformequation mu negative}) as in Proposition \ref{propweaksol_generalcase} with $N=2$, which we will denote by $g_{\epsilon,R}$. The choice $N=2$ was made since we want to test equation (\ref{regularizedformequation mu negative}) with $a^{2}$ in order to obtain uniform moment bounds later in the proof.

To prove Theorem \ref{self sim mu neg} we will follow the strategy of Theorem \ref{themostimportanttheorem}. The main point is to find an invariant set that allows us to pass to the limit as $\epsilon\rightarrow 0$ and $R\rightarrow\infty$ in the two-dimensional coagulation equation. In this case, it will be harder to estimate the moment $M_{1,0}$ than in the case when $\mu>0$. This is because of the fact that, in the case of (\ref{regularizedformequation}), the moment estimates are a consequence of the fast growth of the fusion ratio $r(\eta)$. This fast growth does not take place in (\ref{regularizedformequation mu negative}). We will be able to replace the fast growth of $r(\eta)$ with the choice of a sufficiently large constant $\lambda$. 

We first prove some elementary inequalities that will be useful for the proof of Theorem \ref{self sim mu neg}.
\begin{prop}\label{propinequalities}
Let $\mu\leq 0$. For  $\delta_{1}\in(0,1)$, there exist $\lambda$ sufficiently large (depending on $\delta_{1}$) and some constants $m_{1}=\alpha(\mu-2)-\sigma,$ $m_{2}=\beta(2-\mu)-\sigma$ such that:
\begin{align}
av^{-\alpha}&\leq \frac{1-\gamma}{6}\lambda a^{\mu+1}v^{\sigma-\alpha}+\delta_{1} a^{2} + v^{m_{1}};\label{lambda big alpha}\\
av^{\beta}&\leq \frac{1-\gamma}{2(K_{0}c_{0,\beta-\alpha}+2)}\lambda a^{\mu+1}v^{\sigma+\beta}+\delta_{1} a^{2} +v^{m_{2}}, \label{lambda big beta}
\end{align}
where $c_{0,\beta-\alpha}$ is a constant that will be fixed in Lemma \ref{lemmaweaksol mu neg} and $K_{0}$ is as in (\ref{lower_bound_kernel}).
\end{prop}
\begin{proof}
 We prove (\ref{lambda big alpha}) in detail. The strategy to prove (\ref{lambda big alpha}) and (\ref{lambda big beta}) is the same. The proof is as follows.

For any $\delta_{2},\delta_{3}\in(0,1)$, there exist $C_{\delta_{2}}>0$ and $C_{\delta_{3}}>0$ such that
\begin{align*}
av^{-\alpha}\leq C_{\delta_{2}}a^{\frac{\mu+1}{2}}v^{-\alpha\frac{3-\mu}{2}}+\delta_{2}a^{2}
\end{align*}
and
\begin{align*}
a^{\frac{\mu+1}{2}}v^{-\alpha\frac{3-\mu}{2}}\leq C_{\delta_{3}}a^{\mu+1}v^{\sigma-\alpha}+\delta_{3}v^{m_{1}},
\end{align*}
with $m_{1}=\alpha(\mu-2)-\sigma$. Combining the two inequalities, we obtain
\begin{align*}
av^{-\alpha}\leq C_{\delta_{2}}C_{\delta_{3}}a^{\mu+1}v^{\sigma-\alpha}+C_{\delta_{2}}\delta_{3}v^{m_{1}}+\delta_{2}a^{2}.
\end{align*}
To conclude the argument, we choose $\lambda$ sufficiently large such that $\frac{1-\gamma}{6}\lambda\geq C_{\delta_{2}}C_{\delta_{3}}$ and choose $\delta_{3}$ such that $C_{\delta_{2}}\delta_{3}\leq 1$.
\end{proof}
We now make use of Proposition \ref{propinequalities} to prove that we have uniform bounds in $\epsilon$ and $R$ for the solutions of equation (\ref{regularizedformequation mu negative}).
\begin{lem}\label{lemmaweaksol mu neg}
Let $K_{\epsilon,R}$ be as in (\ref{truncation_kernel_selfsimilar}) and assume $r$ satisfies the assumptions of Theorem \ref{self sim mu neg}. Let $T>0$. We define a semigroup $S(t): U_{\epsilon,R,2}\rightarrow U_{\epsilon,R,2}$ in the following way: $S(t)g_{\textup{in},R}=g_{\epsilon,R}(\cdot,\cdot,t)$, for all $t\in[0,T]$ and with $U_{\epsilon,R,2}$ defined as in (\ref{existence_space_strong_general}) with $N=2$. Take $\mu\leq 0$, fix $\tilde{\epsilon}\in(0,1)$, let  $m_{i}\in\mathbb{R}$, $i=\overline{1,2}$, to be the coefficients found in Proposition \ref{propinequalities} and denote $\tilde{\alpha}:=\alpha+\tilde{\epsilon}$. Then there exist constants $c_{0,-\tilde{\alpha}},$ $c_{0,m_{i}},$ $c_{2,0},$ $c_{1,-\alpha},$ $c_{1,\beta}>0$, $i\in \{1,2\}$, for which the set
\begin{align*}
    \tilde{\omega}=&\{ M_{0,1}(g_{\epsilon,R})=1; M_{0,-\tilde{\alpha}}(g_{\epsilon,R})\leq c_{0,-\tilde{\alpha}}; M_{0,2}(g_{\epsilon,R})\leq c_{0,2};
    M_{0,m_{i}}(g_{\epsilon,R})\leq c_{0,m_{i}};\\
    &M_{2,0}(g_{\epsilon,R})\leq c_{2,0}; 
    M_{1,-\alpha}(g_{\epsilon,R})\leq c_{1,-\alpha}; M_{1,\beta}(g_{\epsilon,R})\leq c_{1,\beta}\},
\end{align*}
$i\in\{1,2\}$, is preserved in time uniformly in $\epsilon,R$ under equation (\ref{regularizedformequation mu negative}), or equivalently $S(t)\tilde{\omega}\subseteq\tilde{\omega},$ for all $t\in[0,T]$.
\end{lem}
\begin{proof}
Uniform bounds for moments of the form $M_{0,d},$ with $d\in\mathbb{R}$, are derived as in Propositions \ref{firstmoments} - \ref{m-lprop}. We now find uniform bounds for $M_{2,0}(g),$ $M_{1,-\alpha}(g)$ and $M_{1,\beta}(g)$. For $\varphi\in\textup{C}^{1}(\mathbb{R}^{2}_{>0})$, whose growth in $a$ is of order $1+a^{2}$  and with $\partial_{a}\varphi(a,v)\geq 0$, we have that:
\begin{align*}
\partial_{t}&\int_{(0,\infty)^{2}}g(\eta,t)\varphi(\eta)\der \eta\leq \int_{(0,\infty)^{2}}g(\eta,t)\Theta_{\epsilon}(v)\varphi(\eta)\der\eta-\frac{2}{3}\int_{(0,\infty)^{2}}g(\eta,t)\tilde{\Theta}_{\epsilon}(a,v) a\partial_{a}\varphi(\eta)\der\eta\\
&-\int_{(0,\infty)^{2}}g(\eta,t)\Theta_{\epsilon}(v)v\partial_{v}\varphi(\eta)\der \eta+(1-\gamma)\langle \mathbb{K}_{\epsilon,R}[g],\varphi\rangle
+\lambda(1-\gamma)\int_{(0,\infty)^{2}}a^{\mu}v^{\sigma}(c_{0}v^{\frac{2}{3}}-a)g(\eta,t)\partial_{a}\varphi(\eta)\der \eta.
\end{align*}
Thus we obtain:
\begin{align}
\partial_{t} M_{2,0}(g(t))&\leq -\frac{1}{3}\int_{(0,\infty)^{2}}g(\eta,t)a^{2}\der\eta+\frac{4}{3}\int_{\{a\leq 1\}}g(\eta,t)a^{2}\der\eta\nonumber\\
&+2K_{0}(1-\gamma)\int_{(0,\infty)^{2}}g(\eta,t)av^{-\alpha}\der\eta\int_{(0,\infty)^{2}}g(\eta',t)a'v'^{\beta}\der\eta'. \nonumber
\\
\partial_{t} M_{1,-\alpha}(g(t))&\leq (\frac{1}{3}+\alpha)\int_{(0,\infty)^{2}}g(\eta,t)\Theta_{\epsilon}(v)av^{-\alpha}\der\eta+(1-\gamma)\lambda\int_{(0,\infty)^{2}}a^{\mu}v^{\sigma-\alpha}(c_{0}v^{\frac{2}{3}}-a)g(\eta,t)\der \eta.\nonumber\\
\partial_{t}M_{1,\beta}(g(t)) &\leq (\frac{1}{3}-\beta)\int_{(0,\infty)^{2}}g(\eta,t)\Theta_{\epsilon}(v)av^{\beta}\der\eta+(1-\gamma)\lambda\int_{(0,\infty)^{2}}a^{\mu}v^{\sigma+\beta}(c_{0}v^{\frac{2}{3}}-a)g(\eta,t)\der \eta \nonumber \\
&+K_{0}(1-\gamma)\int_{(0,\infty)^{2}}g(\eta,t)av^{\beta}\der\eta\int_{(0,\infty)^{2}}g(\eta',t)v'^{\beta-\alpha}\der\eta'\nonumber \\
&+K_{0}(1-\gamma)\int_{(0,\infty)^{2}}g(\eta,t)av^{-\alpha}\der\eta\int_{(0,\infty)^{2}}g(\eta',t)v'^{2\beta}\der\eta'. \label{equations}
\end{align}
Notice that in (\ref{equations}) the truncation on the kernel does not create problems: for $M_{1,-\alpha}$ the term with the coagulation kernel is non-positive and for $M_{2,0}, M_{1,\beta}$ we used $K_{\epsilon,R}\leq K$. 

Making use of (\ref{lambda big alpha}) and (\ref{lambda big beta}), (\ref{equations}) becomes:
\begin{align}
\partial_{t} M_{2,0}(g(t))&\leq -\frac{1}{3}M_{2,0}(g(t))+\frac{4}{3}M_{0,0}(g(t))+2K_{0}(1-\gamma)M_{1,-\alpha}(g(t))M_{1,\beta}(g(t)). \nonumber
\\
\partial_{t} M_{1,-\alpha}(g(t))&\leq (\frac{1}{3}+\alpha)M_{1,-\alpha}(g(t))-\frac{(1-\gamma)\lambda}{2}\int_{\{a\geq2c_{0}v^{\frac{2}{3}}\}}a^{\mu+1}v^{\sigma-\alpha}g(\eta,t)\der \eta\nonumber\\
&\leq -M_{1,-\alpha}(g(t))+6c_{0}M_{0,\frac{2}{3}-\alpha}(g(t))\nonumber\\
&+3\int_{\{a\geq2c_{0}v^{\frac{2}{3}}\}}av^{-\alpha}g(\eta,t)\der\eta-\frac{(1-\gamma)\lambda}{2}\int_{\{a\geq2c_{0}v^{\frac{2}{3}}\}}a^{\mu+1}v^{\sigma-\alpha}g(\eta,t)\der \eta\nonumber\\
&\leq -M_{1,-\alpha}(g(t))+3\delta_{1} M_{2,0}(g(t))+3M_{0,m_{1}}(g(t))+6c_{0}M_{0,\frac{2}{3}-\alpha}(g(t)),\label{first two equations self similarity}
\end{align}
where we have used that $\alpha<1$ (see (\ref{alpha non neg})) and thus $2c_{0}(\alpha+\frac{4}{3})\leq 6c_{0}$ and $\alpha+\frac{4}{3}\leq 3$, and
\begin{align}
\partial_{t}M_{1,\beta}(g(t)) &\leq -M_{1,\beta}(g(t))+2M_{1,\beta}(g(t))-\frac{(1-\gamma)\lambda}{2}\int_{\{a\geq2c_{0}v^{\frac{2}{3}}\}}a^{\mu+1}v^{\sigma+\beta}g(\eta,t)\der \eta \nonumber \\
&+K_{0}(1-\gamma)c_{0,\beta-\alpha}\int_{(0,\infty)^{2}}g(\eta,t)av^{\beta}\der\eta+K_{0}(1-\gamma)c_{0,2\beta}\int_{(0,\infty)^{2}}g(\eta,t)av^{-\alpha}\der\eta\nonumber\\
&\leq -M_{1,\beta}(g(t))-\frac{(1-\gamma)\lambda}{2}\int_{\{a\geq2c_{0}v^{\frac{2}{3}}\}}a^{\mu+1}v^{\sigma+\beta}g(\eta,t)\der \eta+K_{0}c_{0,2\beta}M_{1,-\alpha}(g(t))\nonumber\\
&+(K_{0}c_{0,\beta-\alpha}+2)M_{1,\beta}(g(t))\nonumber\\
&\leq -M_{1,\beta}(g(t))+K_{0}c_{0,2\beta}M_{1,-\alpha}(g(t))+2C(c_{0,\beta-\alpha})c_{0}M_{0,\frac{2}{3}+\beta}(g(t))+ C(c_{0,\beta-\alpha})\delta_{1} M_{2,0}(g(t))\nonumber\\
&+  C(c_{0,\beta-\alpha}) M_{0,m_{2}}(g(t)),\label{equations part 2}
\end{align}
where $C(c_{0,\beta-\alpha})=K_{0}c_{0,\beta-\alpha}+2.$ As we have that $M_{0,0},$ $M_{0,\frac{2}{3}-\alpha},$ $M_{0,m_{1}},$ $M_{0,m_{2}},$ $M_{0,\frac{2}{3}+\beta}$ are uniformly bounded from above, (\ref{first two equations self similarity}) and (\ref{equations part 2}) become: 
\begin{align}
\partial_{t} M_{2,0}(g(t))&\leq -\frac{1}{3}M_{2,0}(g(t))+2K_{0}(1-\gamma)M_{1,-\alpha}(g(t))M_{1,\beta}(g(t))+C; \nonumber
\\
\partial_{t} M_{1,-\alpha}(g(t))&\leq 
 -M_{1,-\alpha}(g(t))+3\delta_{1} M_{2,0}(g(t))+C;\nonumber\\
\partial_{t}M_{1,\beta}(g(t))
&\leq -M_{1,\beta}(g(t))+K_{0}c_{0,2\beta}M_{1,-\alpha}(g(t))+ C(c_{0,\beta-\alpha})\delta_{1} M_{2,0}(g(t))+C,\label{equations part 3}
\end{align}
where $C>0$ is a constant depending only on $\gamma, c_{0}$ and the upper bound of some moments of the form $M_{0,d},$ with $d\in\mathbb{R}$. Using (\ref{equations part 3}) we are able to find a region that stays invariant in time. 

In order to give some insight about the proof, we consider first the case $\delta_{1}=0$, since the result for $\delta_{1}>0$, sufficiently small, follows by a perturbative argument. (\ref{equations part 3}) becomes
\begin{align}
\partial_{t} M_{2,0}(g(t))&\leq -\frac{1}{3}M_{2,0}(g(t))+2K_{0}(1-\gamma)M_{1,-\alpha}(g(t))M_{1,\beta}(g(t))+C; \nonumber
\\
\partial_{t} M_{1,-\alpha}(g(t))&\leq 
 -M_{1,-\alpha}(g(t))+C;\nonumber\\
\partial_{t}M_{1,\beta}(g(t))
&\leq -M_{1,\beta}(g(t))+K_{0}c_{0,2\beta}M_{1,-\alpha}(g(t))+C. \label{equations with delta zero}
\end{align}

Notice that from (\ref{equations with delta zero}) we deduce that the set $\{M_{1,-\alpha}(g)\leq C,$ $M_{1,\beta}(g)\leq C(K_{0}c_{0,2\beta}+1),$ $M_{2,0}(g)\leq 3[C+2(1-\gamma)K_{0}C^{2}(K_{0}c_{0,2\beta}+1)]\}$ stays invariant in time.

Let $\delta_{1}>0$. We prove that the region defined by $\{M_{1,-\alpha}(g)\leq 2C,$ $M_{1,\beta}(g)\leq 2C(K_{0}c_{0,2\beta}+1),$ $M_{2,0}(g)\leq 3[C+32(1-\gamma)K_{0}C^{2}(K_{0}c_{0,2\beta}+1)]\}$ is invariant in time.

Assume that $M_{1,-\alpha}(g(0))\leq 2C,$ $M_{1,\beta}(g(0))\leq 2C(K_{0}c_{0,2\beta}+1)$ and $M_{2,0}(g(0))\leq 3[C+32(1-\gamma)K_{0}C^{2}(K_{0}c_{0,2\beta}+1)]$. By continuity in time, for all $s$ sufficiently small, we obtain that  $M_{1,-\alpha}(g(s))\leq 4C,$ $M_{1,\beta}(g(s))\leq 4C(K_{0}c_{0,2\beta}+1)$. Plugging this in (\ref{equations part 3}), we obtain that
\begin{align}
\partial_{s} M_{2,0}(g(s))&\leq -\frac{1}{3}M_{2,0}(g(s))+32K_{0}(1-\gamma)C^{2}(K_{0}c_{0,2\beta}+1)+C, \label{bound for m2 a}
\end{align}
for all times $s$ small enough. Thus, the region $M_{2,0}(g(s))\leq 3[C+32(1-\gamma)K_{0}C^{2}(K_{0}c_{0,2\beta}+1)]$ is invariant for small enough times. We now make use again of (\ref{equations part 3}) and the newly obtained bound for $M_{2,0}$, to deduce that
\begin{align}
\partial_{s} M_{1,-\alpha}(g(s))&\leq 
 -M_{1,-\alpha}(g(s))+9\delta_{1} [C+32(1-\gamma)K_{0}C^{2}(K_{0}c_{0,2\beta}+1)]+C;\nonumber\\
\partial_{s}M_{1,\beta}(g(s))
&\leq -M_{1,\beta}(g(s))+K_{0}c_{0,2\beta}M_{1,-\alpha}(g(s))\nonumber\\
&+ 3C(c_{0,\beta-\alpha})\delta_{1} [C+32(1-\gamma)K_{0}C^{2}(K_{0}c_{0,2\beta}+1)]+C,\label{the other two moments}
\end{align}
for sufficiently small times. Choosing $\delta_{1}$ sufficiently small so that $9\delta_{1} [C+32(1-\gamma)K_{0}C^{2}(K_{0}c_{0,2\beta}+1)]\leq C$ and $3C(c_{0,\beta-\alpha})\delta_{1} [C+32(1-\gamma)K_{0}C^{2}(K_{0}c_{0,2\beta}+1)]\leq C$, we obtain that the regions $M_{1,-\alpha}(g(s))\leq 2C$ and $M_{1,\beta}(g(s))\leq 2C(K_{0}c_{0,2\beta}+1)$ are invariant in time for all $s$ sufficiently small. We are now able to iterate the argument, extending it to all times. These bounds are independent of $\epsilon,R$.
\end{proof}
We now have all the necessary parts to conclude the proof of Theorem \ref{self sim mu neg}. 
\begin{proof}[Proof of Theorem \ref{self sim mu neg}]
We can uniformly bound $M_{1,0}(g_{\epsilon,R}(t))\leq M_{1,-\alpha}(g_{\epsilon,R}(t))+M_{1,\beta}(g_{\epsilon,R}(t))$, since $M_{1,-\alpha}(g_{\epsilon,R}(t))$ and $M_{1,\beta}(g_{\epsilon,R}(t))$ are bounded. This is enough to enable us to follow the steps of the proof of Theorem \ref{themostimportanttheorem} to conclude that there exists a self-similar profile for the two-dimensional coagulation equation in this case. This finishes the proof of Theorem \ref{self sim mu neg}.
\end{proof}

\appendix
\section{Formal rescaling properties}\label{formal rescaling properties}
It is worthwhile to mention that the equation (\ref{strongfusioneq}) has some useful rescaling properties. More explicitly, if $f$ satisfies (\ref{strongfusioneq}) and $\int_{(0,\infty)^{2}}vf(a,v,t)\der v \der a=v_{0}$, then we can define a set of functions $\tilde{f}_{k}$ 
\begin{align}\label{resclaing volume}
f(a,v,t)=\frac{v_{0}}{k^{\frac{8}{3}}}\tilde{f}_{k}(\frac{a}{k^{\frac{2}{3}}},\frac{v}{k},k^{\gamma-1}v_{0}t).
\end{align}
The rescaling (\ref{resclaing volume}) can be used to remove the dependence on $v_{0}$, the problem being reduced to one where the total volume of particles is equal to $1$. We have that $\tilde{f}_{k}$ solves (\ref{strongfusioneq}) with  the fusion kernel $r$ replaced by $\frac{r}{v_{0}}$ and 
\begin{align*}
    \int_{(0,\infty)^{2}}v\tilde{f}_{k}(a,v,t)\der v \der a=1.
\end{align*}
Notice that, if $f$ satisfies (\ref{strongfusioneq}) and $g$ is defined as 
\begin{align}\label{from normal to similar}
f(a,v,t)=\frac{1}{t^{\frac{8}{3}\xi}}g(\frac{a}{t^{\frac{2}{3}\xi}}, \frac{v}{t^{\xi}}, \xi \log(t)),
\end{align} 
up to a translation in time, then $g$ satisfies (\ref{strong_self_sim}). Combining (\ref{resclaing volume}) and (\ref{from normal to similar}), we obtain that $\tilde{g}_{k}$, which can be expressed in terms of $g$ in the following manner
\begin{align}\label{scaling g g}
    \tilde{g}_{k}(\tilde{A},\tilde{V},\tilde{\tau})=v_{0}^{\frac{8}{3}\xi-1}g(v_{0}^{\frac{2}{3}\xi}\tilde{A},v_{0}^{\xi}\tilde{V},\tau),
\end{align}
where $\tau=\xi\log(t)$ and $\tilde{\tau}=\xi\log(k^{\gamma-1}v_{0}t)$, satisfies (\ref{strong_self_sim}) with total volume equal to $1$. In particular, if we choose $k=v_{0}^{\xi}$, we have $\tilde{\tau}=\tau$.

We obtain the following rescaling property for $\tilde{g}_{k}$ in terms of $g$:
\begin{align}\label{connection moments}
  M_{y_{1},y_{2}}(\tilde{g}_{k}(\tilde{\tau}))=v_{0}^{\frac{\gamma-\frac{2}{3}y_{1}-y_{2}}{1-\gamma}}M_{y_{1},y_{2}}(g(\tau)),
\end{align}
for $y_{1},y_{2}\in\mathbb{R}$, which we make use of in Proposition \ref{prop ramification} and Theorem \ref{remarkramification}. Notice that in the relation between the moments of $g$ and the moments of $\tilde{g}_{k}$ the dependence on the variable $k$ vanishes and the reason we choose to fix $k=v_{0}^{\xi}$ is in order to avoid the time translation.

The existence of a function $g_{\textup{in}}$ which satisfies (\ref{g in 2}), (\ref{initial area large}) and has volume $v_{0}$ is equivalent by means of the scaling (\ref{scaling g g}) to finding a function $\tilde{g}_{\textup{in}}(\tilde{A},\tilde{V})$ such that all these inequalities hold with $v_{0}=1$. The existence of such a function can be easily seen if we choose its support in an appropriate region contained in $\tilde{V}\in[\frac{1}{2},1]$ and with $\tilde{A}$ sufficiently large.

Thus, we can assume without loss of generality that the total volume of the particles is equal to $1$, keeping in mind that the constants $R_{0}$ and $R_{1}$ in (\ref{fusion_form}) change by a factor $v_{0}$. 

\section{Proof of some results used to obtain the existence of solutions for the truncated time-dependent problem}\label{appendix b}
\begin{proof}[Proof of Proposition \ref{contractivemap}]
The following remark is used in order to prove the stated properties:
\begin{align}
  \Phi(V,V',t)\textup{e}^{h_{\epsilon}(l_{2}(V,V',t),t)}=\textup{e}^{h_{\epsilon}(V,t)+h_{\epsilon}(V',t)}.\label{ceafacuteugenia}
\end{align}
We make use of (\ref{eq1}), (\ref{eq2}), the kernel bound and the inequality $|e^{-x_{1}}-e^{-x_{2}}|\leq |x_{1}-x_{2}|,$ for $x_{1},x_{2}\geq 0$. 

We first prove that the map stays in the ball of radius $2M$. In order to bound $J_{1}[G]$ in (\ref{definitionjf}), we use that $x_{t,V}(A)\leq A,$ for every $t\in[0,\tau]$ and the assumption $\tau\leq \ln{2}$. For the term $J_{2}[G]$, we notice that $1+x_{t,V}(A)+x_{t,V'}(A')\leq (1+x_{t,V}(A))(1+x_{t,V'}(A'))$. Thus, for every $\varphi\in\compactfun$ with $||\varphi||_{\infty}\leq 1$, we have:
\begin{align*}
    \int_{(0,\infty)^{2}}\textup{e}^{h_{\epsilon}(V,t)}(1+x_{t,V}(A))J[G](A,V,t)\varphi(A,V)\der V\der A&\leq 2||\varphi||_{\infty}\int_{(0,\infty)^{2}}(1+a)g_{\textup{in},R}(a,v) \der v \der a \\
    &+\frac{1-\gamma}{2}||K_{R}||_{\infty}||G||^{2}||\varphi||_{\infty}\tau\\
    &\leq M+(1-\gamma)2M^{2}||K_{R}||_{\infty}\tau\leq M+1\leq 2M.
\end{align*}
We now prove that the map $J$, which was defined in (\ref{definitionjf}), is contractive:
\begin{align*}
       \bigg|\int_{(0,\infty)^{2}}\textup{e}^{h_{\epsilon}(V,t)}(1&+x_{t,V}(A))[J[F](A,V,t)-J[G](A,V,t)]\varphi(A,V)\der V\der A\bigg|\\
       &\leq 2\int_{(0,\infty)^{2}}(1+a)g_{\textup{in},R}(a,v) \der v \der a||K_{R}||_{\infty} ||\varphi||_{\infty}||F-G||(\textup{e}^{\tau}-1)\\
       &+\frac{1-\gamma}{2}||K_{R}||_{\infty}(||F||+||G||)||\varphi||_{\infty}||F-G||\tau+\frac{1-\gamma}{2}||K_{R}||_{\infty}^{2}||G||^{2}||\varphi||_{\infty}||F-G||\tau(\textup{e}^{\tau}-1)\\
    &\leq (M||K_{R}||_{\infty}+2M(1-\gamma)||K_{R}||_{\infty}+||K_{R}||_{\infty})(\textup{e}^{\tau}-1)||F-G||<\frac{1}{2} ||F-G||.
\end{align*}
\end{proof}
\begin{proof}[Proof of Proposition \ref{extensionsolutionprop}]
By testing with $\textup{e}^{h_{\epsilon}(V,t)}x_{t,v},$ for $t\in[0,\tau],$ and using (\ref{ceafacuteugenia}), we obtain that:
\begin{align*}
   \int_{\mathbb{R}_{>0}^{2}}G_{\epsilon, R,\delta}(\eta,t)\textup{e}^{h_{\epsilon}(V,t)}x_{t,V}(A)\der \eta&\leq  2\int_{\mathbb{R}_{>0}^{2}}g_{\textup{in},R}(\eta)A\der \eta+\frac{1-\gamma}{2}\int_{0}^{t}\int_{\mathbb{R}_{>0}^{2}}\int_{\mathbb{R}_{>0}^{2}}K_{R}(\phi_{s}(\eta),\phi_{s}(\eta'))G_{\epsilon, R,\delta}(\der\eta,s) \\
   &G_{\epsilon, R,\delta}(\der\eta',s)\textup{e}^{h_{\epsilon}(V,s)+h_{\epsilon}(V',s)}[(x_{s,V}(A)+x_{s,V'}(A')) -x_{s,V}(A)-x_{s,V'}(A')]\der s.
\end{align*}
Thus:
\begin{align*}
   \int_{\mathbb{R}_{>0}^{2}}G_{\epsilon, R,\delta}(A,V,t)\textup{e}^{h_{\epsilon}(V,t)}x_{t,V}(A)\der V\der A&\leq  2\int_{\mathbb{R}_{>0}^{2}}g_{\textup{in},R}(A,V)A\der V\der A,
\end{align*}
which is bounded by assumption. Similarly, testing with $\textup{e}^{h_{\epsilon}(V,t)}$, we get:
\begin{align*}
   \int_{\mathbb{R}_{>0}^{2}}G_{\epsilon, R,\delta}(A,V,t)\textup{e}^{h_{\epsilon}(V,t)}\der V\der A&\leq  2\int_{\mathbb{R}_{>0}^{2}}g_{\textup{in},R}(A,V)\der V\der A.
\end{align*}
The condition imposed on $g_{\textup{in},R}$ in order to prove Proposition \ref{contractivemap} was $2\int_{(0,\infty)^{2}}(1+a)g_{\textup{in},R}(a,v)\der v \der a\leq M$. Thus, we can replace $g_{\textup{in},R}$ by $G_{\epsilon,R,\delta}(\cdot,\cdot,\tau)$ and then iterate the argument to extend the solution to all times.
\end{proof}
\begin{proof}[Proof of Proposition \ref{continuous time semigroup}]
 Let $s,t\in[0,T]$. We denote by $g(\cdot,t):= S(t)g(\cdot)$. Let $n\in\mathbb{N}$ be sufficiently large. Let $\varphi\in\textup{C}_{\textup{c}}(\mathbb{R}_{>0}^{2})$ and $\varphi_{n}\in\textup{C}^{1}_{\textup{c}}(\mathbb{R}_{>0}^{2})$ be such that $\sup_{(a,v)\in \mathbb{R}_{>0}^{2}}|\varphi_{n}(a,v)-\varphi(a,v)|\leq\frac{1}{n}$. Then
\begin{align}\label{initial continuous semigroup time}
    \bigg|\int_{(0,\infty)^{2}}&[S(t)g(\eta)-S(s)g(\eta)]\varphi(\eta)\der \eta\bigg|\leq\nonumber\\ &\leq\bigg|\int_{(0,\infty)^{2}}[g(\eta,t)-g(\eta,s)]\varphi_{n}(\eta)\der \eta\bigg|+ \bigg|\int_{(0,\infty)^{2}}[g(\eta,t)-g(\eta,s)][\varphi_{n}(\eta)-\varphi(\eta)]\der \eta\bigg|\nonumber\\
    &\leq \bigg|\int_{(0,\infty)^{2}}[g(\eta,t)-g(\eta,s)]\varphi_{n}(\eta)\der \eta\bigg|+ \frac{2}{n}\sup_{r\in[0,T]}\int_{(0,\infty)^{2}}g(\eta,r)\der \eta
    \nonumber\\
    &\leq \bigg|\int_{(0,\infty)^{2}}[g(\eta,t)-g(\eta,s)]\varphi_{n}(\eta)\der \eta\bigg|+ \frac{C}{n},   
\end{align}
where we used that we can bound $\sup_{r\in[0,T]}\int_{(0,\infty)^{2}}g(\eta,r)\der \eta$ from above. On the other hand,
\begin{align*}
\bigg|\int_{(0,\infty)^{2}}[g(\eta,t)-g(\eta,s)]\varphi_{n}(\eta)\der \eta\bigg|&\leq
\int_{s}^{t}\int_{(0,\infty)^{2}}|g(\eta,r)\varphi_{n}(\eta)|\der\eta\der r+\frac{2}{3}\int_{s}^{t}\int_{(0,\infty)^{2}}|g(\eta,r) a\partial_{a}\varphi_{n}(\eta)|\der\eta\der r\\
&+\int_{s}^{t}\int_{(0,\infty)^{2}}|g(\eta,r)v\partial_{v}\varphi_{n}(\eta)|\der \eta\der r+ (1-\gamma)\int_{s}^{t}|\langle \mathbb{K}_{\epsilon,R}[g](r),\varphi_{n}\rangle|\der r\\
&+(1-\gamma)\int_{s}^{t}\int_{(0,\infty)^{2}}r_{\delta}(\eta)|c_{0}v^{\frac{2}{3}}-a|g(\eta,r)|\partial_{a}\varphi_{n}(\eta)|\der \eta \der r.
\end{align*}
As $K_{\epsilon,R}$ is bounded and $g_{\epsilon,R,\delta}$ has compact support in the $v$ variable, we can find a constant $C(\epsilon,R,\delta)$, which depends on $\epsilon, R, \delta$ such that
\begin{align*}
\bigg|\int_{(0,\infty)^{2}}[g(\eta,t)-g(\eta,s)]\varphi_{n}(\eta)\der \eta\bigg|&\leq
\frac{2}{3}\int_{s}^{t}\int_{(0,\infty)^{2}}|g(\eta,r) a\partial_{a}\varphi_{n}(\eta)|\der\eta\der r+C(\epsilon,R,\delta)|t-s|\\
&+\frac{1}{\delta}C(\epsilon,R,\delta)\int_{s}^{t}\int_{(0,\infty)^{2}}|1+a|g(\eta,r)|\partial_{a}\varphi_{n}(\eta)|\der \eta \der r.
\end{align*}
As $\sup_{r\in[0,T]}\int_{(0,\infty)^{2}}(1+a)g_{\epsilon,R,\delta}(\eta,r)\der \eta$ is bounded from above, we find that there exists $C(\epsilon,R,\delta)$ depending on the written parameters, such that
\begin{align}\label{continuous time for compact support}
\bigg|\int_{(0,\infty)^{2}}&[g(\eta,t)-g(\eta,s)]\varphi_{n}(\eta)\der \eta\bigg|\leq C(\epsilon,R,\delta)|t-s|.
\end{align}
Combining (\ref{continuous time for compact support}) with (\ref{initial continuous semigroup time}), we obtain  the continuity in time of the map $t\rightarrow \int_{(0,\infty)^{2}}g(\eta,t)\varphi(\eta)\der \eta$, if $\varphi\in\textup{C}_{\textup{c}}(\mathbb{R}_{>0}^{2})$. We now extend the argument to all functions $\varphi\in\textup{C}_{0}(\mathbb{R}_{>0}^{2})$. Let $\varphi\in\textup{C}_{0}(\mathbb{R}_{>0}^{2}),$ with $||\varphi||_{\infty}\leq 1.$ Due to the support of $g\in U_{\epsilon,R},$ it is enough to cut the function $\varphi$ for large values of $a$ in order to make it compactly supported. Let $\chi_{M}(a)$ be as in (\ref{cut near infinity}). Then
\begin{align*}
 \bigg|\int_{(0,\infty)^{2}}&[g(\eta,t)-g(\eta,s)]\varphi(\eta)\der \eta\bigg|\leq\\ &\leq\bigg|\int_{(0,\infty)^{2}}[g(\eta,t)-g(\eta,s)]\varphi(\eta)\chi_{M}(a)\der \eta\bigg|+2\int_{(M,\infty)\times[\epsilon,2R]}\big|[g(\eta,t)-g(\eta,s)]\varphi(\eta)\big|\der \eta\\
 &\leq C|t-s|+2M^{-1}\int_{(M,\infty)\times[\epsilon,2R]}a[g(\eta,t)+g(\eta,s)]\der \eta\\
 &\leq C|t-s|+c(t)M^{-1}.
\end{align*}
We conclude the argument by taking $M$ sufficiently large.
\end{proof}
\begin{proof}[Proof of Proposition \ref{extensionalltime_generalcase}]
Let $n\in\mathbb{N}$. By Lemma \ref{inequalityforiterationargument}, we have that:
\begin{align*}
   \int_{\mathbb{R}_{>0}^{2}}G_{\epsilon, R,\delta}&(\eta,t)\textup{e}^{h_{\epsilon}(V,t)}x_{t,V}(A)^{n}\der \eta\leq \\
   \leq &2\int_{\mathbb{R}_{>0}^{2}}g_{\textup{in},R}(\eta)A^{n}\der \eta+\frac{1-\gamma}{2}\int_{0}^{t}\int_{\mathbb{R}_{>0}^{2}}\int_{\mathbb{R}_{>0}^{2}}K_{R}(\phi_{s}(A,V),\phi_{s}(A',V'))G_{\epsilon, R,\delta}(\eta',s)G_{\epsilon, R,\delta}(\eta,s) \\
   &\textup{e}^{h_{\epsilon}(V,s)+h_{\epsilon}(V',s)}[(x_{s,V}(A)+x_{s,V'}(A'))^{n} -x_{s,V}(A)^{n}-x_{s,V'}(A')^{n}]\der\eta' \der \eta \der s\\
    \leq &2\int_{\mathbb{R}_{>0}^{2}}g_{\textup{in},R}(\eta)A^{n}\der \eta+(1-\gamma)||K_{R}||_{\infty}\sum_{k=1}^{k_{n}}\binom{n}{k}\int_{0}^{t}\int_{\mathbb{R}_{>0}^{2}}\textup{e}^{h_{\epsilon}(V,s)}x_{s,V}^{k}(A)G_{\epsilon, R,\delta}(\eta,s)\der \eta \\
    &\int_{\mathbb{R}_{>0}^{2}}\textup{e}^{h_{\epsilon}(V',s)}x_{s,V'}^{n-k}(A')G_{\epsilon, R,\delta}(\eta',s)\der \eta'\der s,
\end{align*}
where $k_{n}$ is taken as in Lemma \ref{inequalityforiterationargument}. The above computations show that, in order to find an upper bound for $\int_{\mathbb{R}_{>0}^{2}}G_{\epsilon, R,\delta}(\eta,t)\textup{e}^{h_{\epsilon}(V,t)}x_{t,V}(A)^{n}\der \eta$, it is enough to estimate $\int_{\mathbb{R}_{>0}^{2}}G_{\epsilon, R,\delta}(\eta,t)\textup{e}^{h_{\epsilon}(V,t)}x_{t,V}(A)^{k}\der \eta$, where $k\in[1,n-1]$. As such, we can use an iteration argument for the exponents of $x_{t,V}(A)$. We then use that $$\int_{(0,\infty)^{2}}G_{\epsilon,R,\delta}(A,V,t,)e^{h_{\epsilon}(V,t)}x_{t,V}(A)\der V\der A\leq 2\int_{(0,\infty)^{2}}ag_{\textup{in},R}(a,v)\der v \der a$$ as in the proof of Proposition \ref{extensionsolutionprop}, which can be found in this appendix. In this manner, we derive suitable moment estimates which allow to extend the obtained solution to all times by iterating the argument.
\end{proof}
\section{Some technical results used to prove the existence of self-similar profiles}\label{appendix c}
\begin{rmk} 
 In order to simplify the notation we will write $g$ and $g_{\textup{in}}$ instead of $g_{\epsilon,R,\delta}$ and $g_{\textup{in},R}$, respectively, in the following computations. It is relevant to take into account that $g_{\epsilon,R,\delta}:=g$ is supported in the region $(a,v)\in [c_{0}\epsilon^{\frac{2}{3}},\infty)\times[\epsilon,2R)$.
\end{rmk}
\begin{proof}[Proof of Proposition \ref{firstmoments}]
We obtain (\ref{m1}) by testing (\ref{regularizedformequation}) with $\varphi(a,v)=v$. In order to derive (\ref{mlambda}), we test (\ref{regularizedformequation}) with $\varphi(a,v)=v^{\gamma}$.

By Remark \ref{kernel_lower_bound_welldef} and since $g$ is supported in the region $(a,v)\in [c_{0}\epsilon^{\frac{2}{3}},\infty)\times[\epsilon,2R)$, we can ignore the dependence of $K_{\epsilon,R}$ on $R$:
\begin{align*}
&\frac{(1-\gamma)}{2}\int_{(0,\infty)^{2}}\int_{(0,\infty)^{2}}K_{\epsilon,R}(a,v,a',v')\xi_{R}(v+v')g(a,v,t)g(a',v',t)[(v+v')^{\gamma}-v^{\gamma}-v'^{\gamma}]\der v' \der a' \der v \der a\\
\leq &\frac{K_{1}(1-\gamma)}{2}\int_{(0,\infty)^{2}}\int_{(0,\infty)^{2}}(v^{-\alpha}v'^{\beta}+v^{\beta}v'^{-\alpha})\xi_{R}(v+v')g(a,v,t)g(a',v',t)[(v+v')^{\gamma}-v^{\gamma}-v'^{\gamma}]\der v' \der a' \der v \der a\\
\leq&\frac{K_{1}(1-\gamma)}{2}\int_{\{v<\frac{R}{2}\}}\int_{\{v'<\frac{R}{2}\}}(v^{-\alpha}v'^{\beta}+v^{\beta}v'^{-\alpha})g(a,v,t)g(a',v',t)[(v+v')^{\gamma}-v^{\gamma}-v'^{\gamma}]\der v' \der a' \der v \der a.
\end{align*}

Suppose $v\geq v'$ and take $z:=\frac{v'}{v}\in(0,1]$. We have:
\begin{align}
(v^{-\alpha}v'^{\beta}+v^{\beta}v'^{-\alpha})((v+v')^{\gamma}-v^{\gamma}-v'^{\gamma})&\leq v^{\gamma}(z^{\beta}+z^{-\alpha})v^{\gamma}((1+z)^{\gamma}-1-z^{\gamma}) \nonumber
\\ &\leq -(1-\gamma)v^{2\gamma}z^{\gamma}=-(1-\gamma)v^{\gamma}v'^{\gamma}, \nonumber
\end{align}
since $(1+z)^{\gamma}-1-z^{\gamma}\leq 0$ and $z^{-\alpha}\geq 1$. We also used that:
$$ (1+z)^{\gamma}-1=\gamma\int_{1}^{1+z}s^{\gamma-1}\der s \leq \gamma z\leq \gamma z^{\gamma}.$$

Equation (\ref{regularizedformequation}) now becomes:
\begin{align*}
\frac{\der}{\der t}M_{0,\gamma}(g(t))&\leq (1-\gamma)\int_{(0,\infty)^{2}}v^{\gamma}\Theta_{\epsilon}(v)g(a,v,t)\der v \der a-\frac{(1-\gamma)^{2}K_{1}}{2}\bigg(\int_{\{v<\frac{R}{2}\}}v^{\gamma}g(a,v,t)\der v \der a\bigg)^{2}\\
&\leq(1-\gamma)M_{0,\gamma}(g(t))-\frac{(1-\gamma)^{2}K_{1}}{4}M_{0,\gamma}(g(t))^{2}+\frac{(1-\gamma)^{2}K_{1}}{2}\bigg(\int_{\{v>\frac{R}{2}\}}v^{\gamma}g(a,v,t)\der v \der a\bigg)^{2}.
\end{align*}
We can control the region where $v>\frac{R}{2}$ using (\ref{m1}), namely:
\begin{align*}
   \int_{\{v>\frac{R}{2}\}}v^{\gamma}g(a,v,t)\der v \der a\leq \bigg(\frac{R}{2}\bigg)^{\gamma-1}\int_{(0,\infty)^{2}}vg(a,v,t)\der v \der a\leq \bigg(\frac{R}{2}\bigg)^{\gamma-1}M_{0,1}(g_{\textup{in}}),
\end{align*}
thus finding a constant $C>0$, independent of $\epsilon\in(0,1), R>1$ and $\delta\in(0,1)$, for which
\begin{align*}
\frac{\der}{\der t}M_{0,\gamma}(g(t))\leq (1-\gamma)M_{0,\gamma}(g(t))-\frac{(1-\gamma)^{2}K_{1}}{4}M_{0,\gamma}(g(t))^{2}+C.
\end{align*}
We then conclude that there exists a constant $C_{0,\gamma}$, independent of $\epsilon,R$ and $\delta,$ for which the region $M_{0,\gamma}(g(t))\leq C_{0,\gamma}$ is invariant in time.
\end{proof}
\begin{proof}[Proof of Proposition \ref{prop m moment}]
We want to bound $$K(a,v,a',v')[(v+v')^{m}-v^{m}-v'^{m}].$$
Assume $v\geq v'$.  Denote $z:=\frac{v'}{v}\in(0,1]$ and observe that:
\begin{align}
K(a,v,a',v')[(v+v')^{m}-v^{m}-v'^{m}]\leq & K(a,v,a',v') v^{m}[(1+z)^{m}-1-z^{m}]\leq 2C_{m}K(a,v,a',v')  v^{m}z \nonumber \\
\leq & 2C_{m}K_{0}v^{\gamma}(z^{\beta}+z^{-\alpha}) v^{m}z \leq 4C_{m}K_{0}v^{\gamma+m}z^{-\alpha+1}\nonumber \\
\leq & 4C_{m}K_{0} (v^{\beta+m-1}v'^{-\alpha+1}+v'^{\beta+m-1}v^{-\alpha+1}), \label{onegreater}
\end{align}
 where we used that $$(1+z)^{m}-1=m\int_{1}^{1+z}s^{m-1}\der s\leq m(1+z)^{m-1}z\leq C_{m}(1+z^{m-1})z,$$ since $m>1$ and $z\in(0,1]$. By symmetry, $(\ref{onegreater})$ holds for all $(v,v')\in(0,\infty)^{2}$.

Equation (\ref{regularizedformequation}) becomes:
\begin{align}\label{choosingm}
\frac{\der}{\der t}M_{0,m}(g(t))&\leq -(m-1)M_{0,m}(g(t))+(m-1)\int_{\{v\leq2\epsilon\}}v^{m}g(\eta,t)\der \eta+2(1-\gamma)C_{m}K_{0}M_{0,\beta+m-1}(g)M_{0,-\alpha+1}(g)\nonumber\\
&\leq -(m-1)M_{0,m}(g(t))+(m-1)2^{m-1}M_{0,1}(g_{\textup{in}})+C M_{0,\beta+m-1}(g(t))M_{0,-\alpha+1}(g(t)).
\end{align}

By (\ref{m1}) and (\ref{mlambda}), we obtain that $M_{0,-\alpha+1}(g(t))$ is uniformly bounded as $-\alpha+1\in(\gamma,1]$.
\

Additionally, we have that $m+\beta-1\in(\gamma,m)$. Let $C_{0,\gamma}$ be the constant found in Proposition \ref{firstmoments}, then there exists $\theta\in(0,1)$ such that:
\begin{align}\label{choosingr}
M_{0,m+\beta-1}(g(t))\leq M_{0,\gamma}(g(t))^{1-\theta}M_{0,m}(g(t))^{\theta}\leq C_{0,\gamma}^{1-\theta}M_{0,m}(g(t))^{\theta}.
\end{align}

Hence, combining (\ref{choosingm}) and (\ref{choosingr}), we find an invariant region in time for the moment $M_{0,m}$.
\end{proof}
\begin{proof}[Proof of Proposition \ref{m-lprop}]
Let $C_{0,\tilde{m}}$ be the constant found in (\ref{mm}), with $\tilde{m}>1$ as in Proposition \ref{lemmaweaksol}, and let $\beta<1$ as in (\ref{lower_bound_kernel}). For all $t\geq 0:$
\begin{align} \label{lower bound is needed}
1&= M_{0,1}(g(t))\leq N^{1-\beta}\int_{(0,\infty)\times(0,N)}v^{\beta}g(a,v,t)\der v \der a+N^{1-\tilde{m}}\int_{(0,\infty)\times(N,\infty)}v^{\tilde{m}}g(a,v,t)\der v\der a \nonumber \\
&\leq N^{1-\beta}M_{0,\beta}(g(t))+N^{1-\tilde{m}}\max\{M_{0,\tilde{m}}(g_{\textup{in}}),C_{0,\tilde{m}}\}. 
\end{align}
Thus, for $N>0$ sufficiently large, we obtain
\begin{align}\label{lower bound moment beta}
    M_{0,\beta}(g(t))\geq\frac{1}{2N^{1-\beta}}.
\end{align}

We analyse the term
\begin{align*}
&\int_{(0,\infty)^{2}}\int_{(0,\infty)^{2}}K_{\epsilon,R}(a,v,a',v')\xi_{R}(v+v')g(a,v,t)g(a',v',t)[(v+v')^{-l}-v^{-l}-v'^{-l}]\der v' \der a' \der v \der a.
\end{align*}
We make use of Remark \ref{kernel_lower_bound_welldef}: From the definition of $K_{\epsilon,R}$ in (\ref{truncation_kernel_selfsimilar}) and the support of $g$, we have that $K_{\epsilon,R}=K$ when $v,v'\in[\epsilon, 2R)^{2}$. Since $(v+v')^{-l}-v^{-l}-v'^{-l}\leq 0$, we can use the lower bound for $K$ in (\ref{lower_bound_kernel}). We obtain:
\begin{align*}
&\int_{(0,\infty)^{2}}\int_{(0,\infty)^{2}}K_{\epsilon,R}(a,v,a',v')\xi_{R}(v+v')g(a,v,t)g(a',v',t)[(v+v')^{-l}-v^{-l}-v'^{-l}]\der v' \der a' \der v \der a\\
\leq&K_{1}\int_{\{v<\frac{R}{2}\}}\int_{\{v'<\frac{R}{2}\}}(v^{-\alpha}v'^{\beta}+v^{\beta}v'^{-\alpha})g(a,v,t)g(a',v',t)[(v+v')^{-l}-v^{-l}-v'^{-l}]\der v' \der a' \der v \der a\\
\leq&K_{1}\int_{\{v<\frac{R}{2}\}}\int_{\{v'<\frac{R}{2}\}}(-v^{-\alpha-l}v'^{\beta}-v'^{-\alpha-l}v^{\beta})g(a,v,t)g(a',v',t)\der v' \der a' \der v \der a
\end{align*} 
and (\ref{regularizedformequation}) thus becomes:
\begin{align}\label{firstform}
    \frac{\der}{\der t} M_{0,-l}(g(t))\leq (1+l)M_{0,-l}(g(t))-(1-\gamma)K_{1}\int_{\{v<\frac{R}{2}\}}v^{-\alpha-l}g(a,v,t)\der v \der a \int_{\{v<\frac{R}{2}\}}v^{\beta}g(a,v,t)\der v \der a. 
\end{align}
By (\ref{lower bound moment beta})
\begin{align}\label{BELOWBOUND}
   \int_{\{v<\frac{R}{2}\}}v^{\beta}g(a,v,t)\der v \der a\geq \frac{1}{2N^{1-\beta}} -\bigg(\frac{R}{2}\bigg)^{\beta-1}  \geq \frac{1}{4N^{1-\beta}},
\end{align}
for all $R>1$ that are sufficiently large. 

As $-l\in(-\alpha-l,1)$, there exists $\theta\in(0,1)$ for which $M_{0,-l}(g(t))\leq$ $ M_{0,1}(g(t))^{1-\theta}M_{0,-\alpha-l}(g(t))^{\theta}$ $\leq$  $M_{0,1}(g_{\textup{in}})^{1-\theta}M_{0,-\alpha-l}(g(t))^{\theta}$. Combining this with (\ref{firstform}) and (\ref{BELOWBOUND}), there exists a constant $C>0$ such that:
$$   \frac{\der}{\der t} M_{0,-l}(g(t))\leq (1+l)M_{0,-l}(g(t))-C\bigg(\int_{\{v<\frac{R}{2}\}}v^{-l}g(a,v,t)\der v \der a\bigg)^{\frac{1}{\theta}}.  $$

We use $x^\frac{1}{\theta}\geq 2^{1-\frac{1}{\theta}}(x+y)^{\frac{1}{\theta}}-y^{\frac{1}{\theta}}$ with $x=\int_{\{v<\frac{R}{2}\}}v^{-l}g(a,v,t)\der v \der a$ and $y=\int_{\{v>\frac{R}{2}\}}v^{-l}g(a,v,t)\der v \der a$. We can bound $\int_{\{v>\frac{R}{2}\}}v^{-l}g(a,v,t)\der v \der a$ from above because $R>1$ and $M_{0,1}(g(t))$ is uniformly bounded. Thus
\begin{align*}
    \frac{\der}{\der t} M_{0,-l}(g(t))\leq (1+l)M_{0,-l}(g(t))-C_{1}M_{0,-l}(g(t))^{\frac{1}{\theta}}+2^{l+1}CM_{0,1}(g_{\textup{in}})
\end{align*}
and we conclude using the uniform bound on $M_{0,1}$ and then a comparison argument.
\end{proof}
\subsection*{Acknowledgements}
The authors would like to thank B. Niethammer for useful suggestions regarding the contents of the paper. The authors gratefully acknowledge the financial support of the collaborative
research centre The mathematics of emerging effects (CRC 1060, Project-ID 211504053) and Bonn International Graduate School of Mathematics at the Hausdorff Center for Mathematics (EXC 2047/1, Project-ID 390685813) funded through the Deutsche Forschungsgemeinschaft (DFG, German Research Foundation).

\subsubsection*{Statements and Declarations}

\textbf{Conflict of interest} The authors declare that they have no conflict of interest.

\textbf{Data availability} Data sharing not applicable to this article as no datasets were generated or analysed during the current study.

\printbibliography

@article{Niethammer_2012,
   title={Self-similar solutions with fat tails for {S}moluchowski’s coagulation equation with locally bounded kernels},
   volume={318},
   number={2},
   journal={Communications in Mathematical Physics},
   publisher={Springer Science and Business Media LLC},
   author={Niethammer, B. and Velázquez, J. J. L.},
   year={2012},

   pages={505–532}
}

@article{ferreira2020stationary,
      title={Stationary non-equilibrium solutions for coagulation systems}, 
      author={M. A. Ferreira and J. Lukkarinen and A. Nota and J. J. L. Velázquez},
      year={2021},
      volume = {240},
      pages = {809–875},
journal = {Archive for Rational Mechanics and Analysis},
      
}

@book{book,
author = {Banasiak, J. and Lamb, W. and Laurençot, P.},
year = {2019},
 publisher = {Boca Raton: CRC Press},
pages = {},
title = {Analytic Methods for Coagulation-Fragmentation Models, Volume II},

}

@book{book1,
author = {Friedlander, S. K. },
year = {2000},
 publisher = {New York: Oxford University Press},
pages = {},
title = {Smoke, Dust, and Haze: Fundamentals of Aerosol Dynamics},

}

@article{KOCH1990419,
title = "The effect of particle coalescence on the surface area of a coagulating aerosol",
journal = "Journal of Colloid and Interface Science",
volume = "140",
number = "2",
pages = "419 - 427",
year = "1990",


author = "Friedlander, S.K. and Koch, W.",

}

@article{norris1999,
author = "Norris, J. R.",

journal = "Annals of Applied Probability",
volume = "9",
number = "1",
pages = "78--109",
publisher = "The Institute of Mathematical Statistics",
title = "Smoluchowski's coagulation equation: uniqueness, nonuniqueness and a hydrodynamic limit for the stochastic coalescent",


year = "1999"
}

@article{ESCOBEDO,
title = "On self-similarity and stationary problem for fragmentation and coagulation models",
journal = "Annales de l'Institut Henri Poincaré C, Analyse Non Linéaire",
volume = "22",
number = "1",
pages = "99 - 125",
year = "2005",

author = "M. Escobedo and S. Mischler and M. {Rodriguez Ricard}",

}

@article{dust,
title = {Dust and self-similarity for the {S}moluchowski coagulation equation},
journal = {Annales de l'Institut Henri Poincaré C, Analyse Non Linéaire},
volume = {23},
number = {3},
pages = {331-362},
year = {2006},
author = {M. Escobedo and S. Mischler},
}

@article{multicom1,
	title = {Localization in stationary non-equilibrium solutions for multicomponent coagulation systems},
	journal = {Communications in Mathematical Physics},
	volume = {388},
	number = {1},
	pages = {479-506},
	year = {2021},
	author = {M. A. Ferreira and J. Lukkarinen and A. Nota and J. J. L. Vel{\'{a}}zquez}
}

@article{multicom2,
      title={Non-equilibrium stationary solutions for multicomponent coagulation systems with injection}, 
     author = {Ferreira, M. A. and Lukkarinen, J. and Nota, A. and Velázquez, J. J. L.},
      year={2023},
      journal={Journal of Statistical Physics},
      volume = {190},
	number = {98},
	pages = {1-35}
}

@article{inequality,
author = {Bobylev, A. and Gamba, I. and Panferov, V.},
year = {2004},
pages = {1651-1682},
title = {Moment inequalities and high-energy tails for {B}oltzmann equations with inelastic interactions},
volume = {116},
journal = {Journal of Statistical Physics}
}

@article{STEWART,
       author = {{Stewart}, I.~W.},
        title = "{A global existence theorem for the general coagulation-fragmentation equation with unbounded kernels}",
      journal = {Mathematical Methods in the Applied Sciences},
         year = 1989,
        
       volume = {11},
       number = {5},
        pages = {627-648}
       
}

@ARTICLE{1916ZPhy...17..557S,
       author = {{Smoluchowski}, M.~V.},
        title = "{Drei Vortr\"age \"{u}ber Diffusion, Brownsche Bewegung und Koagulation von Kolloidteilchen}",
      journal = {Zeitschrift fur Physik},
         year = 1916,
       volume = {17},
        pages = {557-585}
}

@article{multicomponent3,
  author = {Ferreira, M. A. and Lukkarinen, J. and Nota, A. and Velázquez, J. J. L.},
  title = {Asymptotic localization in multicomponent mass conserving coagulation equations},
   year = {2022},
   journal={Preprint, arxiv: 2203.08076}
}

@book{iulia,
  author = {I. Cristian},
  title = {Mathematical theory for
two-dimensional coagulation
equations},
   year = {2021},
publisher = {Master's thesis, University of Bonn}
}

@article{Menon_2004,
	year = 2004,
	volume = {57},
	number = {9},
	pages = {1197--1232},
	author = {G. Menon and R. L. Pego},
	title = {Approach to self-similarity in {S}moluchowski{\textquotesingle}s coagulation equations},
	journal = {Communications on Pure and Applied Mathematics}
}

@article{articlefournier,
author = {Fournier, N. and Laurençot, P.},
year = {2005},
pages = {589-609},
title = {Existence of self-similar solutions to {S}moluchowski’s coagulation equation},
volume = {256},
journal = {Communications in Mathematical Physics}
}

@article{cristian2023fast,
      title={Fast fusion in a two-dimensional coagulation model}, 
      author={I. Cristian and J. J. L. Velázquez},
      year={2023},
        journal={Preprint, arxiv: 2303.09475}
}

@article{wattis1,
author = {Wattis, J.},
year = {2006},
pages = {1-20},
title = {An introduction to mathematical models of coagulation–fragmentation processes: A discrete deterministic mean-field approach},
volume = {222},
journal = {Physica D: Nonlinear Phenomena}
}

@article{wattis2,
author = {Wattis, J.},
year = {2006},
pages = {7283-7298},
title = {Exact solutions for
cluster-growth kinetics with evolving size and shape profiles},
volume = {39},
journal = {Journal of Physics A: Mathematical and General}
}

@article{throm,
  author = {Throm, S.},
  title = {Uniqueness of measure solutions for multi-component coagulation equations},
  publisher = {arXiv}, 
  year = {2023},
  journal={Preprint: 2303.00775}
}

@article{gajewski,
	author = {Gajewski, H.},
	journal = {Mathematische Nachrichten},
	number = {1},
	pages = {289–300},
	title = {On a first order partial differential equation with nonlocal nonlinearity},
	volume = {111},
	year = {1983}
}

@article{LIFSHITZ196135,
	author = {I. M. Lifshitz and V. V. Slyozov},
journal = {Journal of Physics and Chemistry of Solids},
	number = {1},
	pages = {35–50},
	title = {The kinetics of precipitation from supersaturated solid solutions},
	volume = {19},
	year = {1961}
}

@article{LUSHNIKOV2000651,
  title = {Nucleation burst in a coagulating system},
  author = {Lushnikov, A. A. and Kulmala, M.},
  journal = {Physical Review E},
  volume = {62},
  pages = {4932--4939},
  year = {2000},
  publisher = {American Physical Society}
}

@article{HERRMANN2009909,
	author = {Herrmann, M. and Laurençot, P. and Niethammer, B.},
	journal = {Comptes Rendus Mathematique},
	number = {15},
	pages = {909–914},
	title = {Self-similar solutions with fat tails for a coagulation equation with nonlocal drift},
	volume = {347},
	year = {2009}
}

@article{HERRMANN20092282,
	author = {M. Herrmann and B. Niethammer and J. J. L. Velázquez},
	journal = {Journal of Differential Equations},
	number = {8},
	pages = {2282–2309},
	title = {Self-similar solutions for the LSW model with encounters},
	volume = {247},
	year = {2009}
}

@article{gelation,
	author = {M. Escobedo and P. Laurençot and S. Mischler and B. Perthame},
	journal = {Journal of Differential Equations},
	number = {1},
	pages = {143–174},
	title = {Gelation and mass conservation in coagulation-fragmentation models},
	volume = {195},
	year = {2003}
}

@article{articlelau,
	author = {Laurençot, P.},
	journal = {Mathematical Models and Methods in Applied Sciences},
	pages = {731–748},
	title = {The Lifshitz-Slyozov equation with encounters},
	volume = {11},
	year = {2001},
 number={4}
}

@article{menonpegoattractor,
	risfield_0_da = {2008/04/01},
	author = {Menon, G. and Pego, R. L.},
	journal = {Journal of Nonlinear Science},
	number = {2},
	pages = {143–190},
	title = {The scaling attractor and ultimate dynamics for {S}moluchowski’s coagulation equations},
	volume = {18},
	year = {2008}
}

@article{menonpegodynamical,
	author = {Menon, G. and Pego, R. L.},
	journal = {SIAM Review},
	number = {4},
	pages = {745–768},
	title = {Dynamical scaling in {S}moluchowski’s coagulation equations: uniform convergence},
	volume = {48},
	year = {2006}
}

\end{document}